\documentclass{ipiof}
\usepackage{amsmath}
  \usepackage{paralist}
  \usepackage{graphics} 
  \usepackage{epsfig} 
\usepackage{graphicx}  \usepackage{epstopdf}
 \usepackage[colorlinks=true]{hyperref}
\hypersetup{urlcolor=blue, citecolor=red}

\def\currentvolume{X}
 \def\currentissue{X}
  \def\currentyear{200X}
   
    \def\ppages{X--XX}
     \def\DOI{10.3934/ipi.2021020}

\newtheorem{theorem}{Theorem}[section]
\newtheorem{corollary}{Corollary}

\newtheorem{lemma}[theorem]{Lemma}
\newtheorem{proposition}{Proposition}

\theoremstyle{definition}

\newtheorem{remark}{Remark}

\newcommand{\Margin}[1]{}
\newcommand{\revision}[1]{{#1}}

\usepackage[pagewise]{lineno}
\usepackage{amssymb,pictex,caption}
\usepackage{wrapfig} %
\let\oldref\ref
\newbox\figurelegendone
\newbox\figurelegendtwo
\newbox\figureone
\newbox\figuretwo
\newbox\figurethree
\newbox\figurefour
\newbox\figurefive
\newbox\textone
\newbox\texttwo
\newdimen\xfiglen \newdimen\yfiglen  \newdimen\textlen

\title[Nonlinearity coefficient identification]
      {On the identification of the nonlinearity parameter in the Westervelt equation\\ from boundary measurements}

\author[Barbara Kaltenbacher and William Rundell]{}

\subjclass{Primary: 35R30, 35K58, 35L72; Secondary:  78A46.}
 \keywords{Coefficient identification, Westervelt equation, ultrasound imaging,
nonlinearity parameter, nonlinear acoustics}

\thanks{Supported by the Austrian Science Fund {\sc fwf} under grant P30054 and the
National Science Foundation through award {\sc dms}-1620138}

\thanks{$^*$ Corresponding author: Barbara Kaltenbacher}

\begin{document}
\maketitle

\centerline{\scshape Barbara Kaltenbacher$^*$}
\medskip
{\footnotesize
 \centerline{Department of Mathematics}
 \centerline{Alpen-Adria-Universit\"at Klagenfurt}
   \centerline{9020 Klagenfurt, Austria}
}

\medskip

\centerline{\scshape William Rundell }
\medskip
{\footnotesize
 \centerline{Department of Mathematics}
\centerline{Texas A\&M University}
   \centerline{Texas 77843, USA}
}

\bigskip

 \centerline{(Communicated by Habib Ammari)}


\begin{abstract}
We consider an undetermined coefficient inverse problem for a non-\\linear partial
differential equation occurring in high intensity ultrasound propagation
as used in acoustic tomography.
In particular, we investigate the recovery of the nonlinearity coefficient
commonly labeled as $B/A$ in the literature which is part of a space
dependent coefficient $\kappa$ in the  Westervelt equation governing nonlinear
acoustics.
Corresponding to the typical measurement setup, the overposed data consists
of time trace measurements on some zero or one dimensional set $\Sigma$
representing the receiving transducer array.
After an analysis of the map from $\kappa$ to the overposed data, we show
injectivity of its linearisation and use this as motivation for several
iterative schemes to recover $\kappa$.
Numerical simulations will also be shown to illustrate the efficiency of the
methods.
\end{abstract}

\section{Introduction}
The use of ultrasound is well established in the imaging of human tissue.
High intensity ultrasound is modeled by nonlinear wave equations, in which
a certain ratio of Taylor expansion coefficients $B/A$ governs the nonlinearity.
Recently, it has been shown that this $B/A$ parameter is sensitive to
differences in the tissue properties,
thus appropriate for characterization of biological tissues, see, e.g.,
\cite{Bjorno1986, BurovGurinovichRudenkoTagunov1994, Cain1986,
IchidaSatoLinzer1983, VarrayBassetTortoliCachard2011,
ZhangChenGong2001, ZhangChenYe1996}.
Therefore, when viewed as a spatially varying coefficient, it can be used for
medical imaging purposes, known as  acoustic nonlinearity parameter tomography.
This parameter appears in the {\sc pde}s describing high intensity ultrasound
propagation, thus the related imaging problem becomes a coefficient
identification for these {\sc pde}s.

We will therefore give a brief introduction into the relevant
models and highlight the one  of our focus:
namely the Westervelt equation.
Then we will specify the possible physical measurements leading
to overposed data and state the resulting inverse problem.

\subsubsection*{The model}
\revision{For a brief derivation of the fundamental acoustic equations, we refer to, e.g., the review \cite{reviewNonlinearAcoustics}.
More details can be found, for example, in \cite{Crighton79, HamiltonBlackstock98}.
\\
\indent The main physical quantities involved in the description of
sound propagation are
\begin{itemize}
\item the acoustic particle velocity $\vec{v}$;
\item the acoustic pressure $p$;
\item the mass density $\varrho$;
\end{itemize}
that can be decomposed into their constant mean and a fluctuating part
\[
\vec{v}=\vec{v}_0+\vec{v}_\sim\,, \quad p=p_0+p_\sim\,, \quad \varrho=\varrho_0+\varrho_\sim,
\]
where $\vec{v}_0=0$ in the absence of a flow.
\\
\indent These quantities are interrelated by balance equations of momentum and mass, as well as the state equation relating the acoustic pressure and density fluctuations $p_\sim$ and $\varrho_\sim$. The latter contains a so-called parameter of nonlinearity $B/A$.
Combining these balance laws, inserting the state equation and and dropping certain terms according to a certain hierarchy, which in nonlinear acoustics is known as Blackstock's scheme \cite{Lighthill56,Blackstock63} one arrives at second order wave equations in terms of the space- and time dependent fluctuating quantities.
\\
\indent This first of all yields Kuznetsov's equation \cite{LesserSeebass68,Kuznetsov71}
\begin{equation}\label{Kuznetsov}
{p_\sim}_{tt}- c^2\triangle p_\sim - b \triangle {p_\sim}_t =
\left(\frac{1}{\varrho_0 c^2}\frac{B}{2A} p_\sim^2 + \varrho_0 |\vec{v}|^2\right)_{tt}
\end{equation}
where $b$ is the diffusivity of sound,
and we have related the velocity to the pressure via the linearised force balance
\begin{equation}\label{eq:forcebal}
\varrho_0\vec{v}_t=-\nabla p_\sim\,.
\end{equation}
}
\Margin{rep C, 2.}

If we ignore local nonlinear effects modeled by the quadratic velocity term, thus approximating
$\rho_0|\vec{v}|^2_{tt}\approx \frac{1}{\rho_0^2c^2}p^2_{tt}$,
we arrive at the Westervelt equation
\begin{equation}\label{Westervelt}
{p_\sim}_{tt}- c^2\triangle p_\sim - b \triangle {p_\sim}_t =
\frac{\beta_a}{\varrho_0 c^2} {p_\sim^2}_{tt}
\end{equation}
with $\beta_a = 1 + B/(2A)$, cf.,
\cite{Westervelt63}.
Under the already made assumption $\nabla\times\vec{v}=0$ on a simply connected domain there exists an acoustic velocity potential $\psi$ with $\vec{v}=-\nabla\psi$, whose constant part by \eqref{eq:forcebal} can be chosen such that
\begin{equation}\label{prhopsitp}
\varrho_0\psi_t= p\,.
\end{equation}
Hence both equations \eqref{Kuznetsov} and \eqref{Westervelt} can also
be written in terms of the  acoustic velocity potential $\psi$
\begin{equation}\label{WestKuz_psi}
\psi_{tt}-c^2\triangle\psi-b\triangle\psi_t=\frac{1}{c^2}\Bigl(\beta_a(\psi_t)^2
+s_{\mbox{\footnotesize WK}} \left[c^2|\nabla\psi|^2-(\psi_t)^2\right]\Bigr)_t
\end{equation}
with $s_{\mbox{\footnotesize WK}}=0$ for \eqref{Westervelt} and $s_{\mbox{\footnotesize WK}}=1$ for \eqref{Kuznetsov}.

\revision{
Further simplifications of the model lead to the Khokhlov-Zabolotskaya-Kuznet-sov (KZK) equation \cite{ZabolotskayaKhokhlov69}
and the well-known Burgers' equation in one space dimension \cite{Burgers74,Evansbuch,Ockendon06}.
}
\Margin{rep C, 2.}

We mention in passing that there exist also more complex and higher order models. Since it takes into account most of the relevant physical effects, the Westervelt equation appears to be the best established model in the physics and engineering literature of nonlinear acoustics. We will therefore also adhere to this model here.

\subsubsection*{The inverse problem}
The imaging task under consideration consists of identifying $\kappa=\kappa(x)$ in the Westervelt equation in pressure formulation (skipping the subscript $\sim$)
\begin{equation}\label{eqn:Westervelt_p}
p_{tt}-c^2\triangle p - b\triangle p_t = \kappa(x)(p^2)_{tt} + r(x,t)
\end{equation}
($\kappa(x)= \frac{1}{\rho_0c^2} \beta_a(x)$) or in velocity potential formulation
\begin{equation}\label{eqn:Westervelt_psi}
\psi_{tt}-c^2\triangle \psi - b\triangle \psi_t = \kappa(x)(\psi_t^2)_t + r(x,t)
\end{equation}
($\kappa(x)= \frac{1}{c^2} \beta_a(x)$)
where $r$ is a given excitation function and pressure and velocity potential are related by
\eqref{prhopsitp}.

The system will typically be at rest initially, leading to homogeneous initial conditions on $p$, $p_t$ or $\psi$, $\psi_t$, respectively.
The spatial domain $\Omega\subseteq\mathbb{R}^d$, $d\in\{1,2,3\}$ on which the {\sc pde}s are supposed to hold will be assumed to be smooth and bounded and the Laplacian equipped with boundary conditions on $\partial\Omega$. For simplicity one might think of homogeneous Dirichlet ones here but note that also other boundary conditions -- in-homogeneous Neumann for modeling excitation by a transducer array, absorbing or impedance conditions for modeling damping or reflections at the boundary -- are relevant in this context. Excitation will here be modeled by an interior space and time dependent source term $r$, see also \cite{periodicWestervelt}.
\revision{
Since the actuating piezoelectric transducers are typically arranged in an array, that is, a surface $\Gamma$ lying in the interior of the computational domain, modeling excitation by means of interior sources appears to be justified as follows.
We can consider $r$ as an approximation of a source $q\cdot\delta_\Gamma$  concentrated on $\Gamma$, with $q=(c^2\tilde{r}+b\tilde{r}_t)$, (that is, $\tilde{r}(t)= \frac{1}{b}\int_0^t\exp(-\frac{c^2}{b}(t-s))\tilde{r}(s)\, ds$, in view of the fact that formally
\begin{eqnarray}
&&\left\{\begin{array}{r} p_{tt}-c^2\triangle p - b\triangle p_t - \kappa(x)(p^2)_{tt} =0\mbox{ in }\Omega\\ p=0\mbox{ on }\partial\Omega\\
{}[[\partial_\nu p]]=\tilde{r}\mbox{ on }\Gamma\end{array}\right.
\label{eqn:interiorsource}\\
&&\Leftrightarrow \
\int_\Omega
\Bigl((p_{tt}- \kappa(x)(p^2)_{tt})v + (c^2\nabla p+b\nabla p_t)\cdot\nabla v\Bigr)\, dx = \int_{\Gamma} (c^2\tilde{r}+b\tilde{r}_t) v\, ds
\nonumber\\
&& \hspace*{7cm}\mbox{ for all } v\in H_0^1(\Omega)
\nonumber\\
&&\Leftrightarrow \
\left\{\begin{array}{r} p_{tt}-c^2\triangle p - b\triangle p_t - \kappa(x)(p^2)_{tt} =(c^2\tilde{r}+b\tilde{r}_t) \delta_{\Gamma}\mbox{ in }\Omega\\ p =0\mbox{ on }\partial\Omega\end{array}\right.
\nonumber
\end{eqnarray}
where $[[\partial_\nu p]]$ denotes the jump of the normal derivative of $p$ over the interface $\Gamma$.
With $\tilde{r}/\rho_0$ being the acceleration of the transducer in normal direction, via the force balance \eqref{eq:forcebal}, the jump condition in \eqref{eqn:interiorsource} after time integration corresponds to a matching of velocities at the solid-fluid interface.
}
\Margin{rep C, 1.}

As a counterpart to this imposed excitation, measurements of the acoustic pressure at an array of transducers or hydrophones are made.
Thus the overposed data consists of time trace measurements at some point $x_0$ or on some surface $\Sigma$ contained in $\overline{\Omega}$, \revision{see also Figure \ref{fig:Sigma} below}
\Margin{rep M, 1.}
\begin{equation}\label{eqn:obs}
g(t)=p(x_0,t) \quad \mbox{ or } \quad g(x,t)=p(x,t) \,,\ x\in\Sigma\,, t\in(0,T).
\end{equation}

Thus the pressure formulation \eqref{eqn:Westervelt_p} appears to be the more direct one in the sense that the observations are just point evaluations of the state, whereas \eqref{eqn:Westervelt_psi} would require time integration of the data in order to relate state and observations via \eqref{prhopsitp}. Moreover the quadratic nonlinearity comes in terms of a higher derivative in \eqref{eqn:Westervelt_psi} than in \eqref{eqn:Westervelt_p}, which makes analysis and numerics slightly more involved.
Henceforth we will focus on the pressure formulation of the Westervelt equation \eqref{eqn:Westervelt_p}.

\revision{Thus the inverse problem under consideration here is the following.
Identify $\kappa=\kappa(x)$ in
\[
\begin{array}{rcll} p_{tt}-c^2\triangle p - b\triangle p_t &=& \kappa(x)(p^2)_{tt} + r&\mbox{ in }\Omega\times(0,T)\\
p &=&0&\mbox{ on }\partial\Omega\times(0,T)\\
p(x,0)=p_t(x,0)&=&0&x\in\Omega
\end{array}
\]
from observations \eqref{eqn:obs}. Here, $c^2$, $b$, and $r=r(x,t)$ are known constants and functions, respectively.
}


Since this paper is a first step into the mathematics of this inverse problem
we provide some background information.
Thus in Section \oldref{sec:analysis_forward}, we give some regularity results
on the forward operator.
Section \oldref{sec:injectivity} deals with uniqueness.
Note that much of the results from recovery of space dependent coefficients
from time trace data in
hyperbolic or parabolic {\sc pde}s, see, e.g., \cite{Isakov:2006} and the
references therein, do not apply here.
These are the main differences:
first, the wave equations appearing as models are strongly damped and therefore
behave to some extent like a parabolic {\sc pde}, see Remark \oldref{rem:parabolic} below;
second, the reformulation as parabolic {\sc pde} contains a nonlocal in time term due to the integration operator;
third is the fact that the base pde is nonlinear and furthermore the
 coefficient  $\kappa$ to be determined is intrinsically coupled to this
nonlinearity.
Thus, as we will see, the statement
``but linear inverse equations behave even better'' certainly holds true.
We also show injectivity of the linearised inverse problem at $\kappa=0$
\revision{in case the excitation is chosen as an appropriate combination
of products of a space and time dependent functions.}
\Margin{rep C, 4.}
In this section we will also comment on ill-posedness of the problem,
which is expected to be exponential.
In Section \oldref{sec:regmeth} we discuss some of the classical regularisation
paradigms in the context of the nonlinearity imaging task.
Section \oldref{sec:reconstructions} is devoted to numerical experiments
demonstrating the above iterative reconstruction methods.

We mention that an inverse source problem related to the linearisation of this
inverse problem with a higher order model of nonlinear acoustics has recently
been considered in \cite{YamamotoBK20}.

\revision{Moreover, we point to \cite{ImanuvilovYamamoto:2019}, where in case of different observations, namely data $p(x, t_0)$ with some $t_0 > 0$, global Lipschitz stability is shown even in the more complicated setting of the Lam\'{e} operator replacing the Laplacian for the corresponding linearised inverse problem (but probably extendable to the nonlinear setting).}
\Margin{rep M, 2.}

\def\Honedot{\dot{H}^1(\Omega)} 
\def\Htwodot{\dot{H}^2(\Omega)} 
\section{Analysis of the forward problem}\label{sec:analysis_forward}

We consider the operator $F:\mathcal{D}(F)\to Y$, $F=\mbox{tr}_\Sigma\circ G$,
$G(\kappa)=p$, where $p$ solves \eqref{eqn:Westervelt_p} with homogeneous
and linear boundary conditions
-- which we will assume to hold for all equations appearing in this section
without explicitly mentioning it, thus we will write $\mathcal{A}$ for
$-\triangle$ (or more generally a second order elliptic differential operator)
equipped with these boundary conditions on a sufficiently smooth boundary
$\partial\Omega$ so that for $s\in[0,2]$,
$\dot{H}^s(\Omega):=\mathcal{D}(\mathcal{A}^{s/2})\subseteq H^s(\Omega)$
with continuous embedding and equivalent norms.
This includes, e.g., the Laplacian with Dirichlet
($\dot{H}^1(\Omega)=H_0^1(\Omega)$) or impedance
($\dot{H}^1(\Omega)=H^1(\Omega)$) boundary conditions,
but excludes pure Neumann conditions.

\subsection{Well-definedness and Fr\'{e}chet differentiability}\label{sec:frechet}

For taking differences between $\tilde{p}=G(\tilde{\kappa})$ and $p=G(\kappa)$ the following identity on the quadratic nonlinearity will be useful.
\[
\begin{aligned}
\tilde{\kappa}(\tilde{p}^2)_{tt}-\kappa(p^2)_{tt}
=& 2\tilde{\kappa}\tilde{p}\tilde{p}_{tt}-2\kappa p p_{tt} + 2 \tilde{\kappa}\tilde{p}_t^2-2\kappa p_t^2 \\
=& 2\underline{d\kappa}\tilde{p}\tilde{p}_{tt} + 2\kappa v\tilde{p}_{tt} + 2\kappa p v_{tt}
+ 2\underline{d\kappa}\,\tilde{p}_t^2 + 2 \kappa (\tilde{p}_t+ p_t) v_t
\end{aligned}
\]
where $\underline{d\kappa}=\tilde{\kappa}-\kappa$, $v = \tilde{p}-p$.
This implies that $v=G(\tilde{\kappa})-G(\kappa)$ solves
\begin{equation}\label{eqn:v}
(1-2\kappa p)v_{tt}+c^2\mathcal{A} v +b\mathcal{A} v_t - 2 \kappa (\tilde{p}_t+ p_t)\, v_t - 2\kappa \tilde{p}_{tt} \, v
= 2\underline{d\kappa}(\tilde{p}\tilde{p}_{tt} + \tilde{p}_t^2)\,,
\end{equation}
and the (so far formal) linearisation $z=G'(\kappa)\underline{d\kappa}$ solves
\begin{equation}\label{eqn:z}
(1-2\kappa p)z_{tt}+c^2\mathcal{A} z +b\mathcal{A} z_t - 4 \kappa p_t\, z_t - 2\kappa p_{tt} \, z
= 2\underline{d\kappa}(p p_{tt} + p_t^2)\,,
\end{equation}
so that the first order Taylor remainder $w=G(\tilde{\kappa})-G(\kappa)-G'(\kappa)(\tilde{\kappa}-\kappa)$ satisfies
\begin{equation}\label{eqn:w}
\begin{aligned}
(1-2\kappa p)w_{tt}&+c^2\mathcal{A} w +b\mathcal{A} w_t - 4 \kappa p_t\, w_t - 2\kappa p_{tt} \, w\\
&= 2\underline{d\kappa}(v \tilde{p}_{tt} + p v_{tt} + (\tilde{p}_t+p_t) v_t ) + 2\kappa (v v_{tt} + v_t^2) \,,
\end{aligned}
\end{equation}
in all three cases with homogeneous initial and boundary conditions.
Therefore the following lemma will be useful for estimating these differences
and establishing Fr\'{e}chet differentiability of the forward map.
Note that the existing energy estimates for the linearised version of the Westervelt
equation provide higher regularity and exponential decay, but their transfer to
space dependent parameters would also involve derivatives of the coefficients and
in particular of $\kappa$. However, we wish to allow $\kappa$ to be discontinuous
(e.g., piecewise constant of piecewise continuous) as relevant for the underlying
imaging task in order to reproduce sharp interfaces between tissue with different
properties. Thus the use of some $L^p$ space for $\kappa$ is essential.
On the other hand, as the estimates below show, actually $L^\infty$ appears to be
the minimal requirement on $\kappa$.

In order to handle both well-definedness and differentiability of $G$, we first of all establish some energy estimates for a linear version of the Westervelt equation.
\revision{
In the estimates below we will make use of continuity of the embeddings of $\Htwodot\to L^\infty(\Omega)$ and $\Honedot\to L^6(\Omega)$, more precisely
\begin{equation}\label{eqn:C1C2}
\begin{aligned}
&\|\phi\|_{L^6(\Omega)}\leq C_1 \|\mathcal{A}^{1/2}\phi\|_{L^2(\Omega)} \,, \quad \phi \in \mathcal{D}(\mathcal{A}^{1/2})=\Honedot  \\
&\|\phi\|_{L^\infty(\Omega)}\leq C_2 \|\mathcal{A}\phi\|_{L^2(\Omega)} \,, \quad \phi \in \mathcal{D}(\mathcal{A})=\Htwodot  \,.
\end{aligned}
\end{equation}
Moreover, we will impose smallness of a certain combination of the coefficients $\alpha,\gamma$
\begin{equation}\label{eqn:smallness}
\bar{\gamma}:= \|\tfrac{d}{dt}\ln(\tfrac{\gamma}{\alpha})\|_{L^1(0,T;L^\infty(\Omega))} <\tfrac{1}{16}
\end{equation}
Note that since these may contain common physical coefficients,  the quotient may enable some cancellations.
}
\Margin{rep C, 3.}
\begin{lemma}\label{lem:enest}
For $\alpha,\beta,\gamma\in L^\infty(0,T;L^\infty(\Omega))$, $\alpha,\beta,\gamma>0$, $\delta\in L^\infty(0,T;L^3(\Omega))$,  $\mu\in L^\infty(0,T;L^2(\Omega))$,
\revision{$\ln(\tfrac{\gamma}{\alpha}) \in W^{1,1}(0,T;L^\infty(\Omega))$
with \eqref{eqn:smallness}}
any solution $u$ to
\begin{equation}\label{eqn:linWest}
\alpha u_{tt}+\beta\mathcal{A} u_t + \gamma \mathcal{A} u + \delta u_t + \mu u = f
\end{equation}
satisfies the estimates
\begin{equation}
\|\mathcal{A}^{1/2}u_t\|_{L^\infty(0,t;L^2(\Omega))}^2
+ \|\sqrt{\tfrac{\beta}{\alpha}}\mathcal{A}u_t\|_{L^2(0,t;L^2(\Omega))}^2
(1-16\bar{\gamma})\|\sqrt{\tfrac{\gamma}{\alpha}}\mathcal{A}u\|_{L^\infty(0,t;L^2(\Omega))}^2\label{eqn:enest}
\end{equation}
\begin{equation*}
\begin{aligned}
&\leq 8\Bigl(
\|\mathcal{A}^{1/2}u_t(0)\|_{L^2(\Omega)}^2 + \|\sqrt{\tfrac{\gamma(0)}{\alpha(0)}}\mathcal{A}u(0)\|_{L^2(\Omega)}^2
+\|\tfrac{1}{\alpha\beta}\|_{L^\infty(0,t;L^\infty(\Omega))} \|f\|_{L^2(0,t;L^2(\Omega))}^2\\
&+ C_1^2\|\tfrac{\delta}{\sqrt{\alpha\beta}}\|_{L^\infty(0,t;L^{3}(\Omega))}^2  \|\mathcal{A}^{1/2}u_t\|_{L^2(0,t;L^2(\Omega))}^2\\
&+C_2^2\|\tfrac{\mu}{\sqrt{\alpha\beta}}\|_{L^\infty(0,t;L^2(\Omega))}^2\|\mathcal{A}u\|_{L^2(0,t;L^2(\Omega))}^2\Bigr)\,,
\end{aligned}
\end{equation*}
\begin{equation}
\|u_{tt}\|_{L^2(0,t;L^2(\Omega))} \
\leq
\|\tfrac{1}{\alpha}\|_{L^\infty(0,t;L^\infty(\Omega))} \|f\|_{L^2(0,t;L^2(\Omega))}\label{eqn:enest_utt} \qquad \qquad \qquad \qquad
\end{equation}
\begin{equation*}
\begin{aligned}
&+ \|\tfrac{\beta}{\alpha}\|_{L^\infty(0,t;L^\infty(\Omega))}^{1/2}
\|\sqrt{\tfrac{\beta}{\alpha}}\mathcal{A}u_t\|_{L^2(0,t;L^2(\Omega))}\\
&+ \|\tfrac{\gamma}{\alpha}\|_{L^\infty(0,t;L^\infty(\Omega))}^{1/2} \sqrt{T}
\|\sqrt{\tfrac{\gamma}{\alpha}}\mathcal{A}u\|_{L^\infty(0,t;L^2(\Omega))}\\
&+ C_1 \|\tfrac{\delta}{\alpha}\|_{L^2(0,t;L^3(\Omega))} \|\mathcal{A}^{1/2}u_t\|_{L^\infty(0,t;L^2(\Omega))} + C_2 \|\tfrac{\mu}{\alpha}\|_{L^\infty(0,t;L^2(\Omega))}^2  \|\mathcal{A}u\|_{L^2(0,t;L^2(\Omega))}\,,
\end{aligned}
\end{equation*}
\revision{with $C_1$, $C_2$ as in \eqref{eqn:C1C2}.}
\end{lemma}
\begin{proof}
Dividing \eqref{eqn:linWest} by $\alpha$, multiplying with $\mathcal{A} u_t$ and integrating over $(0,t)\times \Omega$, using the identity
\[
\tfrac{\gamma}{\alpha} \mathcal{A}u \mathcal{A}u_t = \tfrac12 \tfrac{d}{dt} \Bigl(\tfrac{\gamma}{\alpha} (\mathcal{A}u)^2\Bigr) - \tfrac12\Bigl[\tfrac{\gamma}{\alpha}\Bigr]_t (\mathcal{A}u)^2
\]
as well as Young's inequality we obtain
\[
\begin{aligned}
&\tfrac12\|\mathcal{A}^{1/2}u_t(t)\|_{L^2(\Omega)}^2
+ \int_0^t\|\sqrt{\tfrac{\beta}{\alpha}}\mathcal{A}u_t(\tau)\|_{L^2(\Omega)}^2\, d\tau
+ \tfrac12\|\sqrt{\tfrac{\gamma}{\alpha}}\mathcal{A}u(t)\|_{L^2(\Omega)}^2
\\
&=\tfrac12\|\mathcal{A}^{1/2}u_t(0)\|_{L^2(\Omega)}^2 + \tfrac12\|\sqrt{\tfrac{\gamma(0)}{\alpha(0)}}\mathcal{A}u(0)\|_{L^2(\Omega)}^2\\
&\qquad+ \int_0^t \Bigl( (\tfrac{1}{\alpha}f(\tau) - \tfrac{\delta}{\alpha} u_t(\tau) - \tfrac{\mu}{\alpha} u(\tau) ) \mathcal{A}u_t(\tau)
+  \tfrac12\Bigl[\tfrac{\gamma}{\alpha}\Bigr]_t (\mathcal{A}u(\tau))^2\Bigr)\, d\tau\\
&\leq
\tfrac12\|\mathcal{A}^{1/2}u_t(0)\|_{L^2(\Omega)}^2 + \tfrac12\|\sqrt{\tfrac{\gamma(0)}{\alpha(0)}}\mathcal{A}u(0)\|_{L^2(\Omega)}^2\\
&\qquad+\tfrac12\int_0^t\|\sqrt{\tfrac{\alpha}{\beta}}(\tfrac{1}{\alpha}f - \tfrac{\delta}{\alpha} u_t - \tfrac{\mu}{\alpha} u )\|_{L^2(\Omega)}^2\, d\tau
+ \tfrac12 \int_0^t\|\sqrt{\tfrac{\beta}{\alpha}}\mathcal{A}u_t(\tau)\|_{L^2(\Omega)}^2\, d\tau\\
&\qquad + \tfrac12\int_0^t\|\tfrac{\alpha}{\gamma}\, \Bigl[\tfrac{\gamma}{\alpha}\Bigr]_t \|_{L^\infty(\Omega)}\, d\tau
\sup_{\tau\in(0,t)}\|\sqrt{\tfrac{\gamma}{\alpha}}(\mathcal{A}u(\tau))\|_{L^2(\Omega)}^2
\end{aligned}
\]
for any $t\in[0,T]$.
After multiplication by two and moving some terms we get, with $\tfrac{\alpha}{\gamma}\, \Bigl[\tfrac{\gamma}{\alpha}\Bigr]_t = \tfrac{\gamma_t}{\gamma}-\tfrac{\alpha_t}{\alpha}$\revision{$=\tfrac{d}{dt}\ln(\tfrac{\gamma}{\alpha})$},
\[
\begin{aligned}
&\|\mathcal{A}^{1/2}u_t(t)\|_{L^2(\Omega)}^2
+ \|\sqrt{\tfrac{\beta}{\alpha}}\mathcal{A}u_t\|_{L^2(0,t;L^2(\Omega))}^2
+ \|\sqrt{\tfrac{\gamma}{\alpha}}\mathcal{A}u(t)\|_{L^2(\Omega)}^2 \\
&\qquad\qquad- 2\|\revision{\tfrac{d}{dt}\ln(\tfrac{\gamma}{\alpha})}\|_{L^1(0,T;L^\infty(\Omega))} \|\sqrt{\tfrac{\gamma}{\alpha}}\mathcal{A}u\|_{L^\infty(0,t;L^2(\Omega))}^2
\\
&\leq
\|\mathcal{A}^{1/2}u_t(0)\|_{L^2(\Omega)}^2 + \|\sqrt{\tfrac{\gamma(0)}{\alpha(0)}}\mathcal{A}u(0)\|_{L^2(\Omega)}^2
+\|\tfrac{1}{\alpha\beta}\|_{L^\infty(0,t;L^\infty(\Omega))} \|f\|_{L^2(0,t;L^2(\Omega))}^2\\
&\qquad+ \|\tfrac{\delta}{\sqrt{\alpha\beta}}\|_{L^\infty(0,t;L^3(\Omega))}^2  \|u_t\|_{L^2(0,t;L^6(\Omega))}^2
+ \|\tfrac{\mu}{\sqrt{\alpha\beta}}\|_{L^\infty(0,t;L^2(\Omega))}^2  \|u\|_{L^2(0,t;L^\infty(\Omega))}^2\,.
\end{aligned}
\]
Thus, for any $t'\in[0,T]$ taking the supremum over $t\in[0,t']$, using the fact that
\begin{small}$\sup_{t\in[0,t']}\sum_{i=1}^m a_i(t)\geq 2^{-m}\sum_{i=1}^m\sup_{t_i\in[0,t']} a_i(t_i)$\end{small}, and assuming \begin{small}$\|\revision{\tfrac{d}{dt}\ln(\tfrac{\gamma}{\alpha})}\|_{L^1(0,T;L^\infty(\Omega))}$\end{small} to be smaller than $\frac{1}{16}$, we obtain \eqref{eqn:enest}.

To obtain equation~\eqref{eqn:enest_utt}, we just insert
\eqref{eqn:enest} after applying the $L^2(0,T;L^2(\Omega))$ norm to both sides
of the {\sc pde} identity $u_{tt}= \tfrac{1}{\alpha} f
-\frac{\beta}{\alpha}\mathcal{A} u_t - \frac{\gamma}{\alpha} \mathcal{A} u - \frac{\delta}{\alpha} u_t - \frac{\mu}{\alpha} u$.
\end{proof}

We first of all apply this lemma together with a fixed point argument to conclude well-posedness of the nonlinear equation \eqref{eqn:Westervelt_p}, i.e., of
\[
(1-2\kappa p)p_{tt}+c^2\mathcal{A} p +b\mathcal{A} p_t = r+2\kappa p_t^2
\]
by means of Banach's Fixed Point Theorem applied to the operator $\mathcal{T}$
mapping $p$ to a solution $\mathcal{T}(p)=p^+$ of the linear problem
\begin{equation}\label{eqn:calT}
(1-2\kappa p)p^+_{tt}+c^2\mathcal{A} p^+ +b\mathcal{A} p^+_t = r+2\kappa p_t^2\,,
\end{equation}
The difference $\hat{p}^+=\tilde{p}^+-p^+$ between values $\tilde{p}^+=\mathcal{T}(\tilde{p})$ and $p^+=\mathcal{T}(p)$ can be characterized by the {\sc pde}
\begin{equation}\label{eqn:calT_diff}
(1-2\kappa p)\hat{p}^+_{tt}+c^2\mathcal{A} \hat{p}^+ +b\mathcal{A} \hat{p}^+_t
= 2\kappa (\tilde{p}^+_{tt}\, \hat{p} + (p_t+\tilde{p}_t)\, \hat{p}_t)\,,
\end{equation}
where $\hat{p}=\tilde{p}-p$.
Therefore we will make use of Lemma \oldref{lem:enest} with $\alpha=1-2\kappa p$, $\beta\equiv b$, $\gamma\equiv c^2$, $\delta\equiv\mu\equiv0$, both for proving that $\mathcal{T}$ is a self-mapping and for establishing its contractivity.
For this purpose it will be convenient to note down the estimates \eqref{eqn:enest}, \eqref{eqn:enest_utt} in the following somewhat compressed form for this particular setting, assuming additionally that $-\tfrac12\leq 2\kappa p\leq \tfrac12$, which we will actually guarantee by a proper choice of the domain of $\mathcal{T}$, and which implies that $\tfrac12\leq\alpha\leq\tfrac32$.
\begin{eqnarray}
&&\|\mathcal{A}^{1/2}u_t\|_{L^\infty(0,t;L^2(\Omega))}^2
+ \tfrac23 b\|\mathcal{A}u_t\|_{L^2(0,t;L^2(\Omega))}^2
+ \tfrac{2(1-16\bar{\gamma})}{3} c^2\|\mathcal{A}u\|_{L^\infty(0,t;L^2(\Omega))}^2
\nonumber\\
&&\quad\leq 8\Bigl(\|\mathcal{A}^{1/2}u_t(0)\|_{L^2(\Omega)}^2
+2c^2\|\mathcal{A}u(0)\|_{L^2(\Omega)}^2
+\tfrac{2}{b} \|f\|_{L^2(0,t;L^2(\Omega))}^2\Bigr)\,,
\label{eqn:enest_calT}\\
&&\|u_{tt}\|_{L^2(0,t;L^2(\Omega))}
\nonumber\\
&&\quad \leq 2 \|f\|_{L^2(0,t;L^2(\Omega))}
+ \sqrt{2b}\|\sqrt{\tfrac{\beta}{\alpha}}\mathcal{A}u_t\|_{L^2(0,t;L^2(\Omega))}^2
\nonumber\\
&&\quad+\sqrt{2c^2T}\|\sqrt{\tfrac{\gamma}{\alpha}}\mathcal{A}u_t\|_{L^\infty(0,t;L^2(\Omega))}^2
\nonumber\\
&&\quad\leq 2\Bigl( 2 +\sqrt{\tfrac{c^2 T}{(1-16\bar{\gamma})b}}\Bigr)
\|f\|_{L^2(0,t;L^2(\Omega))}\,.
\label{eqn:enest_utt_calT}
\end{eqnarray}
We aim at keeping track of the constants $b$, $c$, $T$, $|\Omega|$ in our
estimates (at least those for establishing $\mathcal{T}$ as a self-mapping),
since in reality they can be of very different orders of magnitude.

In view of the estimates \eqref{eqn:enest_calT}, \eqref{eqn:enest_utt_calT}, we will work on the function space
\[
V = H^2(0,T;L^2(\Omega))\cap W^{1,\infty}(0,T;\Honedot  )\cap H^1(0,T;\Htwodot  )
\]
and consider the fixed point operator defined in \eqref{eqn:calT} $\mathcal{T}:M\to M$ (with the self-mapping property yet to be established) on a bounded subset of $V$
\begin{equation}\label{eqn:M}
\begin{aligned}
M:=\{&u\in V\,:\, \|u\|_{L^\infty(0,T;L^\infty(\Omega))}\leq R_0,\\
&\|\mathcal{A}^{1/2}u_t\|_{L^\infty(0,T;L^2(\Omega))}^2
+ \tfrac23 b\|\mathcal{A}u_t\|_{L^2(0,T;L^2(\Omega))}^2\\
&\qquad\qquad\qquad+ \tfrac{2(1-16\bar{\bar{\gamma}})}{3} c^2\|\mathcal{A}u\|_{L^\infty(0,T;L^2(\Omega))}^2
\leq\frac{R_1^2}{b},\\
&\|u_{tt}\|_{L^2(0,t;L^2(\Omega))}\leq R_2\}
\end{aligned}
\end{equation}
where
\begin{equation}\label{eqn:R0}
R_0\leq\frac{1}{4\|\kappa\|_{L^\infty}}\,, \quad
\bar{\bar{\gamma}}<\frac{1}{16}\,,
\end{equation}
and further conditions will be imposed on $R_1,R_2$, cf. \eqref{eqn:R1}, \eqref{eqn:R2a}, \eqref{eqn:R2b}, \eqref{eqn:R3} below.

Now fix $p\in M$. We are going to prove that then $p^+ = \mathcal{T}p\in M$, that is, $\mathcal{T}$ is a self-mapping on $M$, by applying the estimates \eqref{eqn:enest_calT}, \eqref{eqn:enest_utt_calT} to \eqref{eqn:calT}.
To this end, observe that \eqref{eqn:R0} immediately implies
\[
\tfrac12\leq\alpha\leq\tfrac32 \mbox{ for }\alpha=1-2\kappa p\,.
\]
Moreover,
\[
\bar{\gamma} = \|\tfrac{\alpha_t}{\alpha}\|_{L^1(0,T;L^\infty(\Omega))}
= \|\tfrac{2\kappa p_t}{\alpha}\|_{L^1(0,T;L^\infty(\Omega))}
\leq 4\|\kappa\|_{L^\infty(\Omega)}\|p_t\|_{L^1(0,T;L^\infty(\Omega))}\,,
\]
where by interpolation and with the constant $C_{(1+s)/2}$ of the embedding
$\mathcal{D}(\mathcal{A}^{(1+s)/2})$ $(=\dot{H}^{1+s}(\Omega))\to L^\infty(\Omega)$
for $s>\frac12$ (note that we are trying to be minimal with respect to negative powers of the typically small constant $b$ here)
\begin{equation}\label{eqn:L1}
\begin{aligned}
&\|p_t\|_{L^1(0,T;L^\infty(\Omega))}
\leq C_{(1+s)/2} \sqrt{T}\|\mathcal{A}^{s/2}\mathcal{A}^{1/2}p_t\|_{L^2(0,T;L^2(\Omega))}\\
&\qquad\leq C_{(1+s)/2} \sqrt{T}
\|\mathcal{A}p_t\|_{L^2(0,T;L^2(\Omega))}^s
\|\mathcal{A}^{1/2}p_t\|_{L^2(0,T;L^2(\Omega))}^{1-s}\\
&\qquad\leq C_{(1+s)/2} \sqrt{T}
\|(b\mathcal{A})p_t\|_{L^2(0,T;L^2(\Omega))}^s
\|(b\mathcal{A})^{1/2}p_t\|_{L^2(0,T;L^2(\Omega))}^{1-s}\, b^{-(1+s)/2}\\
&\qquad\leq C_{(1+s)/2} \sqrt{T}
(\sqrt{\tfrac32}R_1)^s
(\sqrt{T}R_1)^{1-s}\, b^{-(1+s)/2}
\end{aligned}
\end{equation}
(where we have used the first two terms in the $R_1$ estimate of \eqref{eqn:M}, after multiplying this estimate with $b$)
so that we get $\bar{\gamma}\leq\bar{\bar{\gamma}}$, provided
\begin{equation}\label{eqn:R1}
4\|\kappa\|_{L^\infty(\Omega)} C_{(1+s)/2}
(\tfrac32)^{s/2} T^{1-s/2}\, R_1\, b^{-(1+s)/2}
\leq\bar{\bar{\gamma}}\,.
\end{equation}
Thus, the assumptions of Lemma \oldref{lem:enest} are satisfied and we can make use of estimates \eqref{eqn:enest_calT}, \eqref{eqn:enest_utt_calT} with $f=r+2\kappa p_t^2$ cf. \eqref{eqn:calT}.
For this purpose we estimate
\[
\|f\|_{L^2(0,T;L^2(\Omega))} \leq
\|r\|_{L^2(0,T;L^2(\Omega))} + 2\|\kappa\|_{L^\infty(\Omega)} \|p_t\|_{L^4(0,T;L^4(\Omega))}^2
\]
where by H\"older's inequality with exponent $\frac32$
\begin{equation}\label{eqn:L4}
\begin{aligned}
\|p_t\|_{L^4(0,T;L^4(\Omega))}^2 =& \Bigl(\int_0^T\int_\Omega p_t^4\, dx\, dt\Bigr)^{1/2}
\leq |\Omega|^{1/6} \Bigl(\int_0^T\|p_t\|_{L^6(\Omega)}^4\, dt\Bigr)^{1/2}\\
\leq& |\Omega|^{1/6} T^{1/2} C_1^2\|\mathcal{A}^{1/2}p_t\|_{L^\infty(0,T;L^2(\Omega))}^2
\leq |\Omega|^{1/6} T^{1/2} C_1^2 \frac{R_1^2}{b}
\end{aligned}
\end{equation}
Thus \eqref{eqn:enest_calT}, \eqref{eqn:enest_utt_calT} yield the $R_1$ and $R_2$ estimates on $p^+$ in \eqref{eqn:M} provided that for
\begin{equation}\label{eqn:fbar}
\bar{f}:= \|r\|_{L^2(0,T;L^2(\Omega))} + 2\|\kappa\|_{L^\infty(\Omega)} |\Omega|^{1/6} T^{1/2} C_1^2 \frac{R_1^2}{b}
\end{equation}
we can guarantee
\begin{equation}\label{eqn:R2a}
8\Bigl(
\|\mathcal{A}^{1/2}u_1\|_{L^2(\Omega)}^2 + 2c^2\|\mathcal{A}u_0\|_{L^2(\Omega)}^2
+ \frac{2}{b}\bar{f}^2\Bigr) \leq \frac{R_1^2}{b}
\end{equation}
and
\begin{equation}\label{eqn:R2b}
2\Bigl( 2 +\sqrt{\tfrac{c^2 T}{(1-16\bar{\gamma})b}}\Bigr)\bar{f} \leq R_2\,.
\end{equation}
Finally, the $R_0$ estimate on $p^+$ in \eqref{eqn:M} follows from the $R_1$ estimate on $p^+$ in \eqref{eqn:M} (which we have just established) provided that
\begin{equation}\label{eqn:R3}
C_2 \Bigl({\frac{3}{2(1-16\bar{\bar{\gamma}})bc^2}} R_1\Bigr)^{1/2} \leq R_0\,,
\end{equation}
since $\|p^+\|_{L^\infty(0,T;L^\infty(\Omega))}\leq C_2 \|\mathcal{A}p^+\|_{L^\infty(0,T;L^2(\Omega))}$.
Thus we have shown that $\mathcal{T}$ is a self-mapping, provided \eqref{eqn:R0}, \eqref{eqn:R1}, \eqref{eqn:R2a}, \eqref{eqn:R2b}, \eqref{eqn:R3} hold. These can be achieved by making $\|\kappa\|_{L^\infty(\Omega)}$ small:
Choose $R_1 > 4 \|r\|_{L^2(0,T;L^2(\Omega))}$, set $R_0:=C_2 \sqrt{\frac{3}{2(1-16\bar{\bar{\gamma}})bc^2}}$\\ $R_1$,
$R_2:=2\Bigl( 2 +\sqrt{\tfrac{c^2 T}{(1-16\bar{\gamma})b}}\Bigr)\bar{f}$, so that we have \eqref{eqn:R2b}, \eqref{eqn:R3};
then possibly decrease $\|\kappa\|_{L^\infty(\Omega)}$ to achieve \eqref{eqn:R0}, \eqref{eqn:R1}, \eqref{eqn:R2a}.
These requirements also show that smallness of $b$ can to some extent be compensated by smallness of $T$, cf. \eqref{eqn:R1}, \eqref{eqn:fbar}, \eqref{eqn:R2b}.

We now proceed to derive contractivity of $\mathcal{T}$ by applying the
estimates \eqref{eqn:enest_calT}, \eqref{eqn:enest_utt_calT} to the {\sc pde}
\eqref{eqn:calT_diff}.
In the above, we have already shown that for any $p\in M$,
under conditions \eqref{eqn:R0}, \eqref{eqn:R1}, \eqref{eqn:R2a},
\eqref{eqn:R2b}, \eqref{eqn:R3}, the coefficients $\alpha=1-2\kappa p_t$,
$\beta\equiv b$, $\gamma\equiv c^2$, $\delta\equiv0$, $\mu\equiv0$ satisfy the
assumptions of Lemma \oldref{lem:enest} with $\frac12\leq\alpha\leq\frac32$,
$\bar{\gamma}\leq\bar{\bar{\gamma}}<\frac{1}{16}$, so that the estimates
\eqref{eqn:enest_calT}, \eqref{eqn:enest_utt_calT} with $u(0)=0$, $u_t(0)=0$
apply to \eqref{eqn:calT_diff}. Using them together with Gronwall's inequality
yields existence of  constants $C,\;\tilde{C}>0$
(depending only on the constants $R_0,R_1,R_2,\bar{\bar{\gamma}}$
in the definition of $M$ as well as $b,c^2$ and $T$), such that
\[
\begin{aligned}
&\|\hat{p}^+\|_V
\leq C \|2\kappa (\tilde{p}^+_{tt}\, \hat{p} + (p_t+\tilde{p}_t)\, \hat{p}_t)\|_{L^2(0,T;L^2(\Omega))}\\
&\qquad\leq 2C \|\kappa\|_{L^\infty(\Omega)}
\Bigl(\|\tilde{p}^+_{tt}\|_{L^2(0,T;L^2(\Omega))}\|\hat{p}\|_{L^\infty(0,T;L^\infty(\Omega))}
\\&\qquad\qquad\qquad\qquad+ \|p_t+\tilde{p}_t\|_{L^4(0,T;L^4(\Omega))} \|\hat{p}_t\|_{L^4(0,T;L^4(\Omega))}\Bigr)\\
&\qquad\leq 2C \|\kappa\|_{L^\infty(\Omega)}
\Bigl(R_2C_2\|\mathcal{A}\hat{p}\|_{L^\infty(0,T;L^2(\Omega))}
\\&\qquad\qquad\qquad\qquad+ 2|\Omega|^{1/6} T^{1/2} C_1^2 \frac{R_1}{\sqrt{b}} \|\mathcal{A}^{1/2}\hat{p}_t\|_{L^\infty(0,T;L^2(\Omega))}\Bigr)\\
&\qquad\leq \tilde{C} \|\kappa\|_{L^\infty(\Omega)} \|\hat{p}\|_V
\end{aligned}
\]
by \eqref{eqn:C1C2}, \eqref{eqn:M} and an estimate analogous to \eqref{eqn:L4}.
For $\|\kappa\|_{L^\infty(\Omega)}$ small enough this implies contractivity.

\begin{proposition} \label{prop:forward_welldef}
For any $u_0\in \Htwodot $, $u_1\in \Honedot  $, $r\in L^2(0,T;L^2(\Omega))$, $T>0$ and any $C^{1,1}$ domain $\Omega$ there exists $R>0$ such that for all $\kappa\in L^\infty(\Omega)$, $\|\kappa\|_{L^\infty(\Omega)}<R$ the quasilinear {\sc pde} \eqref{eqn:Westervelt_p} with homogeneous Dirichlet boundary conditions and initial conditions $u(0)=u_0$, $u_t(0)=u_1$ is uniquely solvable on the space $M$ defined in \eqref{eqn:M}.

Hence, for any $C^{1,1}$ manifold $\Sigma\subseteq\overline{\Omega}$,
the forward operator $F:\mathcal{D}(F)\to Y$, $F(\kappa)=\mbox{tr}_\Sigma p$
is well-defined and bounded on $\mathcal{D}(F)=B_R(0)\subseteq L^\infty(\Omega)$
and maps into any space $Y$ in which the space
$W^{1,\infty}(0,T;H^{1/2}(\Sigma))\cap H^1(0,T;H^{3/2}(\Sigma))$ is continuously embedded.
\end{proposition}
Since measurements are done on values and not on derivatives of $p$, a natural image space $Y$ of $F$ is actually $L^Q(0,T;L^P(\Omega))$ for some $P,Q\in[1,\infty]$.

\medskip

To prove Fr\'{e}chet differentiability of $F$, we apply the estimates from Lemma \oldref{lem:enest} together with Gronwall's inequality to  \eqref{eqn:v}, \eqref{eqn:z}, \eqref{eqn:w} with $\alpha=1-\kappa p$, $\beta\equiv b$, $\gamma\equiv c^2$,
\[
\delta = \begin{cases}
2\kappa(\tilde{p}_t+p_t) \mbox{ for } \eqref{eqn:v}\\
4\kappa p_t \mbox{ for } \eqref{eqn:z}, \eqref{eqn:w}
\end{cases}\,, \quad
\mu = \begin{cases}
2\kappa\tilde{p}_{tt} \mbox{ for } \eqref{eqn:v}\\
2\kappa p_{tt} \mbox{ for } \eqref{eqn:z}, \eqref{eqn:w}
\end{cases}\,,
\]
\[
f=\begin{cases}
2\underline{d\kappa}(\tilde{p}\tilde{p}_{tt} + \tilde{p}_t^2) \mbox{ for } \eqref{eqn:v}\\
2\underline{d\kappa}(p p_{tt} + p_t^2)  \mbox{ for } \eqref{eqn:z}\\
2\underline{d\kappa}(v \tilde{p}_{tt} + p v_{tt} + (\tilde{p}_t+p_t) v_t ) + 2\kappa (v v_{tt} + v_t^2)
 \mbox{ for } \eqref{eqn:w}\end{cases}\,.
\]
This together with estimates similar to those in the proof of
Proposition~\oldref{prop:forward_welldef} yields existence of a constant
$C>0$ (depending only on the constants $R_0,R_1,R_2,\bar{\bar{\gamma}}$
in the definition of $M$ as well as $b,\;c^2$, the time horizon $T$, and the
radius $R$ according to Proposition \oldref{prop:forward_welldef}), such that
for any $\kappa\in B_R(0)\subseteq L^\infty(\Omega)$
\[
\begin{aligned}
&\|v\|_V\leq C\|\underline{d\kappa}\|_{L^\infty(\Omega)}\,, \quad
\|z\|_V\leq C\|\underline{d\kappa}\|_{L^\infty(\Omega)}\,, \\
&\|w\|_V\leq C\Bigl(\|\underline{d\kappa}\|_{L^\infty(\Omega)}\|v\|_V+ R\|v\|_V^2)\leq C^2(1+CR)\|\underline{d\kappa}\|_{L^\infty(\Omega)}^2 \,.
\end{aligned}
\]

\begin{proposition} \label{prop:forward_diff}
For any $u_0\in \Htwodot $, $u_1\in \Honedot  $, $r\in L^2(0,T;L^2(\Omega))$, $T>0$ any $C^{1,1}$ domain $\Omega$ and any $C^{1,1}$ manifold $\Sigma\subseteq\overline{\Omega}$, with $R$ as in Proposition \oldref{prop:forward_welldef}, the forward operator $F:\mathcal{D}(F)\to Y$, $F(\kappa)=\mbox{tr}_\Sigma p$ is Fr\'{e}chet differentiable on $\mathcal{D}(F)=B_R(0)\subseteq L^\infty(\Omega)$ with derivative given by $F'(\kappa)\underline{d\kappa} = \mbox{tr}_\Sigma z$ for $z$ solving \eqref{eqn:z} with homogeneous initial and boundary conditions.
\end{proposition}
\begin{remark}\label{rem:contFprime}
Applying a similar reasoning to the difference $\bar{w}=\tilde{z}-z$ between $G'(\tilde{\kappa})=\tilde{z}$ and $G'(\kappa)=z$, which then with $\underline{d\kappa}=\tilde{\kappa}-\kappa$ and $v$ as in \eqref{eqn:v} satisfies (similarly to $w$ in \eqref{eqn:w})
\begin{equation}\label{eqn:wbar}
\begin{aligned}
(1-2\kappa p)\bar{w}_{tt}&+c^2\mathcal{A} \bar{w} +b\mathcal{A} \bar{w}_t - 4 \kappa p_t\, \bar{w}_t - 2\kappa p_{tt} \, \bar{w}\\
&= 2\underline{d\kappa}(v \tilde{p}_{tt} + p v_{tt} + (\tilde{p}_t+p_t) v_t ) + 2\kappa (\tilde{z} v_{tt} + 2v_t\tilde{z}_t)
\end{aligned}
\end{equation}
we obtain Lipschitz continuity of $F'$ on $B_R(0)\subseteq L^\infty(\Omega)$.
\end{remark}

\subsection{Adjoint of $F'(\kappa)$}\label{sec:adj}
The goal of this section is to compute the adjoint of $F'(\kappa)$, which is
needed, for example, for formulating Landweber iteration,
a Gauss-Newton method, or the necessary optimality conditions for Tikhonov
regularisation, as well as for assessing source conditions.
From the implementation point of view (see Section \oldref{sec:regmeth}),
we prefer a Hilbert space setting to a general Banach space one and therefore
use
\[
X=\dot{H}^s(\Omega)\,, \quad Y=L^2(0,T;L^2(\Sigma))\,,
\]
where $\dot{H}^s(\Omega)=\mathcal{D}(\mathcal{A}^{s/2})$ with $s\geq0$ appropriately chosen, see Section \oldref{sec:regmeth}.

The adjoint must satisfy the identity
\begin{equation}\label{eqn:adj_def}
\langle \underline{d\kappa}, F'(\kappa)^*y\rangle_X \stackrel{!}{=} \langle F'(\kappa)\underline{d\kappa}, y\rangle_Y =
\int_0^T\int_\Sigma z\, y \, dS \, dt
\end{equation}
for any $\underline{d\kappa}\in X$, $y\in Y$.
Taking into account \eqref{eqn:z} and the fact that $(1-2\kappa p)z_{tt} - 4 \kappa p_t\, z_t - 2\kappa p_{tt} \, z = \Bigl((1-2\kappa p)z\Bigr)_{tt}$, this leads us to defining an adjoint state as the solution $a$ to
\begin{equation*}
a(T)=0, \ a'(T)=0 \mbox{ and for all } t\in(0,T)\,: \  a(t)\in \Honedot   \mbox{ and for all } \phi\in \Honedot  :
\end{equation*}
\begin{equation}\label{eqn:a_weak}
\int_\Omega\Bigl[ (1-2\kappa p) a_{tt} \, \phi - (b \nabla a_t - c^2 \nabla a)\cdot \nabla \phi\Bigr]\, dx = \int_\Sigma y\, \phi\, dS\,,
\end{equation}


\noindent which is the variational form of
\begin{equation}\label{eqn:a}
\begin{aligned}
(1-2\kappa p) a_{tt} - b\mathcal{A} a_t + c^2\mathcal{A} a = 0 \mbox{ in }\Omega,\
[[\partial_{\nu_\Sigma} (b\,a_t - c^2a)]] = y \mbox{ on }\Sigma
\end{aligned}
\end{equation}
with homogeneous end and boundary conditions.

\smallskip

To see that \eqref{eqn:a_weak} is the weak form of \eqref{eqn:a}, we consider a partition of $\Omega$ into two subdomains, whose interface contains $\Sigma$,
see Figure \oldref{fig:Sigma},
with $\Omega_{in}$, $\Omega_{out}$ open, $\Omega_{in}\cap\Omega_{out}=\emptyset$,
$\overline{\Omega}_{in}\cup\Omega_{out}=\overline{\Omega}$,
$\Sigma\subseteq \overline{\Omega}_{in}\cup\overline{\Omega}_{out}$,
so that the outer normal vector $\nu$ coincides with $\nu_\Sigma$ on $\Sigma\cup\overline{\Omega}_{in}$ and with $-\nu_\Sigma$ on $\Sigma\cup\overline{\Omega}_{out}$ and integrate by parts to obtain
\font\smallsymbol = cmmi8

\xfiglen=0.7 true in
\yfiglen=0.7 true in
\textlen=\hsize \advance\textlen by 1.75\xfiglen
%

\setbox\figureone=\vbox{\hsize=\xfiglen
\beginpicture
  \setcoordinatesystem units <\xfiglen,\yfiglen>  point at 0.0 0.0
  \setplotarea x from -1 to 1, y from -0.7 to 0.6
\linethickness = 0.6pt
   \circulararc 90 degrees from 0.2 0 center at -0.3 0
   \ellipticalarc axes ratio 2:1.3 360 degrees from 1.1 0.0 center at 0.0 0.0
   \arrow <4pt> [0.5,0.7] from 0.1 0.3 to 0.3 0.5
\put {$\Sigma$} at 0.3 0.2
\put {$\nu_\Sigma$} at 0.1 0.5
\put {$\Omega_{\rm in}$} at -0.3 0.0
\put {$\Omega_{\rm out}$} at 0.6 -0.2
\setdashes <3pt> 
   \circulararc 270 degrees from -0.3 0.5 center at -0.3 0
\endpicture
}

\begin{figure}
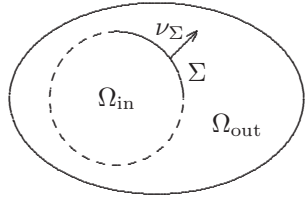

\hbox to \hsize{\hss\copy\figureone\hss}
\caption{\small {\bf The surface $\Sigma$}}
\label{fig:Sigma}
\end{figure}
\[
\begin{aligned}
&\int_\Omega \nabla a\cdot\nabla\phi \, dx =
\int_{\Omega_{in}} \nabla a\cdot\nabla\phi \, dx + \int_{\Omega_{out}} \nabla a\cdot\nabla\phi \, dx\\
&= \int_{\Omega_{in}} \mathcal{A} a \, \phi \, dx + \int_{\Omega_{out}} \mathcal{A} a \, \phi \, dx
- \int_{\partial\Omega_{in}}\partial_\nu a \, \phi\, dS - \int_{\partial\Omega_{out}}\partial_\nu a \, \phi\, dS\\
&= \int_\Omega \mathcal{A} a\,\phi \, dx - \int_\Sigma [[\partial_{\nu_\Sigma} a]] \, \phi\, dS\,,
\end{aligned}
\]
where ${\displaystyle [[\partial_{\nu_\Sigma} a]](x)
=\Bigl(\lim_{\tilde{x}\to x,\, \tilde{x}\in\Omega_{in}}\nabla a - \lim_{\tilde{x}\to x,\, \tilde{x}\in\Omega_{out}}\nabla a\Bigr)\cdot \nu_\Sigma}$.

With $a$ according to \eqref{eqn:a}, we can indeed establish the identity \eqref{eqn:adj_def}: Inserting \eqref{eqn:a_weak} with $\phi=z(t)$ into \eqref{eqn:adj_def} and integrating over $(0,T)$, using integration by parts and the initial/end conditions on $z$ and $a$, respectively, we obtain
\[
\begin{aligned}
\langle F'(\kappa)\underline{d\kappa}, y\rangle_Y =&
\int_0^T\int_\Omega\Bigl[ (1-2\kappa p) a_{tt} \, z - (b \nabla a_t - c^2 a)\cdot \nabla z\Bigr]\, dx\, dt\\
=&
\int_0^T\int_\Omega\Bigl[ ((1-2\kappa p) z)_{tt} \, a + (b \nabla z_t + c^2 z)\cdot \nabla a\Bigr]\, dx\, dt\\
=& \int_0^T\int_\Omega\underline{d\kappa}(p^2)_{tt}\, a \, dx\, dt
=\int_\Omega\underline{d\kappa} \int_0^T(p^2)_{tt}\, a \, dt\, dx\\
=&\int_\Omega\mathcal{A}^{s/2} \underline{d\kappa} \, \int_\Omega\mathcal{A}^{s/2}\Bigl[\mathcal{A}^{-s}\Bigl(\int_0^T(p^2)_{tt}\, a \, dt\Bigr)\Bigr]\, dx\,,
\end{aligned}
\]
where we have also used the weak form of \eqref{eqn:z} with $a$ as a test function.
Thus with $\langle \kappa, \mu\rangle_X=: \int_\Omega \mathcal{A}^{s/2} \kappa\, \mathcal{A}^{s/2} \mu\,dx$ (where we can replace $\mathcal{A}^{s/2}$ by $\nabla$ in case $s=1$ and $\mathcal{A}$ is equipped with homogeneous Dirichlet conditions), we get the following.
\begin{proposition}\label{prop:adj}
Under the assumptions of Proposition \oldref{prop:forward_diff}, the \begin{small}$\dot{H}^s(\Omega) - L^2(0,T;L^2(\Sigma))$\end{small} adjoint of $F'(\kappa)$ is given by $F'(\kappa)^*y = \mathcal{A}^{-s}\Bigl(\int_0^T(p^2)_{tt}\, a \, dt\Bigr)$, where $a$ solves \eqref{eqn:a} with $a(T)=a_t(T)=0$.
\end{proposition}

\subsection{Weak sequential closedness}

Application of convergence results for variational regularisation methods like Tikhonov, Ivanov, or Morozov regularisation \cite{SeidmanVogel89,EKN89,HKPS07,Ivanov62,KKlassen17,LorenzWorliczek13,Morozov66} requires $F$ to be weak* sequentially closed, i.e.,
\begin{equation}\label{eqn:Fseqclos}
\Bigl(\kappa_j\stackrel{*}{\rightharpoonup} \kappa \mbox{ in }X\ \mbox{ and }
F(\kappa_j)\stackrel{*}{\rightarrow} y \mbox{ in }Y\Bigr) \ \Longrightarrow \
\Bigl(\kappa \in \mathcal{D}(F) \mbox{ and } F(\kappa)=y\Bigr)\,.
\end{equation}
Here we return to $X=L^\infty(\Omega)$, which
also implies \eqref{eqn:Fseqclos} with  $X=\dot{H}^s(\Omega)$ for $s>d/2$.
We verify \eqref{eqn:Fseqclos} by showing
\begin{proposition}\label{prop:Fseqclos}
For any $u_0\in\Htwodot $, $u_1\in \Honedot  $, $r\in L^2(0,T;L^2(\Omega))$, $T>0$ any $C^{1,1}$ domain $\Omega$ and $R$, $\mathcal{D}(F)$ as in Proposition \oldref{prop:forward_welldef}, the operator $G:\mathcal{D}(F)\to Y$, $G(\kappa)=p$  solving \eqref{eqn:Westervelt_p} satisfies
\begin{equation}\label{eqn:Gseqclos}
\Bigl(\kappa_j\stackrel{*}{\rightharpoonup} \kappa \mbox{ in }X\ \mbox{ and }
G(\kappa_j)\stackrel{*}{\rightharpoonup} g \mbox{ in }V\Bigr) \ \Longrightarrow \
\Bigl(\kappa \in \mathcal{D}(F) \mbox{ and } G(\kappa)=g\Bigr)\,.
\end{equation}
Thus for any $C^{1,1}$ manifold $\Sigma\subseteq\overline{\Omega}$, $Y$ as in Proposition \oldref{prop:forward_welldef}, the forward operator $F:\mathcal{D}(F)\to Y$, $F(\kappa)=\mbox{tr}_\Sigma p$ satisfies \eqref{eqn:Fseqclos}.
\end{proposition}
\begin{proof}
Abbreviating $\tilde{p}_j=G(\kappa_j)$, $p=G(\kappa)$, $v_j=\tilde{p}_j-p$, $\underline{d\kappa}_j=\kappa_j-\kappa$, we get, analogously to \eqref{eqn:v} that for any $\phi\in C_0^\infty(\Omega)$, $\psi\in C_0^\infty(0,T)$
\begin{equation}\label{eqn:vj}
\begin{aligned}
&\int_0^T\int_\Omega\psi(t)\phi(x)\Bigl(
(1-2\kappa p)v_{j,tt}+c^2\mathcal{A} v +b\mathcal{A} v_{j,t} \\
&\hspace*{3cm}- 2 \kappa (\tilde{p}_{j,t}+ p_t)\, v_{j,t} - 2\kappa \tilde{p}_{j,tt} \, v_j
- \underline{d\kappa}_j(\tilde{p}_j^2)_{tt}\Bigr)\, dx\, dt=0\,.
\end{aligned}
\end{equation}
For the linear (with respect to the sequences) terms we clearly get convergence from $v_j\stackrel{*}{\rightharpoonup} g-p=:\underline{dp}$  in $V$, that is,
\[
\begin{aligned}
&\int_0^T\int_\Omega\psi(t)\phi(x)\Bigl(
(1-2\kappa p)v_{j,tt}+c^2\mathcal{A} v +b\mathcal{A} v_{j,t} - 2 \kappa p_t\, v_{j,t} - 2\kappa p_{tt} \, v_j\Bigr)\, dx\, dt\\
&\ \to \
\int_0^T\int_\Omega\psi(t)\phi(x)\Bigl(
(1-2\kappa p)\underline{dp}_{tt}+c^2\mathcal{A} \underline{dp} +b\mathcal{A} \underline{dp}_t - 2 \kappa p_t\, \underline{dp}_t - 2\kappa p_{tt} \, \underline{dp}\Bigr)\, dx\, dt
\,.
\end{aligned}
\]
It remains to consider the nonlinear terms.
\\
\indent For $\int_0^T\int_\Omega\psi(t)\phi(x) \kappa \tilde{p}_{j,t}\, v_{j,t} \, dx\, dt$,
we can make use of compactness of the embedding $V\to W^{1,4}(0,T;L^4(\Omega))$
so that there exists a subsequence $j_\ell$ along which the convergence
$\tilde{p}_{j,t}\to g_t$ and $v_j\to \underline{dp}$ is actually strong in
$L^4(0,T;L^4(\Omega))$ and we get
\[
\begin{aligned}
&\Big|\int_0^T\int_\Omega\psi(t)\phi(x) \kappa \Bigl(\tilde{p}_{j_\ell,t}\, v_{j_\ell,t} - g_t\underline{dp}_t\Bigr)\, dx\, dt\Big|\\
&\leq \|\kappa\|_{L^\infty(\Omega)} \|\phi\|_{L^2(\Omega)} \|\psi\|_{L^2(0,T)}
\Bigl(
\|\tilde{p}_{j_\ell,t}\|_{L^4(0,T;L^4(\Omega))} \|v_{j_\ell,t} - \underline{dp}_t\|_{L^4(0,T;L^4(\Omega))}\\
&\qquad\qquad\qquad\qquad\qquad\qquad
+ \|\tilde{p}_{j_\ell,t} - g_t\|_{L^4(0,T;L^4(\Omega))} \|\underline{dp}_t\|_{L^4(0,T;L^4(\Omega))}\Bigr)
\to0
\end{aligned}
\]
as $\ell\to\infty$. A subsequence-subsequence argument yields convergence of
$\int_0^T\int_\Omega\psi(t)\phi$\\$(x) \kappa \tilde{p}_{j,t}\, v_{j,t} \, dx\, dt$
to $\int_0^T\int_\Omega\psi(t)\phi(x) \kappa g_t\, \underline{dp} \, dx\, dt$.

We cannot make immediate use of such a compactness argument for
$\int_0^T\int_\Omega\psi(t)\phi$\\$(x) \kappa \tilde{p}_{j,tt} \, v_j \, dx\, dt$,
since $\tilde{p}_{j,tt}$ is just in $L^2(0,T;L^2(\Omega))$ and no more.
Instead we integrate by parts with respect to time to obtain,
for any $W^{1,4}(0,T;L^4(\Omega))$ convergent subsequence
\[
\begin{aligned}
&\int_0^T\int_\Omega\psi(t)\phi(x) \kappa \Bigl(\tilde{p}_{j_\ell,tt}\, v_{j_\ell} - g_{tt}\underline{dp}\Bigr)\, dx\, dt\\
&\quad=\int_0^T\int_\Omega\Bigl\{
\psi(t)\phi(x) \kappa \Bigl(\tilde{p}_{j_\ell,t}\, v_{j_\ell,t} - g_t\underline{dp}_t\Bigr)
+\psi'(t)\phi(x) \kappa \Bigl(\tilde{p}_{j_\ell,t}\, v_j - g_t\underline{dp}\Bigr)
\Bigr\}\, dx\, dt\,,
\end{aligned}
\]
where the first term can be tackled with exactly the same
$L^4(0,T;L^4(\Omega))$ convergence argument as above, and the second term can
be estimated even slightly easier in an analogous way:
\[
\begin{aligned}
&\Big|\int_0^T\int_\Omega\psi'(t)\phi(x) \kappa \Bigl(\tilde{p}_{j_\ell,t}\, v_{j_\ell} - g_t\underline{dp}\Bigr)\, dx\, dt\Big|\\
&\quad\leq \|\kappa\|_{L^\infty(\Omega)} \|\phi\|_{L^2(\Omega)} \|\psi'\|_{L^2(0,T)}
\Bigl(
\|\tilde{p}_{j_\ell,t}\|_{L^4(0,T;L^4(\Omega))} \|v_{j_\ell} - \underline{dp}\|_{L^4(0,T;L^4(\Omega))}\\
&\qquad\qquad\qquad
+ \|\tilde{p}_{j_\ell,t} - g_t\|_{L^4(0,T;L^4(\Omega))} \|\underline{dp}\|_{L^4(0,T;L^4(\Omega))}\Bigr)
\to0\qquad \mbox{as\ }
\ell\to\infty.
\end{aligned}
\]
Similarly, for
$\int_0^T\int_\Omega\psi(t)\phi(x) \underline{d\kappa}_j(\tilde{p}_j^2)_{tt} \, dx\, dt$ we get
\[
\begin{aligned}
&\Big|\int_0^T\int_\Omega\psi(t)\phi(x) \underline{d\kappa}_j (\tilde{p}_j^2)_{tt}\, dx\, dt\Big|
=\Big|\int_0^T\int_\Omega\psi''(t)\phi(x) \underline{d\kappa}_j \tilde{p}_j^2\, dx\, dt\Big|\\
&\leq\Big|\int_0^T\int_\Omega\psi''(t)\phi(x) \underline{d\kappa}_j g^2\, dx\, dt\Big|\\
&\qquad+\|\psi''\|_{L^\infty(0,T)} \|\phi\|_{L^\infty(\Omega)} \|\underline{d\kappa}_j\|_{L^\infty(\Omega)}
\int_0^T\int_\Omega |\tilde{p}_j^2-g^2|\, dx\, dt\\
\end{aligned}
\]
where the first term tends to zero by weak convergence of $\underline{d\kappa}_j$ to zero and the second one by $L^4(0,T;L^4(\Omega))$ convergence of $\tilde{p}_j$ to $g$ along subsequences.

Taking the limit $j\to\infty$ in \eqref{eqn:vj} and making use of the fact that $\phi\in C_0^\infty(\Omega)$, $\psi\in C_0^\infty(0,T)$ are arbitrary, we therefore get
\[
(1-2\kappa p)\underline{dp}_{tt}+c^2\mathcal{A} \underline{dp} +b\mathcal{A} \underline{dp}_t - 2 \kappa (p_t+g_t)\, \underline{dp}_t - 2\kappa g_{tt} \, \underline{dp}
= 0\,.
\]
(This identity actually holds in $L^2(0,T;L^2(\Omega)$.)
Moreover, like $\tilde{p}_j$ and $p$, also $g$ satisfies the initial conditions $g(0)=p_0$, $g_t(0)=p_1$, so $\underline{dp}=g-p$ satisfies homogeneous initial conditions.
Lemma \oldref{lem:enest} -- which is applicable due to $v,g\in V$, $p\in M$ -- together with Gronwall's inequality therefore yields $\underline{dp}=0$, thus $g=p$.

The conclusion \eqref{eqn:Fseqclos} on $F$ follows by boundedness of the trace operator $\mbox{tr}_\Sigma $, which is linear and therefore weak* continuous.
\end{proof}

\begin{remark}\label{rem:parabolic}
We mention in passing that due to the strong damping present in the equation, one might, alternatively to the second order wave equation, use a first order in time heat equation type reformulation.
Integrating with respect to time we get, in place of \eqref{eqn:Westervelt_p}, 
the parabolic {\sc pde} with memory
\begin{equation}\label{eqn:Westervelt_p_int}
p_t - b\triangle p -c^2\triangle \int_0^t p(\tau)\, d\tau = \kappa(x)(p^2)_t + R(x,t)
\end{equation}
where $R(x,t)=\int_0^t r(x,\tau)\, d\tau$.
Note that the right hand side also contains the first time derivative of the state, so we rewrite the equation as
\[
(1-2\kappa(x)p) p_t - b\triangle p -c^2\int_0^t\triangle p(\tau)\, d\tau =  R(x,t)
\]
Although the underlying {\sc pde} is clearly supposed to model wave propagation,
this parabolic reformulation is justified mathematically by the fact that due to the strong damping, the
linearisation of \eqref{eqn:Westervelt_p} is known to give rise to an analytic
semigroup and maximal parabolic regularity in the appropriate spaces, cf.
\cite{KL09,MeyerWilke11}.
We make use of this reformulation mainly for numerical computations where it allows
us to adapt a highly efficient Crank-Nicolson solver for the forward problem.
\end{remark}

\section{Uniqueness and ill-posedness}\label{sec:injectivity}

In either of the formulations the nonlocal parabolic \eqref{eqn:Westervelt_p_int}
or the hyperbolic \eqref{eqn:Westervelt_p} a time trace data can be expected to give valuable
information on unknown spatially-dependent coefficients occurring in the
operator.  For linear equations this idea goes back to at least the work of
Pierce \cite{Pierce:1979}, where the unique determination of a potential
$q$ in $\,u_t -u_{xx} + q(x) u = 0$ could be obtained from measurements of the
overposed value of $h(t) = u(1,t)$ give Neumann base conditions at that point.
The technique was to convert the problem to one of Sturm-Liouville type and
thus the method worked for all the physically important coefficients in a
parabolic equation in one space dimension by means of the Liouville transform.
As we will see later in this section the recovery of the spectral information
relies on an analytic continuation argument.
In its practical aspect it amounts to recovering the coefficients
and exponents in a Dirichlet series representing an analytic function
(which is merely sampled at a discrete set of points).
Even after this has been accomplished what must be recovered is the variation
of the actual sequence of eigenvalues of the operator from their expected
asymptotic form given by the Weyl formula.
However,  this means the information content is a decreasing sequence
(at least in $\ell^2$ but with additional smoothness on the
coefficients can have a much faster decay) that has to be
recovered when masked by an increasing sequence that is quadratically growing.
See,~\cite{RundellSacks:1992a} for details.
Thus even in the case of a linear equation the recovery of a spatial
coefficient from time trace data is exponentially ill-conditioned and
any constants that one might hope work in ones favour turn out to be in
opposition.

However the method breaks down completely in the case of nonlinear equations.
In particular, our unknown coefficient $\kappa$ appears intrinsically coupled
to the nonlinearity.
Thus instead we shall make do with a local injectivity argument for the
recovery of $\kappa$ leaving the broader issue of actual uniqueness
of the inverse problem for future work.
Although this is certainly a lesser result,
it does provide hope and insight on possible uniqueness
for the original nonlinear inverse problems and is also crucial for
linearisation methods such as  Newton and Halley.
In particular, we shall show that it leads to the frozen Newton and
frozen Halley schemes being well-defined and justifies our reconstruction
methods to be presented in Section~\oldref{sec:reconstructions}.

From \eqref{eqn:z}, \eqref{eqn:Westervelt_p} we have $F'(0)\underline{d\kappa} = \mbox{tr}_\Sigma z_0$ where $p_0$ and $z_0$ solve
\begin{equation}\label{eqn:p0}
p_{0,tt}+c^2\mathcal{A} p_0 + b\mathcal{A} p_{0,t} = r
\end{equation}
and
\begin{equation}\label{eqn:z0}
z_{0,tt}+c^2\mathcal{A} z_0 +b\mathcal{A} z_{0,t}
 = \underline{d\kappa}(p_0^2)_{tt}\,,
\end{equation}
both with homogeneous initial conditions.

In setting up an experiment designed to recover $\kappa$,
since the initial conditions are zero, we can use non-homogeneous boundary
conditions to drive the system or a nonzero forcing function $r(x,t)$.
We assume here that the latter has been taken and specifically that $r$
has the form
\revision{
\begin{equation}\label{eqn:choicer}
r(x,t) = f(x) \beta''(t) + \mathcal{A}f(x)(c^2\beta(t) + b\beta'(t))
\end{equation}
with some function $f$ in the domain of $\mathcal{A}$ vanishing only on a set of
measure zero and some twice differentiable function $\beta$ of time such that $(\beta^2)''(t_0)\not=0$ for some $t_0>0$.
Considering, for example the simple case $\beta(t)=t$, one can see that due to
the properties of solutions to Poisson's equation (that is, $\triangle f\leq0$ in $int(\Gamma)\subset\Omega$ and  $f > 0$ on $\Gamma$ implies $f>0$ in $int(\Gamma)$ by the maximum principle), this is possible even with
excitations concentrated on a surface $\Gamma$ like those pointed to in the introduction.
With \eqref{eqn:choicer}, the solution $p_0$ of equation~\eqref{eqn:p0} is clearly given by
$p_0(x,t) = f(x) \beta(t)$.
}
With this, now $\underline{d\kappa}(p_0^2)_{tt}$ can be written in the form
\begin{equation}\label{eqn:p0_fac2}
\bigl(\underline{d\kappa}(x)p_0^2(x,t)\bigr)_{tt} = \sum_{j=1}^\infty a_j\phi_j(x)\psi_j(t)
\end{equation}
\noindent where $a_j$ are the coefficients of $\underline{d\kappa} f$ with respect to the basis $(\phi_j)_{\j\in\mathbb{N}}$, $\psi_j=(\beta^2)''$,
and the eigensystem $(\lambda_j,\phi_j)_{j\in\mathbb{N}}$ of the elliptic operator
$\mathcal{A}$ and the functions $\psi_j$ are known and we normalise $\phi_j$
to have $L^2$ norm unity.
\revision{Note that $r$ is the chosen excitation of the system. So the above condition \eqref{eqn:choicer} -- although it might appear as a restriction at first glance -- actually gives a hint on how to select the input $r$ to enable sufficient sensitivity of the observations with respect to the searched for coefficient.
Also note that condition \eqref{eqn:p0_fac2} might be achieved even in more general settings than \eqref{eqn:choicer}.
}
\Margin{rep C, 4.}

We can rewrite equation~\eqref{eqn:z0} as

\begin{equation}\label{eqn:z0j2}
z_j''(t)+c^2\lambda_j z_j(t) +b\lambda_j z_j'(t)
 = a_j \psi_j(t)\,, \quad t>0\,, \qquad z_j(0)=0\,, \ z_j'(0)=0
\end{equation}
for all $j\in\mathbb{N}$, where
\[
z_0(x,t)=\sum_{j=1}^\infty z_j(t)\phi_j(x)\, .
\qquad
\]
Applying the Laplace transform to both sides of \eqref{eqn:z0j2} yields
\begin{equation}\label{eqn:z0jhat2}
\hat{z}_j(s)= \hat w_j(s) a_j \widehat{\psi}_j(s), \qquad
\mbox{where }\ \hat w_j(s) = \frac{1}{s^2+c^2\lambda_j  +b\lambda_j s}
\,, \qquad s\in\mathbb{C}\,,
\end{equation}
where we have used homogeneity of the initial conditions.
The poles of the function $\hat{w}_j(s)=\frac{1}{s^2+c^2\lambda_j  +b\lambda_j s}$ can be explicitly computed
\[
p_{j,\pm} = \tfrac12\left(-b\pm\sqrt{b^2-\tfrac{4c^2}{\lambda_j}}\right)\lambda_j
= \tfrac{2c^2}{-b\mp\sqrt{b^2-\tfrac{4c^2}{\lambda_j}}}
\]
are single, real up to finitely many complex conjugate pairs and they lie strictly in the left half of the complex plane, accumulating only at $-\tfrac{c^2}{b}$ and $-\infty$, more precisely,
\[
p_{j,+} \to -\tfrac{c^2}{b}\,, \quad p_{j,-} \to -\infty \mbox{ with } \frac{p_{j,-}}{\lambda_j}\to-b
\mbox{ as }j\to\infty.
\]

Moreover, they are different for different $\lambda_j$, i.e.,
\begin{equation}
p_{j,+}=p_{k,+} \ \Rightarrow \ \lambda_j=\lambda_k\,.
\end{equation}
Indeed, if $p_{j,+}=p_{k,+}$ then $-\lambda_k(c^2+bp_{k,+})=p_{k,+}^2=p_{j,+}^2=-\lambda_j(c^2+bp_{j,+})=-\lambda_j(c^2+bp_{k,+})$, thus $(\lambda_j-\lambda_k)(c^2+bp_{k,+})=0$ where $(c^2+bp_{k,+})$ does not vanish for otherwise $p_{k,+}^2$ would vanish.

Thus, assuming that $F'(0)\underline{d\kappa} = \mbox{tr}_\Sigma z_0=0$ implies that
\[
0 = \hat{z}_0(x_0,s) = \sum_{j=1}^\infty a_j \phi_j(x_0) \hat{w}_j(s) \widehat{\psi}_j(s)\,, \qquad
\mbox{ for all }s\in\mathbb{C}\,, \ x_0\in  \Sigma\,.
\]
Considering the residues at some pole $p_{m,+}$ yields
\[
\begin{aligned}
0 &= \mbox{Res}(\hat{z}_0(x_0;p_{m,+})
= \sum_{j=1}^\infty a_j \phi_j(x_0) \lim_{s\to p_{m,+}}(s-p_{m,+}) \hat{w}_j(s) \widehat{\psi}_j(s)\\
&=  \sum_{k\in K_m} a_k \phi_k(x_0) \mbox{Res}(\hat{w}_k;p_{m,+}) \widehat{\psi}_k(p_{m,+})
=  \frac{\widehat{\psi}_k(p_{m,+})}{\sqrt{b^2-\tfrac{4c^2}{\lambda_m}}\lambda_m} \sum_{k\in K_m} a_k \phi_k(x_0) \,,
\end{aligned}
\]
where $K_m\subseteq\mathbb{N}$ is an enumeration of the eigenspace basis $(\phi_k)_{k\in K_m}$ corresponding to the eigenvalue $\lambda_m$.
Assuming now that
\begin{equation}\label{eqn:ass_inj_psi}
\widehat{\psi}_k(p_{m,+})\not=0
\end{equation}
and there exists points $x_{0,m,1}, \ldots x_{0,m,N_m}\in \Sigma$, $N_m\geq \# K_m$ such that
\begin{equation}\label{eqn:ass_inj_Sigma}
\mbox{the matrix }\phi_k(x_{0,m,i})_{k\in K_m,i\in\{1,\ldots,N_m\}} \mbox{ has full rank }\# K_m
\end{equation}
we can conclude that $a_k=0$ for all $k\in K_m$.
The same argument goes through with $p_{m,-}$ in place of $p_{m,+}$.


Now since $(p_0^2)_{tt}(t_0)=f (\beta^2)''(t_0)$ only vanishes on a set of measure zero
we can conclude that $\underline{d\kappa}=0$ almost everywhere.

\begin{theorem}\label{th:lininj}
Under the above assumptions \eqref{eqn:p0_fac2}, \eqref{eqn:ass_inj_psi}, \eqref{eqn:ass_inj_Sigma} for all $m\in\mathbb{N}$, $k\in K_m$, the linearised derivative at $\kappa=0$, $F'(0)$ is injective.
\end{theorem}

In particular, \eqref{eqn:ass_inj_Sigma} is satisfied in the spatially 1-dimensional case $\Sigma=\{x_0\}$, where all eigenvalues of $\mathcal{A}$ are single, i.e., $\# K_m=1$ for all $m$, provided none of the eigenfunctions vanish at $x_0$; this can be achieved by taking $x_0$ on the boundary and where $\phi_j$ is subject to non-Dirichlet conditions.

\begin{remark}\label{remark:dont_need_eigenvalues}
Note that the very reasonable assumption that the eigenvalues of the
operator $\mathcal{A}$ are known is in fact partially redundant:
these can be obtained from the above argument in the case of one spatial
dimension and up to multiplicity in the case of higher space dimensions.
\end{remark}
\begin{remark}\label{remark:get_b_c}
Although the ability to obtain the pole locations and hence $\{\lambda_j\}$
given the values of $b$ and $c^2$
from the functions $\hat w_j(s)$ appearing in equation~\eqref{eqn:z0jhat2}
is redundant, we can nevertheless use the fact that $\{\lambda_j\}$ is
already assumed known to obtain the values of both $b$ and $c$ from
the Laplace transform of the overposed data $h(t)$.
Once again, this is an argument based on analytic continuation of
an analytic function and hence severely ill-conditioned.
\end{remark}

\section{Regularization methods}\label{sec:regmeth}

In this section we revisit some of the commonly used regularisation paradigms and discuss their application to the inverse problems of nonlinearity imaging.
To this end, we will write the inverse problem as an operator equation
\[
F(\kappa)=h
\]
where $F:X\to L^2(0,T;L^2(\Sigma))$ is the forward operator analyzed in Section \oldref{sec:analysis_forward} and we consider the choices $X=\dot{H}^s(\Omega)$ or $X=L^\infty(\Omega)$. In place of the exact time trace data $h$, we are given a noisy version $h^\delta$ or actually usually only its sample at a finite number of time instances, which is the setting we use in our numerical experiments in Section \oldref{sec:reconstructions}. The superscript $\delta$ indicates the noise level in
\begin{equation}\label{eqn:delta}
\|F(\kappa_{act})-h^\delta\|\leq\delta
\end{equation}
where $\kappa_{act}$ denotes an exact solution of the inverse problem, i.e., such that $F(\kappa_{act})=h$.

\subsection{Tikhonov regularisation}
\Margin{rep C, 7.}
Certainly the most well-known, most widely used, and -- via the choice of data misfit and regularisation functionals -- most versatile approach is Tikhonov-Philips regularisation.

Using norms for both functionals we define a regulariser $\kappa_\alpha$ by minimizing $J_\alpha$ given by
\[
J_\alpha(\kappa)=\frac12\|F(\kappa)-h^\delta\|_{L^2(0,T;L^2(\Sigma))}^2+\frac{\alpha}{2}\|\kappa-\kappa_0\|_X^2\,,
\]
where $\kappa_0$ is an {\it a priori\/} guess and we either choose
$X=\dot{H}^s(\Omega)$ with $s>\frac{d}{2}$ to be a Hilbert space that is
embedded in $L^\infty(\Omega)$, or directly use the Banach space
$X=L^\infty(\Omega)$, although this makes the minimization of $J_\alpha$
more challenging, due to additional factors of nonlinearity and nonsmoothness.

The actual computation of the Tikhonov regulariser will usually be based on
descent methods using the gradient of the cost function $J_\alpha$.
The use of the Hessian in principle gives better search directions but comes
at a high computational cost, as we also point out in
Section~\oldref{sec:reconstructions} in the context of Halley's method.
The gradient of $J_\alpha$ in the case of $X=\dot{H}^s(\Omega)$ is given by
\[
J_\alpha'(\kappa)=F'(\kappa)^*(F(\kappa)-h)+\alpha\revision{(\kappa-\kappa_0)}\,,
\]
\Margin{rep C, 5.}
with the Hilbert space adjoint as in Proposition \oldref{prop:adj},
\revision{see also the lines following \eqref{eqn:LW} for its implementation in the context of the inverse problem under consideration here.}
For the use of $X=L^\infty(\Omega)$, considering the special case $s=0$
in Proposition \oldref{prop:adj}, we immediately obtain the
$L^\infty(\Omega)$-$L^2(0,T;L^2(\Omega))$ Banach space adjoint of $F'(\kappa)$,
which by  applying a duality mapping on $L^\infty(\Omega)$, in principle allows
one to construct the gradient of $J_\alpha$.
\revision{
However, convergence of gradient type methods in nonreflexive Banach spaces is in general a highly nontrivial question.
If $X$ is the dual of a separable space, as is the case for $X=L^\infty(\Omega)$, the Tikhonov minimization problem can be tackled by means of a duality approach. This also naturally leads to an appropriate discretization that lends itself to iterative minimization by, e.g., a semismooth Newton method, see \cite{CK11,CK18}.
}
\Margin{rep C, 6.}

Propositions \oldref{prop:forward_welldef}, \oldref{prop:Fseqclos} together with \cite[Theorem 1]{SeidmanVogel89}, \cite[Theorem 2.3]{EKN89} or \cite[Theorems 3.1, 3.5]{HKPS07} respectively, yield well-definedness and convergence of Tikhonov regularisation.

\begin{corollary}
Tikhonov regularisation with $X=\dot{H}^s(\Omega)$, $s>\frac{d}{2}$ or $X=L^\infty(\Omega)$ is well-defined and converges in the sense that for data $h^\delta$ satisfying
\begin{equation*} \|h^\delta-\mbox{tr}_\Sigma p_{act}\|_{L^2(0,T;L^2(\Sigma))}\leq\delta
\end{equation*}
 and
\begin{equation*}
\alpha(\delta)\to0, \frac{\delta}{\alpha(\delta)}\to0
\end{equation*}
we have subsequential\footnote{every sequence has a convergent subsequence and the limit of every subsequence is a solution to the inverse problem; in case of uniqueness of $\kappa_{act}$, the whole sequence converges} convergence of $\kappa_{\alpha(\delta)}$ to $\kappa_{act}$ as $\delta\to0$ where convergence takes place strongly in case of $X=\dot{H}^s(\Omega)$, and weakly* in case of $X=L^\infty(\Omega)$.
\end{corollary}


\medskip

An alternative option is to use $X=L^2(\Omega)$ and constrain minimization to a subset $\{\kappa\in L^2(\Omega)\, : \, \underline{\kappa}\leq\kappa\leq\overline{\kappa}\}$ of $L^\infty(\Omega)$, see \cite{Neubauer88}.

\subsection{Iterative methods}

As already mentioned in the previous section, any variational regularisation method such as Tikhonov, Ivanov or Morozov regularisation and versions thereof, that are based on minimization of some cost functional, require employment of some iterative descent algorithm for their numerical implementation.
Alternatively, one can directly use iterative solution methods, equipped with some regularisation, as reconstruction methods.

Due to the above mentioned difficulties with constructing duality mappings in $L^\infty(\Omega)$, we here consider iterative methods in a Hilbert space setting $X=\dot{H}^s(\Omega)$ only.

\subsubsection{Landweber iteration}\label{sec:LW}
The most simple iterative approach consists of performing gradient descent for the least squares functional $\frac12\|F(\kappa)-h^\delta\|^2$, which results in Landweber iteration
\begin{equation}\label{eqn:LW}
\begin{aligned}
&\kappa_{n+1}=\kappa_n+\mu_n F'(\kappa_n)^*\bigl(h^\delta-F(\kappa_n)\bigr).
\end{aligned}
\end{equation}

Therefore, one step of Landweber iteration (see, e.g., \cite{HNS95}) reads as follows
\begin{itemize}
\item given $\kappa=\kappa_n$, solve \eqref{eqn:Westervelt_p} to obtain $p$;
\item set $y=h^\delta-p$;
\item solve \eqref{eqn:a} to obtain $a$ and set $\delta\kappa=\mathcal{A}^{-s}\int_0^Ta(t)(p^2)_{tt}(t)\, dt$;
\item choose a stepsize $\mu_n$ and set $\kappa_{n+1}=\kappa_n+\mu_n\delta\kappa$.
\end{itemize}

Freezing the derivative at a fixed argument, e.g., the starting value, still yields a fixed point iteration, which one might call {\it frozen Landweber}:
\begin{equation}\label{eqn:frozenLW}
\begin{aligned}
&\kappa_{n+1}=\kappa_n+\mu_n F'(\kappa_0)^* (h^\delta-F(\kappa_n))
\,.
\end{aligned}
\end{equation}

By Propositions \oldref{prop:forward_welldef}, \oldref{prop:forward_diff}, \oldref{prop:adj}, Landweber iteration is well-defined.
As they stand, the Landweber steps do not require further regularisation (although there do exist modified versions that also contain regularisation terms, see, e.g., \cite{Sche98}).
Still, in order to avoid explosion of the noise propagated through the iterations, one has to stop after an appropriately chosen number of steps.
A well-established criterion for this purpose is the discrepancy principle, which terminates the iteration as soon as the residual is of the order of the noise level.
\begin{equation}\label{eqn:discrprinc}
k_*=\min\{k\in \mathbb{N}\, : \, \|F(\kappa)-h^\delta\|\leq\tau\delta\}\,,
\end{equation}
for some fixed safety factor $\tau>1$. (In our computations we used $\tau=2$ which is known to be the minimal choice for nonlinear problems \cite{HNS95}.)
The availability of the noise level $\delta$ in practical applications might be
debatable; yet its necessity for establishing convergence guarantees is known
(as Bakushinski's veto \cite{Bakushinski1984}, see also \cite[Theorem 3.3]{EnglHankeNeubauer:1996}).
The stepsize $\mu_k$ might be chosen as a sufficiently small constant cf. \cite{HNS95}, or adapted in each step, cf., e.g., \cite{NeSc95}.

Landweber iteration is known to be notoriously slow. However, with the ingredients above, also state-of the art accelerated versions of Landweber iteration, such as the steepest descent or the minimal error method \cite{NeSc95} or Nesterov iteration and versions thereof \cite{HubmerRamlau2018, Neubauer2017} can be implemented in a straightforward manner.

Also note that alternatively to choosing $s$ is sufficiently large to
enforce continuity of the embedding $H^s(\Omega)\to L^\infty(\Omega)$,
we can just take $s=0$, that is, $X=L^2(\Omega)$, and project the regularised
iterates into an $L^\infty(\Omega)$ ball in order to guarantee well-definedness
of $F$ and $F'$ according to Propositions~\oldref{prop:forward_welldef},
\oldref{prop:forward_diff}.

\subsubsection{Newton's method}\label{sec:Newton}
Significantly faster methods result from using the first order Taylor expansion of the forward operator which results in Newton's method
\begin{equation}\label{eqn:Newton}
\begin{aligned}
&\kappa_{n+1}=\kappa_n+F'(\kappa_n)^{-1} (h-F(\kappa_n))
\end{aligned}
\end{equation}
or its frozen version
\begin{equation}\label{eqn:frozenNewton}
\begin{aligned}
&\kappa_{n+1}=\kappa_n+F'(\kappa_0)^{-1} (h-F(\kappa_n))
\,.
\end{aligned}
\end{equation}
In the context of nonlinearity imaging, $F(\kappa_n)=\mbox{tr}_\Sigma p$ with $p$ solving \eqref{eqn:Westervelt_p} and $F'(\kappa_n)[\underline{d\kappa}]=\mbox{tr}_\Sigma z$ with $z$ solving \eqref{eqn:z}.
In particular, in the frozen Newton case at $\kappa=0$, both equations needed to evaluate $F'(0)[\underline{d\kappa}]=\mbox{tr}_\Sigma z_0$, namely \eqref{eqn:p0} and \eqref{eqn:z0} are linear, while of course we still have to solve the nonlinear {\sc pde} to obtain $F(\kappa_n)$.

By Propositions \oldref{prop:forward_welldef}, \oldref{prop:forward_diff}, and Theorem \oldref{th:lininj}, Newton's method is well-defined; however, in case of noisy data $h^\delta\approx h$ it needs regularisation.
Regularized versions of Newton usually rely on Tikhonov's method applied to the linearised problem, which in a Hilbert space setting leads to the Levenberg-Marquardt method
\begin{equation}\label{eqn:regNewton}
\begin{aligned}
&\kappa_{n+1}=\kappa_n+(F'(\kappa_n)^*F'(\kappa_n)+\alpha_nI)^{-1} F'(\kappa_n)^*(h^\delta-F(\kappa_n))
\end{aligned}
\end{equation}
or its frozen version
\begin{equation}\label{eqn:regfrozenNewton}
\begin{aligned}
&\kappa_{n+1}=\kappa_n+(F'(\kappa_0)^*F'(\kappa_0)+\alpha_nI)^{-1} F'(\kappa_0)^*(h^\delta-F(\kappa_n))
\,.
\end{aligned}
\end{equation}
cf., e.g., \cite{Hanke97,Ried01} or the iteratively regularised Gauss-Newton method, cf., e.g., \cite{Baku92,BNS97,Hoha97}.

Here $\alpha_k$ is a sequence of regularisation parameters that can be chosen a posteriori according to some inexact Newton strategy (see \cite{Hanke97}) or simply as a geometrically decaying sequence $\alpha_n=\theta^n\alpha_0$.
The overall iteration again needs to be appropriately stopped, e.g. by the discrepancy principle \eqref{eqn:discrprinc}.

\subsubsection{Halley's method}\label{sec:Halley}
An even faster iterative method can be achieved by including Hessian information on the forward operator, which leads to Halley's method:
\begin{equation}\label{eqn:Halley}
\begin{aligned}
&\kappa_{n+\frac12}=\kappa_n+F'(\kappa_n)^{-1} (h-F(\kappa_n))\\
&\kappa_{n+1}=\kappa_n+(F'(\kappa_n)+\tfrac12 F''(\kappa_n)[\kappa_{n+\frac12}-\kappa_n,\cdot])^{-1} (h-F(\kappa_n))\,.
\end{aligned}
\end{equation}
Here it can also make sense to freeze evaluation of $F'$ and $F''$ to some
point $\kappa_0$:
\begin{equation}\label{eqn:frozenHalley}
\begin{aligned}
&\kappa_{n+\frac12}=\kappa_n+F'(\kappa_0)^{-1} (h-F(\kappa_n))\\
&\kappa_{n+1}=\kappa_n+(F'(\kappa_0)+\tfrac12 F''(\kappa_0)[\kappa_{n+\frac12}-\kappa_n,\cdot])^{-1} (h-F(\kappa_n))\,.
\end{aligned}
\end{equation}
In this case, the function values and derivatives of $F$ can be computed as
follows
\begin{itemize}
\item
$F(\kappa_n)=\mbox{tr}_\Sigma p$ with $p$ solving \eqref{eqn:Westervelt_p};
\item
$F'(\kappa_0)[\underline{d\kappa}]=\mbox{tr}_\Sigma z$ with $z$ solving \eqref{eqn:z} with $\kappa=\kappa_0$, $p=p_0$;
\item
$F''(\kappa_0)[\underline{d\kappa}^{(1)},\underline{d\kappa}^{(2)}]=\mbox{tr}_\Sigma w$ with $w$ solving
\begin{equation}\label{eqn:Fpp}
\Bigl((1-2\kappa_0 p_0) w\Bigr)_{tt} + b\mathcal{A} w_t +c^2\mathcal{A} w =
2\Bigl( \kappa_0 \, z^{(1)} \, z^{(2)} + p_0 ( \underline{d\kappa}^{(1)} z^{(2)} + \underline{d\kappa}^{(2)} z^{(1)})\Bigr)_{tt}\,,
\end{equation}
and $z^{(i)}$ satisfying \eqref{eqn:z} with $\kappa=\kappa_0$, $p=p_0$, $\underline{d\kappa}=\underline{d\kappa}^{(i)}$.
\end{itemize}
In particular note that the evaluation of $F'$ and $F''$ at $\kappa_n$ lead to linear initial boundary value problems for the same {\sc pde} just with different right hand sides.

Halley's method also needs to be regularised when used with noisy data,
cf. \cite{HeRu00,BK15}.

\bigskip

\begin{remark}\label{rem:conv_iter}
Some remarks on convergence of iterative regularisation methods for the inverse
problem under consideration are in order.
Restrictions on the nonlinearity of $F$ such as the tangential cone condition
are not likely to hold here, in view of the fact that only boundary
observations are available.
Thus, convergence of Newton or gradient type (Landweber) regularisation methods cannot be proven without assuming further regularity.
\revision{According to, e.g., \cite[Theorems 3.17 and 4.12]{KNS08}, for the regularised Landweber iteration and the iteratively regularized Gauss-Newton method, respectively,}
\Margin{rep C, 8.}
there is still the option of obtaining convergence (with rates) under a source condition
$\kappa_{act}-\kappa_0=F'(\kappa)^*w$ for some $w\in Y$ assuming (instead of
the tangential cone condition) only Lipschitz continuity of $F'$, cf. Remark \oldref{rem:contFprime}.
However due to the severe ill-posedness, that is, the infinite smoothing
properties of $F'(\kappa)^*$, this would only yield convergence (with rates)
in case of infinitely smooth initial error
$\kappa_{act}-\kappa_0=F'(\kappa)^*w$.

\end{remark}

\section{Reconstructions}\label{sec:reconstructions}

In this section we show reconstructions of $\kappa$ from time trace data.
The spatial set will be the interval $[0,1]$ and we will take the measurement
point of $h(t)=p(1,t)$ to be the right-hand endpoint $x=1$.

Our numerical implementation uses \eqref{eqn:Westervelt_p}
in the transformed version as in \eqref{eqn:Westervelt_p_int}
and so treat it as a parabolic equation with nonlocal memory
term $c^2\int_0^t \triangle p(\tau)\,d\tau$.
A Crank-Nicolson integrator was used with an inner iteration loop
to handle the nonlinear term $\,-2\kappa\,p\,p_t$.
A Neumann boundary condition was imposed at the right hand endpoint;
the left hand condition could be Dirichlet, Neumann or impedance type.
Typically, in the physical model one would have zero initial conditions but
this isn't necessary for the mathematical formulation.

Data consisted of the time trace measurement $h(t) = p(1,t)$.
As a practical matter we used the above mentioned solver to obtain this data
and collected a sample at $50$ equally spaced points on the interval $[0,T]$.
Uniformly distributed random noise was then added to these values to obtain
$h_{\rm meas}(t)$.
This was then pre-filtered by smoothing and up-resolving to the working
resolution of the number of time points taken ($\sim$400 for the interval
$t\in [0,1]$) for the direct solver used in the inversion routine.
Our reconstructions are mainly
based on two noise levels: $0.1\%$ and $1\%$, which in view of the exponential
ill-posedness appear to be in the most reasonable range.

The unknown $\kappa$ was represented in terms of given basis functions.
Since we wish to make no constraints on the form of
$\kappa$ other than sufficient regularity and positivity,
we do not choose a basis with in-built restrictions as would be
obtained from an eigenfunction expansion.
Instead we used a radial basis set consisting of either
shifted Gaussian functions $b_j(x) := e^{-(x-x_j)^2/\sigma}$
centered at nodal points  $\{x_j\}$
and with width specified by the parameter $\sigma$ or,
since we also wish to reconstruct non-smooth $\kappa$,
we also used chapeau piecewise linear functions.
Many of the figures shown below use the latter type and the interval
$x\in [0,1]$ contained $41$ such basis functions.
We also show a reconstruction
of a piecewise constant $\kappa$ since this situation is physically
meaningful and for this case we also adopted a Haar basis.
In all cases the starting approximation for the iterative methods used
was the value $\kappa=0$.

The assumption of $1-2\kappa p$ being positive and bounded away from zero is easily fulfilled in practice, see, e.g., \cite[Section 2.1]{MNWW19}. Also in our numerical experiments we chose coefficient functions $\kappa$ and excitations $r$ such that this condition holds.

From a physical standpoint $\kappa(x)$ should be nonnegative and in some cases
(for example in Figure~\oldref{fig:Newton_iteration_kappa6}) we forced this
situation by truncating all negative values to zero at the end of
each iteration of the scheme.
Typically, the need for this decreased as the iteration scheme progressed.

The Newton scheme as described in Section~\oldref{sec:Newton}, in particular
the frozen version about $\kappa=0$, performed reliably and
converged rapidly for a wide range of trial $\kappa$ functions.
Regularization was by the Tikhonov method with regularisation parameter
$\alpha_k=0.5^k\alpha_0$ and as a stopping criterion we used the
Discrepancy Principle~\eqref{eqn:discrprinc}.

The Landweber scheme from Section~\oldref{sec:LW} was also effective and
able to give reconstructions under higher noise levels than Newton.
We do not dwell on this situation as the convergence rate was indeed
notoriously slow and we did not implement any
of the possible acceleration schemes mentioned earlier.

We also show reconstructions using Halley's method, again frozen about a fixed
$\kappa$ value (here $\kappa=0$) taking the predictor-corrector approach
of \cite{HeRu00} and as shown in equation~\eqref{eqn:frozenHalley}.

\input colordvi
%
%
%
\font\tenrm=cmr10
\font\teni=cmmi10 \skewchar\teni='177
\font\tensy=cmsy10 \skewchar\tensy='60
\font\tenex=cmex10
\font\tenit=cmti10
\font\tensl=cmsl10
\font\tenbf=cmbx10
\font\tentt=cmtt10
\font\ninerm=cmr9
\font\ninei=cmmi9 \skewchar\ninei='177
\font\ninesy=cmsy9 \skewchar\ninesy='60
\font\nineit=cmti9
\font\ninesl=cmsl9
\font\ninebf=cmbx9
\font\ninett=cmtt9
\font\eightrm=cmr8
\font\eighti=cmmi8 \skewchar\eighti='177
\font\eightsy=cmsy8 \skewchar\eightsy='60
\font\eightit=cmti8
\font\eightsl=cmsl8
\font\eightbf=cmbx8
\font\eighttt=cmtt8
\font\sevenrm=cmr7
\font\seveni=cmmi7 \skewchar\seveni='177
\font\sevensy=cmsy7 \skewchar\sevensy='60
\font\sevenbf=cmbx7
\font\sevenit=cmmi7
\font\sevensl=cmmi7
\font\seventt=cmr7
\font\sixrm=cmr6
\font\sixi=cmmi6 \skewchar\sixi='177
\font\sixsy=cmsy6 \skewchar\sixsy='60
\font\sixbf=cmbx6
\font\fiverm=cmr5
\font\fivei=cmmi5 \skewchar\fivei='177
\font\fivesy=cmsy5 \skewchar\fivesy='60
\font\fivebf=cmbx5
\def\tenpoint{\def\rm{\fam0\tenrm}%
        \textfont0=\tenrm \scriptfont0=\sevenrm \scriptscriptfont0=\fiverm
        \textfont1=\teni \scriptfont1=\seveni \scriptscriptfont1=\fivei
        \textfont2=\tensy \scriptfont2=\sevensy \scriptscriptfont2=\fivesy
        \textfont3=\tenex \scriptfont3=\tenex \scriptscriptfont3=\tenex
        \def\it{\fam\itfam\tenit}%
        \textfont\itfam=\tenit
        \def\sl{\fam\slfam\tensl}%
        \textfont\slfam=\tensl
        \def\bf{\fam\bffam\tenbf}%
        \textfont\bffam=\tenbf \scriptfont\bffam=\sevenbf
                \scriptscriptfont\bffam=\fivebf
        \def\tt{\fam\ttfam\tentt}%
        \textfont\ttfam=\tentt
        \normalbaselineskip=12pt%
        \let\sc=\eightrm        
        \setbox\strutbox=\hbox{\vrule height8.5pt depth3.5pt width0pt}%
        \normalbaselines\rm}
\def\ninepoint{\def\rm{\fam0\ninerm}%
        \textfont0=\ninerm \scriptfont0=\sixrm \scriptscriptfont0=\fiverm
        \textfont1=\ninei \scriptfont1=\sixi \scriptscriptfont1=\fivei
        \textfont2=\ninesy \scriptfont2=\sixsy \scriptscriptfont2=\fivesy
        \textfont3=\tenex \scriptfont3=\tenex \scriptscriptfont3=\tenex
        \def\it{\fam\itfam\nineit}%
        \textfont\itfam=\nineit
        \def\sl{\fam\slfam\ninesl}%
        \textfont\slfam=\ninesl
        \def\bf{\fam\bffam\ninebf}%
        \textfont\bffam=\ninebf \scriptfont\bffam=\sixbf
                \scriptscriptfont\bffam=\fivebf
        \def\tt{\fam\ttfam\ninett}%
        \textfont\ttfam=\ninett
        \normalbaselineskip=11pt%
        \let\sc=\sevenrm        
        \setbox\strutbox=\hbox{\vrule height8pt depth3pt width0pt}%
        \normalbaselines\rm}
\def\eightpoint{\def\rm{\fam0\eightrm}%
        \textfont0=\eightrm \scriptfont0=\sixrm \scriptscriptfont0=\fiverm
        \textfont1=\eighti \scriptfont1=\sixi \scriptscriptfont1=\fivei
        \textfont2=\eightsy \scriptfont2=\sixsy \scriptscriptfont2=\fivesy
        \textfont3=\tenex \scriptfont3=\tenex \scriptscriptfont3=\tenex
        \def\it{\fam\itfam\eightit}%
        \textfont\itfam=\eightit
        \def\sl{\fam\slfam\eightsl}%
        \textfont\slfam=\eightsl
        \def\bf{\fam\bffam\eightbf}%
        \textfont\bffam=\eightbf \scriptfont\bffam=\sixbf
                \scriptscriptfont\bffam=\fivebf
        \def\tt{\fam\ttfam\eighttt}%
        \textfont\ttfam=\eighttt
        \normalbaselineskip=9pt%
        \let\sc=\sixrm  
        \setbox\strutbox=\hbox{\vrule height7pt depth2pt width0pt}%
        \normalbaselines\rm}
\def\sevenpoint{\def\rm{\fam0\sevenrm}%
        \textfont0=\sevenrm \scriptfont0=\fiverm \scriptscriptfont0=\fiverm
        \textfont1=\seveni \scriptfont1=\fivei \scriptscriptfont1=\fivei
        \textfont2=\sevensy \scriptfont2=\fivesy \scriptscriptfont2=\fivesy
        \textfont3=\tenex \scriptfont3=\tenex \scriptscriptfont3=\tenex
        \def\it{\fam\itfam\sevenit}%
        \textfont\itfam=\sevenit
        \def\sl{\fam\slfam\sevensl}%
        \textfont\slfam=\sevensl
        \def\bf{\fam\bffam\sevenbf}%
        \textfont\bffam=\sevenbf \scriptfont\bffam=\fivebf
                \scriptscriptfont\bffam=\fivebf
        \def\tt{\fam\ttfam\seventt}%
        \textfont\ttfam=\seventt
        \normalbaselineskip=8pt%
        \let\sc=\fiverm  
        \setbox\strutbox=\hbox{\vrule height6pt depth2pt width0pt}%
        \normalbaselines\rm}

\font\smallsymbol = cmmi8
\newdimen\xfiglen \newdimen\yfiglen
\xfiglen=2 true in
\yfiglen=10 true in
\newbox\figurelegend
\newbox\figurelegendtwo
\newbox\figurelegendthree
\newbox\figurelegendfour
\newbox\figurelegendfive
\newbox\figureone
\newbox\figuretwo
\newbox\figurethree
\newbox\figurefour
\newbox\figurefive
\newbox\figuresix
\newbox\figureseven
\newbox\figureeight
\newbox\figurenine
\newbox\figureten
\newbox\figureleven
\newbox\figuretwelve
\newbox\figurethirteen
\newbox\figurefourteen
\newbox\figurefifteen


\setbox\figurelegend=\hbox{
\small
\beginpicture
  \setcoordinatesystem units <0.25\xfiglen,0.05\yfiglen> 
  \setplotarea x from 0 to 0.8, y from 0 to 0.6
\linethickness=0.5pt
\sevenrm
\footnotesize
  \setsolid
  \putrule from 0 0.6 to 0.2 0.6 
  \put {$\kappa_{\rm act}$} [l] at 0.3 0.6
  \setdashes <3pt>
  \Blue{\putrule from 0 0.4 to 0.2 0.4 }\relax
  \put {\Blue{Iter 1}} [l] at 0.3 0.4
  \setdots <2pt>
  \Green{\putrule from 0 0.2 to 0.2 0.2 }\relax
  \put {\Green{Iter 2}} [l] at 0.3 0.2
  \setsolid
  \Red{\putrule from 0 0.0 to 0.2 0.0 }\relax
  \put {\Red{Iter 3}} [l] at 0.3 0.0
\endpicture
}
\setbox\figurelegendtwo=\hbox{
\small
\beginpicture
  \setcoordinatesystem units <0.4\xfiglen,0.12\yfiglen> 
  \setplotarea x from 0 to 1, y from 0 to 1
\linethickness=0.5pt
\footnotesize
  \setsolid
  \putrule from 0 1 to 0.3 1 
  \put {$\kappa_{\rm act}$} [l] at 0.4 1
\setplotsymbol ({\tiny{\Red{$\circ$}}})
\plotsymbolspacing 8pt
\plot 0 0.9  0.1 0.9  0.2 0.9  0.3 0.9 0.4 0.9 /  \put {Iter 1} [l] at 0.6 0.9
\setplotsymbol ({\tiny{\Green{$\circ$}}})
\plotsymbolspacing 8pt
\plot 0 0.8  0.1 0.8  0.2 0.8  0.3 0.8  0.4 0.8 / \put {Iter 10} [l] at 0.6 0.8
\setplotsymbol ({\tiny{\Blue{$\circ$}}})
\plotsymbolspacing 8pt
\plot 0 0.7  0.1 0.7  0.2 0.7  0.3 0.7 0.4 0.7 /  \put {Iter 100} [l] at 0.6 0.7
\setplotsymbol ({\tiny{\Red{$\star$}}})
\plotsymbolspacing 8pt
\plot 0 0.6  0.1 0.6  0.2 0.6  0.3 0.6  0.4 0.6 / \put {Iter 200} [l] at 0.6 0.6
\setplotsymbol ({\tiny{\Green{$\star$}}})
\plotsymbolspacing 8pt
\plot 0 0.5  0.1 0.5  0.2 0.5  0.3 0.5  0.4 0.5 / \put {Iter 500} [l] at 0.6 0.5
\setplotsymbol ({\tiny{\Blue{$\star$}}})
\plotsymbolspacing 8pt
\plot 0 0.4  0.1 0.4  0.2 0.4  0.3 0.4 0.4 0.4 / \put {Iter 1000} [l] at 0.6 0.4
\setplotsymbol ({\tiny{\Red{$\bullet$}}})
\plotsymbolspacing 8pt
\plot 0 0.3  0.1 0.3  0.2 0.3  0.3 0.3 0.4 0.3 / \put {Iter 2000} [l] at 0.6 0.3
\setplotsymbol ({\tiny{\Green{$\bullet$}}})
\plotsymbolspacing 8pt
\plot 0 0.2  0.1 0.2  0.2 0.2  0.3 0.2 0.4 0.2 / \put {Iter 3000} [l] at 0.6 0.2
\setplotsymbol ({\tiny{\Blue{$\bullet$}}})
\plotsymbolspacing 8pt
\plot 0 0.1  0.1 0.1  0.2 0.1  0.3 0.1  0.4 0.1 /\put {Iter 5000} [l] at 0.6 0.1
\endpicture
}
\setbox\figurelegendthree=\hbox{
\beginpicture
  \setcoordinatesystem units <0.25\xfiglen,0.01\yfiglen> 
  \setplotarea x from 0 to 1, y from 0 to 5
\linethickness=0.5pt
\scriptsize
  \setsolid
  \putrule from 0 5 to 0.3 5 
  \put {$\kappa_{\rm act}$} [l] at 0.4 5
  \setdots <2pt>
  \Blue{\putrule from 0 4 to 0.3 4 }\relax
  \put {it 1} [l] at 0.4 4
  \setsolid
  \LimeGreen{\putrule from 0 3 to 0.3 3 }\relax
  \put {it 10} [l] at 0.4 3
  \Orange{\putrule from 0 2 to 0.3 2 }\relax
  \put {it 100} [l] at 0.4 2
  \setdashes <3pt>
  \Brown{\putrule from 0 1 to 0.3 1 }\relax
  \put {it 200} [l] at 0.4 1
  \setsolid
  \BrickRed{\putrule from 0 0 to 0.3 0 }\relax
  \put {it 500} [l] at 0.4 0
\endpicture
}
\setbox\figurelegendfour=\hbox{
\beginpicture
  \setcoordinatesystem units <0.25\xfiglen,0.01\yfiglen> 
  \setplotarea x from 0 to 1, y from 2 to 5
\linethickness=0.5pt
\scriptsize
  \setsolid
  \putrule from 0 5 to 0.3 5 
  \put {$\kappa_{\rm act}$} [l] at 0.4 5
  \setsolid
  \Blue{\putrule from 0 4 to 0.3 4 }\relax
  \put {it 1000} [l] at 0.4 4
  \setsolid
  \Red{\putrule from 0 3 to 0.3 3 }\relax
  \put {it 2000} [l] at 0.4 3
  \Goldenrod{\putrule from 0 2 to 0.3 2 }\relax
  \put {it 5000} [l] at 0.4 2
\endpicture
}
\setbox\figurelegendfive=\hbox{
\beginpicture
  \setcoordinatesystem units <0.25\xfiglen,0.008\yfiglen> 
  \setplotarea x from 0 to 1, y from 0 to 5
\linethickness=0.5pt
\scriptsize
  \setsolid
  \putrule from 0 5 to 0.3 5 
  \put {$\kappa_{\rm act}$} [l] at 0.4 5
  \setsolid
  \Blue{\putrule from 0 4 to 0.3 4 }\relax
  \put {it 1} [l] at 0.4 4
  \setdots <2pt>
  \Orange{\putrule from 0 3 to 0.3 3 }\relax
  \put {it 10} [l] at 0.4 3
  \setsolid
  \OliveGreen{\putrule from 0 2 to 0.3 2 }\relax
  \put {it 100} [l] at 0.4 2
  \Goldenrod{\putrule from 0 1 to 0.3 1 }\relax
  \put {it 1000} [l] at 0.4 1
  \Red{\putrule from 0 0 to 0.3 0 }\relax
  \put {it 10000} [l] at 0.4 0
\endpicture
}
%
%
\yfiglen=16 true in
\setbox\figureone=\vbox{\hsize=\xfiglen
\beginpicture
\eightrm
  \setcoordinatesystem units <\xfiglen,\yfiglen> 
  \setplotarea x from 0 to 1, y from 0 to 0.13
  \axis bottom shiftedto y=0 ticks short numbered from 0 to 1 by 0.2 /
  \axis left ticks short numbered from 0 to 0.12 by 0.02 /
\put{\copy\figurelegend} [rt] at 0.9 0.13
\footnotesize
\put {$\kappa(x)$} [lt] at 0.02 0.13
\small
\linethickness=0.5pt
\setplotsymbol ({\sevenrm .})
\setquadratic
\setlinear
\setsolid
\Black{\relax  
\plot
         0         0
    0.0100         0
    0.0200         0
    0.0300         0
    0.0400         0
    0.0500         0
    0.0600         0
    0.0700         0
    0.0800         0
    0.0900         0
    0.1000         0
    0.1100    0.0003
    0.1200    0.0012
    0.1300    0.0027
    0.1400    0.0047
    0.1500    0.0073
    0.1600    0.0104
    0.1700    0.0140
    0.1800    0.0181
    0.1900    0.0227
    0.2000    0.0277
    0.2100    0.0331
    0.2200    0.0389
    0.2300    0.0450
    0.2400    0.0514
    0.2500    0.0580
    0.2600    0.0647
    0.2700    0.0716
    0.2800    0.0784
    0.2900    0.0851
    0.3000    0.0917
    0.3100    0.0979
    0.3200    0.1038
    0.3300    0.1092
    0.3400    0.1140
    0.3500    0.1182
    0.3600    0.1215
    0.3700    0.1239
    0.3800    0.1254
    0.3900    0.1259
    0.4000    0.1254
    0.4100    0.1237
    0.4200    0.1209
    0.4300    0.1171
    0.4400    0.1122
    0.4500    0.1064
    0.4600    0.0996
    0.4700    0.0922
    0.4800    0.0841
    0.4900    0.0755
    0.5000    0.0667
    0.5100    0.0578
    0.5200    0.0490
    0.5300    0.0405
    0.5400    0.0326
    0.5500    0.0253
    0.5600    0.0190
    0.5700    0.0138
    0.5800    0.0097
    0.5900    0.0068
    0.6000    0.0053
    0.6100    0.0051
    0.6200    0.0062
    0.6300    0.0084
    0.6400    0.0117
    0.6500    0.0159
    0.6600    0.0207
    0.6700    0.0260
    0.6800    0.0315
    0.6900    0.0369
    0.7000    0.0419
    0.7100    0.0464
    0.7200    0.0500
    0.7300    0.0526
    0.7400    0.0541
    0.7500    0.0544
    0.7600    0.0534
    0.7700    0.0512
    0.7800    0.0479
    0.7900    0.0437
    0.8000    0.0387
    0.8100    0.0332
    0.8200    0.0275
    0.8300    0.0218
    0.8400    0.0164
    0.8500    0.0116
    0.8600    0.0074
    0.8700    0.0041
    0.8800    0.0018
    0.8900    0.0004
    0.9000         0
    0.9100         0
    0.9200         0
    0.9300         0
    0.9400         0
    0.9500         0
    0.9600         0
    0.9700         0
    0.9800         0
    0.9900         0
    1.0000         0
 /\relax}\relax
\setdashes <3pt>
\Blue{\relax
\plot
  0.0000   0.0000
  0.0150   0.0001
  0.0300   0.0002
  0.0450   0.0004
  0.0600   0.0010
  0.0750   0.0016
  0.0900   0.0030
  0.1050   0.0046
  0.1200   0.0068
  0.1350   0.0098
  0.1500   0.0131
  0.1650   0.0178
  0.1800   0.0229
  0.1950   0.0290
  0.2100   0.0359
  0.2250   0.0433
  0.2400   0.0518
  0.2550   0.0604
  0.2700   0.0695
  0.2850   0.0787
  0.3000   0.0878
  0.3150   0.0961
  0.3300   0.1039
  0.3450   0.1105
  0.3600   0.1154
  0.3750   0.1195
  0.3900   0.1203
  0.4050   0.1199
  0.4200   0.1169
  0.4350   0.1115
  0.4500   0.1049
  0.4650   0.0950
  0.4800   0.0845
  0.4950   0.0725
  0.5100   0.0602
  0.5250   0.0477
  0.5400   0.0363
  0.5550   0.0260
  0.5700   0.0175
  0.5850   0.0119
  0.6000   0.0076
  0.6150   0.0079
  0.6300   0.0096
  0.6450   0.0141
  0.6600   0.0204
  0.6750   0.0275
  0.6900   0.0351
  0.7050   0.0422
  0.7200   0.0480
  0.7350   0.0514
  0.7500   0.0536
  0.7650   0.0514
  0.7800   0.0478
  0.7950   0.0417
  0.8100   0.0341
  0.8250   0.0259
  0.8400   0.0180
  0.8550   0.0108
  0.8700   0.0053
  0.8850   0.0018
  0.9000   0.0000
  0.9150   0.0000
  0.9300   0.0000
  0.9450   0.0000
  0.9600   0.0002
  0.9750   0.0003
  0.9900   0.0000
 /\relax}\relax
\setdots <2pt>
\Green{\relax
\plot
  0.0000   0.0000
  0.0150   0.0000
  0.0300   0.0001
  0.0450   0.0003
  0.0600   0.0007
  0.0750   0.0012
  0.0900   0.0022
  0.1050   0.0035
  0.1200   0.0054
  0.1350   0.0081
  0.1500   0.0112
  0.1650   0.0157
  0.1800   0.0207
  0.1950   0.0268
  0.2100   0.0340
  0.2250   0.0417
  0.2400   0.0507
  0.2550   0.0600
  0.2700   0.0698
  0.2850   0.0797
  0.3000   0.0896
  0.3150   0.0987
  0.3300   0.1072
  0.3450   0.1145
  0.3600   0.1198
  0.3750   0.1242
  0.3900   0.1250
  0.4050   0.1243
  0.4200   0.1209
  0.4350   0.1148
  0.4500   0.1074
  0.4650   0.0965
  0.4800   0.0848
  0.4950   0.0718
  0.5100   0.0584
  0.5250   0.0449
  0.5400   0.0330
  0.5550   0.0222
  0.5700   0.0136
  0.5850   0.0082
  0.6000   0.0043
  0.6150   0.0054
  0.6300   0.0079
  0.6450   0.0134
  0.6600   0.0206
  0.6750   0.0286
  0.6900   0.0368
  0.7050   0.0443
  0.7200   0.0503
  0.7350   0.0536
  0.7500   0.0555
  0.7650   0.0526
  0.7800   0.0484
  0.7950   0.0416
  0.8100   0.0335
  0.8250   0.0248
  0.8400   0.0168
  0.8550   0.0096
  0.8700   0.0043
  0.8850   0.0011
  0.9000   0.0000
  0.9150   0.0000
  0.9300   0.0001
  0.9450   0.0001
  0.9600   0.0003
  0.9750   0.0002
  0.9900   0.0000
 /\relax}\relax
\setdashes <4pt>
\setsolid
\Red{\relax
\plot
  0.0000   0.0000
  0.0150   0.0000
  0.0300   0.0001
  0.0450   0.0002
  0.0600   0.0004
  0.0750   0.0008
  0.0900   0.0016
  0.1050   0.0027
  0.1200   0.0044
  0.1350   0.0069
  0.1500   0.0098
  0.1650   0.0142
  0.1800   0.0192
  0.1950   0.0253
  0.2100   0.0326
  0.2250   0.0405
  0.2400   0.0499
  0.2550   0.0595
  0.2700   0.0697
  0.2850   0.0800
  0.3000   0.0903
  0.3150   0.0997
  0.3300   0.1084
  0.3450   0.1158
  0.3600   0.1211
  0.3750   0.1255
  0.3900   0.1260
  0.4050   0.1251
  0.4200   0.1214
  0.4350   0.1150
  0.4500   0.1072
  0.4650   0.0960
  0.4800   0.0841
  0.4950   0.0709
  0.5100   0.0574
  0.5250   0.0439
  0.5400   0.0321
  0.5550   0.0215
  0.5700   0.0131
  0.5850   0.0079
  0.6000   0.0043
  0.6150   0.0056
  0.6300   0.0084
  0.6450   0.0139
  0.6600   0.0211
  0.6750   0.0292
  0.6900   0.0373
  0.7050   0.0447
  0.7200   0.0505
  0.7350   0.0536
  0.7500   0.0554
  0.7650   0.0525
  0.7800   0.0482
  0.7950   0.0414
  0.8100   0.0333
  0.8250   0.0246
  0.8400   0.0166
  0.8550   0.0095
  0.8700   0.0043
  0.8850   0.0011
  0.9000   0.0000
  0.9150   0.0000
  0.9300   0.0001
  0.9450   0.0001
  0.9600   0.0002
  0.9750   0.0002
  0.9900   0.0001
 /\relax}\relax
\endpicture
}
%
\setbox\figuretwo=\vbox{\hsize=\xfiglen
\beginpicture
\eightrm
  \setcoordinatesystem units <\xfiglen,\yfiglen> 
  \setplotarea x from 0 to 1, y from 0 to 0.13
  \axis bottom shiftedto y=0 ticks short numbered from 0 to 1 by 0.2 /
  \axis left ticks short numbered from 0 to 0.12 by 0.02 /
\footnotesize
\put {$\kappa(x)$} [lt] at 0.02 0.13
\linethickness=0.5pt
\setsolid
\Black{\relax
\plot
         0         0
    0.0100         0
    0.0200         0
    0.0300         0
    0.0400         0
    0.0500         0
    0.0600         0
    0.0700         0
    0.0800         0
    0.0900         0
    0.1000         0
    0.1100    0.0003
    0.1200    0.0012
    0.1300    0.0027
    0.1400    0.0047
    0.1500    0.0073
    0.1600    0.0104
    0.1700    0.0140
    0.1800    0.0181
    0.1900    0.0227
    0.2000    0.0277
    0.2100    0.0331
    0.2200    0.0389
    0.2300    0.0450
    0.2400    0.0514
    0.2500    0.0580
    0.2600    0.0647
    0.2700    0.0716
    0.2800    0.0784
    0.2900    0.0851
    0.3000    0.0917
    0.3100    0.0979
    0.3200    0.1038
    0.3300    0.1092
    0.3400    0.1140
    0.3500    0.1182
    0.3600    0.1215
    0.3700    0.1239
    0.3800    0.1254
    0.3900    0.1259
    0.4000    0.1254
    0.4100    0.1237
    0.4200    0.1209
    0.4300    0.1171
    0.4400    0.1122
    0.4500    0.1064
    0.4600    0.0996
    0.4700    0.0922
    0.4800    0.0841
    0.4900    0.0755
    0.5000    0.0667
    0.5100    0.0578
    0.5200    0.0490
    0.5300    0.0405
    0.5400    0.0326
    0.5500    0.0253
    0.5600    0.0190
    0.5700    0.0138
    0.5800    0.0097
    0.5900    0.0068
    0.6000    0.0053
    0.6100    0.0051
    0.6200    0.0062
    0.6300    0.0084
    0.6400    0.0117
    0.6500    0.0159
    0.6600    0.0207
    0.6700    0.0260
    0.6800    0.0315
    0.6900    0.0369
    0.7000    0.0419
    0.7100    0.0464
    0.7200    0.0500
    0.7300    0.0526
    0.7400    0.0541
    0.7500    0.0544
    0.7600    0.0534
    0.7700    0.0512
    0.7800    0.0479
    0.7900    0.0437
    0.8000    0.0387
    0.8100    0.0332
    0.8200    0.0275
    0.8300    0.0218
    0.8400    0.0164
    0.8500    0.0116
    0.8600    0.0074
    0.8700    0.0041
    0.8800    0.0018
    0.8900    0.0004
    0.9000         0
    0.9100         0
    0.9200         0
    0.9300         0
    0.9400         0
    0.9500         0
    0.9600         0
    0.9700         0
    0.9800         0
    0.9900         0
    1.0000         0
 /\relax}\relax
\setdashes <3pt>
\Blue{\relax
\plot
  0.0000   0.0000
  0.0150   0.0000
  0.0300   0.0000
  0.0450   0.0000
  0.0600   0.0000
  0.0750   0.0000
  0.0900   0.0000
  0.1050   0.0000
  0.1200   0.0000
  0.1350   0.0010
  0.1500   0.0031
  0.1650   0.0071
  0.1800   0.0118
  0.1950   0.0182
  0.2100   0.0259
  0.2250   0.0344
  0.2400   0.0446
  0.2550   0.0550
  0.2700   0.0660
  0.2850   0.0768
  0.3000   0.0875
  0.3150   0.0968
  0.3300   0.1053
  0.3450   0.1122
  0.3600   0.1170
  0.3750   0.1208
  0.3900   0.1208
  0.4050   0.1196
  0.4200   0.1157
  0.4350   0.1095
  0.4500   0.1020
  0.4650   0.0916
  0.4800   0.0806
  0.4950   0.0684
  0.5100   0.0561
  0.5250   0.0436
  0.5400   0.0327
  0.5550   0.0229
  0.5700   0.0152
  0.5850   0.0103
  0.6000   0.0069
  0.6150   0.0081
  0.6300   0.0106
  0.6450   0.0158
  0.6600   0.0226
  0.6750   0.0302
  0.6900   0.0381
  0.7050   0.0454
  0.7200   0.0513
  0.7350   0.0546
  0.7500   0.0567
  0.7650   0.0542
  0.7800   0.0504
  0.7950   0.0439
  0.8100   0.0359
  0.8250   0.0273
  0.8400   0.0190
  0.8550   0.0115
  0.8700   0.0057
  0.8850   0.0022
  0.9000   0.0000
  0.9150   0.0000
  0.9300   0.0002
  0.9450   0.0009
  0.9600   0.0014
  0.9750   0.0018
  0.9900   0.0012
 /\relax}\relax
\setdots <2pt>
\Green{\relax
\plot
  0.0000   0.0000
  0.0150   0.0000
  0.0300   0.0000
  0.0450   0.0000
  0.0600   0.0000
  0.0750   0.0000
  0.0900   0.0000
  0.1050   0.0000
  0.1200   0.0000
  0.1350   0.0000
  0.1500   0.0000
  0.1650   0.0000
  0.1800   0.0023
  0.1950   0.0094
  0.2100   0.0186
  0.2250   0.0288
  0.2400   0.0412
  0.2550   0.0540
  0.2700   0.0673
  0.2850   0.0801
  0.3000   0.0927
  0.3150   0.1032
  0.3300   0.1126
  0.3450   0.1198
  0.3600   0.1244
  0.3750   0.1277
  0.3900   0.1267
  0.4050   0.1243
  0.4200   0.1191
  0.4350   0.1116
  0.4500   0.1028
  0.4650   0.0912
  0.4800   0.0791
  0.4950   0.0660
  0.5100   0.0530
  0.5250   0.0401
  0.5400   0.0290
  0.5550   0.0191
  0.5700   0.0116
  0.5850   0.0071
  0.6000   0.0042
  0.6150   0.0062
  0.6300   0.0095
  0.6450   0.0156
  0.6600   0.0231
  0.6750   0.0315
  0.6900   0.0399
  0.7050   0.0474
  0.7200   0.0534
  0.7350   0.0566
  0.7500   0.0584
  0.7650   0.0554
  0.7800   0.0510
  0.7950   0.0438
  0.8100   0.0353
  0.8250   0.0262
  0.8400   0.0177
  0.8550   0.0103
  0.8700   0.0048
  0.8850   0.0016
  0.9000   0.0000
  0.9150   0.0002
  0.9300   0.0005
  0.9450   0.0012
  0.9600   0.0015
  0.9750   0.0017
  0.9900   0.0012
 /\relax}\relax
\endpicture
}
%
\yfiglen=13 true in
\setbox\figurethree=\vbox{\hsize=\xfiglen
\beginpicture
\eightrm
  \setcoordinatesystem units <\xfiglen,\yfiglen> 
  \setplotarea x from 0 to 1, y from 0 to 0.16
  \axis bottom shiftedto y=0 ticks short numbered from 0 to 1 by 0.2 /
  \axis left ticks short numbered from 0 to 0.16 by 0.04 /
\put{\copy\figurelegend} [lt] at 0.25 0.15
\footnotesize
\put {$\kappa(x)$} [lt] at 0.02 0.16
\small
\linethickness=0.5pt
\setplotsymbol ({\sevenrm .})
\setquadratic
\setlinear
\setsolid
\Black{\relax  
\plot
         0         0
    0.0100         0
    0.0200         0
    0.0300         0
    0.0400         0
    0.0500         0
    0.0600    0.0015
    0.0700    0.0030
    0.0800    0.0045
    0.0900    0.0060
    0.1000    0.0075
    0.1100    0.0090
    0.1200    0.0105
    0.1300    0.0120
    0.1400    0.0135
    0.1500    0.0150
    0.1600    0.0165
    0.1700    0.0180
    0.1800    0.0195
    0.1900    0.0210
    0.2000    0.0225
    0.2100    0.0240
    0.2200    0.0255
    0.2300    0.0270
    0.2400    0.0285
    0.2500    0.0300
    0.2600    0.0340
    0.2700    0.0380
    0.2800    0.0420
    0.2900    0.0460
    0.3000    0.0500
    0.3100    0.0503
    0.3200    0.0505
    0.3300    0.0508
    0.3400    0.0510
    0.3500    0.0513
    0.3600    0.0515
    0.3700    0.0517
    0.3800    0.0520
    0.3900    0.0523
    0.4000    0.0525
    0.4100    0.0528
    0.4200    0.0530
    0.4300    0.0533
    0.4400    0.0535
    0.4500    0.0538
    0.4600    0.0540
    0.4700    0.0543
    0.4800    0.0545
    0.4900    0.0548
    0.5000    0.0550
    0.5100    0.0570
    0.5200    0.0590
    0.5300    0.0610
    0.5400    0.0630
    0.5500    0.0650
    0.5600    0.0670
    0.5700    0.0690
    0.5800    0.0710
    0.5900    0.0730
    0.6000    0.0750
    0.6100    0.0825
    0.6200    0.0900
    0.6300    0.0975
    0.6400    0.1050
    0.6500    0.1125
    0.6600    0.1200
    0.6700    0.1275
    0.6800    0.1350
    0.6900    0.1425
    0.7000    0.1500
    0.7100    0.1500
    0.7200    0.1500
    0.7300    0.1500
    0.7400    0.1500
    0.7500    0.1500
    0.7600    0.1500
    0.7700    0.1500
    0.7800    0.1500
    0.7900    0.1500
    0.8000    0.1500
    0.8100    0.1450
    0.8200    0.1400
    0.8300    0.1350
    0.8400    0.1300
    0.8500    0.1250
    0.8600    0.1200
    0.8700    0.1150
    0.8800    0.1100
    0.8900    0.1050
    0.9000    0.1000
    0.9100    0.0800
    0.9200    0.0600
    0.9300    0.0400
    0.9400    0.0200
    0.9500         0
    0.9600         0
    0.9700         0
    0.9800         0
    0.9900         0
    1.0000         0
 /\relax}\relax
\setdashes <3pt>
\Blue{\relax
\plot
  0.0000   0.0000
  0.0150   0.0000
  0.0300   0.0001
  0.0450   0.0003
  0.0600   0.0007
  0.0750   0.0012
  0.0900   0.0022
  0.1050   0.0033
  0.1200   0.0048
  0.1350   0.0067
  0.1500   0.0089
  0.1650   0.0116
  0.1800   0.0146
  0.1950   0.0179
  0.2100   0.0215
  0.2250   0.0252
  0.2400   0.0290
  0.2550   0.0328
  0.2700   0.0366
  0.2850   0.0401
  0.3000   0.0434
  0.3150   0.0462
  0.3300   0.0487
  0.3450   0.0508
  0.3600   0.0523
  0.3750   0.0536
  0.3900   0.0542
  0.4050   0.0546
  0.4200   0.0546
  0.4350   0.0545
  0.4500   0.0543
  0.4650   0.0543
  0.4800   0.0544
  0.4950   0.0549
  0.5100   0.0561
  0.5250   0.0576
  0.5400   0.0607
  0.5550   0.0643
  0.5700   0.0692
  0.5850   0.0754
  0.6000   0.0823
  0.6150   0.0909
  0.6300   0.0999
  0.6450   0.1096
  0.6600   0.1191
  0.6750   0.1286
  0.6900   0.1366
  0.7050   0.1437
  0.7200   0.1489
  0.7350   0.1520
  0.7500   0.1540
  0.7650   0.1529
  0.7800   0.1511
  0.7950   0.1478
  0.8100   0.1437
  0.8250   0.1391
  0.8400   0.1330
  0.8550   0.1259
  0.8700   0.1166
  0.8850   0.1031
  0.9000   0.0876
  0.9150   0.0654
  0.9300   0.0432
  0.9450   0.0208
  0.9600   0.0064
  0.9750   0.0000
  0.9900   0.0000
 /\relax}\relax
\setdots <2pt>
\Green{\relax
\plot
  0.0000   0.0000
  0.0150   0.0000
  0.0300   0.0001
  0.0450   0.0003
  0.0600   0.0007
  0.0750   0.0011
  0.0900   0.0020
  0.1050   0.0031
  0.1200   0.0047
  0.1350   0.0066
  0.1500   0.0088
  0.1650   0.0118
  0.1800   0.0149
  0.1950   0.0185
  0.2100   0.0224
  0.2250   0.0264
  0.2400   0.0305
  0.2550   0.0346
  0.2700   0.0386
  0.2850   0.0422
  0.3000   0.0456
  0.3150   0.0483
  0.3300   0.0506
  0.3450   0.0524
  0.3600   0.0536
  0.3750   0.0545
  0.3900   0.0547
  0.4050   0.0547
  0.4200   0.0545
  0.4350   0.0541
  0.4500   0.0537
  0.4650   0.0536
  0.4800   0.0536
  0.4950   0.0541
  0.5100   0.0552
  0.5250   0.0567
  0.5400   0.0597
  0.5550   0.0632
  0.5700   0.0680
  0.5850   0.0742
  0.6000   0.0810
  0.6150   0.0898
  0.6300   0.0990
  0.6450   0.1090
  0.6600   0.1191
  0.6750   0.1291
  0.6900   0.1376
  0.7050   0.1451
  0.7200   0.1506
  0.7350   0.1535
  0.7500   0.1551
  0.7650   0.1532
  0.7800   0.1506
  0.7950   0.1466
  0.8100   0.1422
  0.8250   0.1375
  0.8400   0.1323
  0.8550   0.1262
  0.8700   0.1183
  0.8850   0.1056
  0.9000   0.0905
  0.9150   0.0664
  0.9300   0.0420
  0.9450   0.0170
  0.9600   0.0024
  0.9750   0.0000
  0.9900   0.0000
 /\relax}\relax
\setdashes <4pt>
\setsolid
\Red{\relax
\plot
  0.0000   0.0000
  0.0150   0.0000
  0.0300   0.0001
  0.0450   0.0002
  0.0600   0.0006
  0.0750   0.0010
  0.0900   0.0018
  0.1050   0.0029
  0.1200   0.0044
  0.1350   0.0063
  0.1500   0.0086
  0.1650   0.0116
  0.1800   0.0148
  0.1950   0.0186
  0.2100   0.0226
  0.2250   0.0269
  0.2400   0.0312
  0.2550   0.0354
  0.2700   0.0395
  0.2850   0.0431
  0.3000   0.0465
  0.3150   0.0490
  0.3300   0.0512
  0.3450   0.0527
  0.3600   0.0536
  0.3750   0.0542
  0.3900   0.0541
  0.4050   0.0540
  0.4200   0.0536
  0.4350   0.0532
  0.4500   0.0528
  0.4650   0.0528
  0.4800   0.0531
  0.4950   0.0538
  0.5100   0.0551
  0.5250   0.0568
  0.5400   0.0599
  0.5550   0.0635
  0.5700   0.0683
  0.5850   0.0743
  0.6000   0.0811
  0.6150   0.0898
  0.6300   0.0990
  0.6450   0.1090
  0.6600   0.1192
  0.6750   0.1294
  0.6900   0.1381
  0.7050   0.1457
  0.7200   0.1512
  0.7350   0.1540
  0.7500   0.1554
  0.7650   0.1532
  0.7800   0.1503
  0.7950   0.1459
  0.8100   0.1413
  0.8250   0.1366
  0.8400   0.1318
  0.8550   0.1262
  0.8700   0.1190
  0.8850   0.1068
  0.9000   0.0920
  0.9150   0.0670
  0.9300   0.0415
  0.9450   0.0151
  0.9600   0.0004
  0.9750   0.0000
  0.9900   0.0000
 /\relax}\relax
\endpicture
}
%
\setbox\figurefour=\vbox{\hsize=\xfiglen
\beginpicture
\eightrm
  \setcoordinatesystem units <\xfiglen,\yfiglen> 
  \setplotarea x from 0 to 1, y from 0 to 0.16
  \axis bottom shiftedto y=0 ticks short numbered from 0 to 1 by 0.2 /
  \axis left ticks short numbered from 0 to 0.16 by 0.04 /
%
\footnotesize
\put {$\kappa(x)$} [lt] at 0.02 0.16
\small
\linethickness=1pt
\setplotsymbol ({\sevenrm .})
\setquadratic
\setlinear
\setsolid
\Black{\relax  
\plot
         0         0
    0.0100         0
    0.0200         0
    0.0300         0
    0.0400         0
    0.0500         0
    0.0600    0.0015
    0.0700    0.0030
    0.0800    0.0045
    0.0900    0.0060
    0.1000    0.0075
    0.1100    0.0090
    0.1200    0.0105
    0.1300    0.0120
    0.1400    0.0135
    0.1500    0.0150
    0.1600    0.0165
    0.1700    0.0180
    0.1800    0.0195
    0.1900    0.0210
    0.2000    0.0225
    0.2100    0.0240
    0.2200    0.0255
    0.2300    0.0270
    0.2400    0.0285
    0.2500    0.0300
    0.2600    0.0340
    0.2700    0.0380
    0.2800    0.0420
    0.2900    0.0460
    0.3000    0.0500
    0.3100    0.0503
    0.3200    0.0505
    0.3300    0.0508
    0.3400    0.0510
    0.3500    0.0513
    0.3600    0.0515
    0.3700    0.0517
    0.3800    0.0520
    0.3900    0.0523
    0.4000    0.0525
    0.4100    0.0528
    0.4200    0.0530
    0.4300    0.0533
    0.4400    0.0535
    0.4500    0.0538
    0.4600    0.0540
    0.4700    0.0543
    0.4800    0.0545
    0.4900    0.0548
    0.5000    0.0550
    0.5100    0.0570
    0.5200    0.0590
    0.5300    0.0610
    0.5400    0.0630
    0.5500    0.0650
    0.5600    0.0670
    0.5700    0.0690
    0.5800    0.0710
    0.5900    0.0730
    0.6000    0.0750
    0.6100    0.0825
    0.6200    0.0900
    0.6300    0.0975
    0.6400    0.1050
    0.6500    0.1125
    0.6600    0.1200
    0.6700    0.1275
    0.6800    0.1350
    0.6900    0.1425
    0.7000    0.1500
    0.7100    0.1500
    0.7200    0.1500
    0.7300    0.1500
    0.7400    0.1500
    0.7500    0.1500
    0.7600    0.1500
    0.7700    0.1500
    0.7800    0.1500
    0.7900    0.1500
    0.8000    0.1500
    0.8100    0.1450
    0.8200    0.1400
    0.8300    0.1350
    0.8400    0.1300
    0.8500    0.1250
    0.8600    0.1200
    0.8700    0.1150
    0.8800    0.1100
    0.8900    0.1050
    0.9000    0.1000
    0.9100    0.0800
    0.9200    0.0600
    0.9300    0.0400
    0.9400    0.0200
    0.9500         0
    0.9600         0
    0.9700         0
    0.9800         0
    0.9900         0
    1.0000         0
 /\relax}\relax
\setdashes <3pt>
\Blue{\relax
\plot
  0.0000   0.0000
  0.0150   0.0000
  0.0300   0.0000
  0.0450   0.0000
  0.0600   0.0000
  0.0750   0.0000
  0.0900   0.0000
  0.1050   0.0000
  0.1200   0.0000
  0.1350   0.0000
  0.1500   0.0000
  0.1650   0.0012
  0.1800   0.0038
  0.1950   0.0074
  0.2100   0.0119
  0.2250   0.0167
  0.2400   0.0223
  0.2550   0.0279
  0.2700   0.0335
  0.2850   0.0387
  0.3000   0.0437
  0.3150   0.0474
  0.3300   0.0506
  0.3450   0.0529
  0.3600   0.0542
  0.3750   0.0551
  0.3900   0.0548
  0.4050   0.0543
  0.4200   0.0533
  0.4350   0.0523
  0.4500   0.0512
  0.4650   0.0505
  0.4800   0.0501
  0.4950   0.0504
  0.5100   0.0516
  0.5250   0.0532
  0.5400   0.0567
  0.5550   0.0610
  0.5700   0.0666
  0.5850   0.0737
  0.6000   0.0815
  0.6150   0.0910
  0.6300   0.1009
  0.6450   0.1113
  0.6600   0.1214
  0.6750   0.1315
  0.6900   0.1397
  0.7050   0.1470
  0.7200   0.1523
  0.7350   0.1553
  0.7500   0.1571
  0.7650   0.1558
  0.7800   0.1537
  0.7950   0.1500
  0.8100   0.1455
  0.8250   0.1405
  0.8400   0.1341
  0.8550   0.1266
  0.8700   0.1171
  0.8850   0.1035
  0.9000   0.0880
  0.9150   0.0659
  0.9300   0.0438
  0.9450   0.0217
  0.9600   0.0076
  0.9750   0.0000
  0.9900   0.0000
 /\relax}\relax
\setdots <2pt>
\Green{\relax
\plot
  0.0000   0.0000
  0.0150   0.0000
  0.0300   0.0000
  0.0450   0.0000
  0.0600   0.0000
  0.0750   0.0000
  0.0900   0.0000
  0.1050   0.0000
  0.1200   0.0000
  0.1350   0.0000
  0.1500   0.0000
  0.1650   0.0000
  0.1800   0.0000
  0.1950   0.0014
  0.2100   0.0073
  0.2250   0.0140
  0.2400   0.0217
  0.2550   0.0294
  0.2700   0.0369
  0.2850   0.0436
  0.3000   0.0498
  0.3150   0.0538
  0.3300   0.0570
  0.3450   0.0586
  0.3600   0.0588
  0.3750   0.0583
  0.3900   0.0564
  0.4050   0.0544
  0.4200   0.0521
  0.4350   0.0500
  0.4500   0.0481
  0.4650   0.0472
  0.4800   0.0468
  0.4950   0.0474
  0.5100   0.0491
  0.5250   0.0514
  0.5400   0.0556
  0.5550   0.0604
  0.5700   0.0665
  0.5850   0.0739
  0.6000   0.0820
  0.6150   0.0917
  0.6300   0.1016
  0.6450   0.1121
  0.6600   0.1224
  0.6750   0.1325
  0.6900   0.1409
  0.7050   0.1482
  0.7200   0.1534
  0.7350   0.1560
  0.7500   0.1574
  0.7650   0.1554
  0.7800   0.1526
  0.7950   0.1483
  0.8100   0.1436
  0.8250   0.1386
  0.8400   0.1331
  0.8550   0.1268
  0.8700   0.1186
  0.8850   0.1059
  0.9000   0.0908
  0.9150   0.0668
  0.9300   0.0426
  0.9450   0.0179
  0.9600   0.0040
  0.9750   0.0000
  0.9900   0.0002
 /\relax}\relax
\endpicture
}
%
\setbox\figurefive=\vbox{\hsize=\xfiglen
\beginpicture
\eightrm
  \setcoordinatesystem units <\xfiglen,0.95\yfiglen> 
  \setplotarea x from 0 to 1, y from 0 to 0.16
  \axis bottom shiftedto y=0 ticks short numbered from 0 to 1 by 0.2 /
  \axis left ticks short numbered from 0 to 0.16 by 0.04 /
\footnotesize
\put {$\kappa(x)$} [lt] at 0.02 0.16
\small
\put{\box\figurelegendthree} [lt] at 0.1 0.14
\linethickness=1.5pt
\setquadratic
\setlinear
\setsolid
\Black{\relax  
\plot
         0         0
    0.0100         0
    0.0200         0
    0.0300         0
    0.0400         0
    0.0500         0
    0.0600    0.0015
    0.0700    0.0030
    0.0800    0.0045
    0.0900    0.0060
    0.1000    0.0075
    0.1100    0.0090
    0.1200    0.0105
    0.1300    0.0120
    0.1400    0.0135
    0.1500    0.0150
    0.1600    0.0165
    0.1700    0.0180
    0.1800    0.0195
    0.1900    0.0210
    0.2000    0.0225
    0.2100    0.0240
    0.2200    0.0255
    0.2300    0.0270
    0.2400    0.0285
    0.2500    0.0300
    0.2600    0.0340
    0.2700    0.0380
    0.2800    0.0420
    0.2900    0.0460
    0.3000    0.0500
    0.3100    0.0503
    0.3200    0.0505
    0.3300    0.0508
    0.3400    0.0510
    0.3500    0.0513
    0.3600    0.0515
    0.3700    0.0517
    0.3800    0.0520
    0.3900    0.0523
    0.4000    0.0525
    0.4100    0.0528
    0.4200    0.0530
    0.4300    0.0533
    0.4400    0.0535
    0.4500    0.0538
    0.4600    0.0540
    0.4700    0.0543
    0.4800    0.0545
    0.4900    0.0548
    0.5000    0.0550
    0.5100    0.0570
    0.5200    0.0590
    0.5300    0.0610
    0.5400    0.0630
    0.5500    0.0650
    0.5600    0.0670
    0.5700    0.0690
    0.5800    0.0710
    0.5900    0.0730
    0.6000    0.0750
    0.6100    0.0825
    0.6200    0.0900
    0.6300    0.0975
    0.6400    0.1050
    0.6500    0.1125
    0.6600    0.1200
    0.6700    0.1275
    0.6800    0.1350
    0.6900    0.1425
    0.7000    0.1500
    0.7100    0.1500
    0.7200    0.1500
    0.7300    0.1500
    0.7400    0.1500
    0.7500    0.1500
    0.7600    0.1500
    0.7700    0.1500
    0.7800    0.1500
    0.7900    0.1500
    0.8000    0.1500
    0.8100    0.1450
    0.8200    0.1400
    0.8300    0.1350
    0.8400    0.1300
    0.8500    0.1250
    0.8600    0.1200
    0.8700    0.1150
    0.8800    0.1100
    0.8900    0.1050
    0.9000    0.1000
    0.9100    0.0800
    0.9200    0.0600
    0.9300    0.0400
    0.9400    0.0200
    0.9500         0
    0.9600         0
    0.9700         0
    0.9800         0
    0.9900         0
    1.0000         0
 /\relax}\relax
%
\linethickness=0.4pt
\setlinear
\setdots <2pt>
\Blue{\relax  
\plot  
  0.0000   0.0000
  0.0150   0.0000
  0.0300   0.0000
  0.0450   0.0001
  0.0600   0.0002
  0.0750   0.0003
  0.0900   0.0005
  0.1050   0.0007
  0.1200   0.0010
  0.1350   0.0015
  0.1500   0.0019
  0.1650   0.0026
  0.1800   0.0033
  0.1950   0.0041
  0.2100   0.0050
  0.2250   0.0060
  0.2400   0.0073
  0.2550   0.0086
  0.2700   0.0100
  0.2850   0.0116
  0.3000   0.0133
  0.3150   0.0152
  0.3300   0.0171
  0.3450   0.0192
  0.3600   0.0215
  0.3750   0.0238
  0.3900   0.0264
  0.4050   0.0290
  0.4200   0.0318
  0.4350   0.0347
  0.4500   0.0376
  0.4650   0.0408
  0.4800   0.0440
  0.4950   0.0473
  0.5100   0.0507
  0.5250   0.0542
  0.5400   0.0578
  0.5550   0.0614
  0.5700   0.0651
  0.5850   0.0689
  0.6000   0.0727
  0.6150   0.0765
  0.6300   0.0804
  0.6450   0.0842
  0.6600   0.0881
  0.6750   0.0920
  0.6900   0.0958
  0.7050   0.0996
  0.7200   0.1033
  0.7350   0.1070
  0.7500   0.1107
  0.7650   0.1142
  0.7800   0.1177
  0.7950   0.1210
  0.8100   0.1242
  0.8250   0.1274
  0.8400   0.1304
  0.8550   0.1332
  0.8700   0.1359
  0.8850   0.1385
  0.9000   0.1410
  0.9150   0.1432
  0.9300   0.1453
  0.9450   0.1473
  0.9600   0.1504
  0.9750   0.1539
  0.9900   0.1056
 /\relax}\relax
\setsolid
\LimeGreen{\relax  
\plot  
  0.0000   0.0000
  0.0150   0.0000
  0.0300   0.0000
  0.0450   0.0001
  0.0600   0.0002
  0.0750   0.0003
  0.0900   0.0006
  0.1050   0.0009
  0.1200   0.0012
  0.1350   0.0017
  0.1500   0.0023
  0.1650   0.0030
  0.1800   0.0039
  0.1950   0.0048
  0.2100   0.0059
  0.2250   0.0071
  0.2400   0.0085
  0.2550   0.0100
  0.2700   0.0117
  0.2850   0.0135
  0.3000   0.0154
  0.3150   0.0175
  0.3300   0.0197
  0.3450   0.0220
  0.3600   0.0245
  0.3750   0.0271
  0.3900   0.0298
  0.4050   0.0326
  0.4200   0.0355
  0.4350   0.0386
  0.4500   0.0416
  0.4650   0.0449
  0.4800   0.0481
  0.4950   0.0514
  0.5100   0.0548
  0.5250   0.0582
  0.5400   0.0616
  0.5550   0.0651
  0.5700   0.0685
  0.5850   0.0720
  0.6000   0.0754
  0.6150   0.0787
  0.6300   0.0820
  0.6450   0.0852
  0.6600   0.0883
  0.6750   0.0913
  0.6900   0.0941
  0.7050   0.0967
  0.7200   0.0992
  0.7350   0.1013
  0.7500   0.1033
  0.7650   0.1049
  0.7800   0.1062
  0.7950   0.1073
  0.8100   0.1080
  0.8250   0.1085
  0.8400   0.1086
  0.8550   0.1085
  0.8700   0.1081
  0.8850   0.1075
  0.9000   0.1068
  0.9150   0.1060
  0.9300   0.1052
  0.9450   0.1044
  0.9600   0.1047
  0.9750   0.1056
  0.9900   0.0720
 /\relax}\relax
%
\Orange{\relax
\plot  
  0.0000   0.0000
  0.0150   0.0000
  0.0300   0.0001
  0.0450   0.0001
  0.0600   0.0003
  0.0750   0.0005
  0.0900   0.0010
  0.1050   0.0015
  0.1200   0.0021
  0.1350   0.0030
  0.1500   0.0039
  0.1650   0.0052
  0.1800   0.0065
  0.1950   0.0081
  0.2100   0.0099
  0.2250   0.0118
  0.2400   0.0141
  0.2550   0.0164
  0.2700   0.0190
  0.2850   0.0217
  0.3000   0.0246
  0.3150   0.0277
  0.3300   0.0309
  0.3450   0.0342
  0.3600   0.0376
  0.3750   0.0411
  0.3900   0.0447
  0.4050   0.0484
  0.4200   0.0521
  0.4350   0.0559
  0.4500   0.0596
  0.4650   0.0635
  0.4800   0.0674
  0.4950   0.0713
  0.5100   0.0753
  0.5250   0.0793
  0.5400   0.0835
  0.5550   0.0877
  0.5700   0.0920
  0.5850   0.0964
  0.6000   0.1009
  0.6150   0.1054
  0.6300   0.1098
  0.6450   0.1143
  0.6600   0.1185
  0.6750   0.1227
  0.6900   0.1262
  0.7050   0.1294
  0.7200   0.1320
  0.7350   0.1337
  0.7500   0.1349
  0.7650   0.1344
  0.7800   0.1332
  0.7950   0.1307
  0.8100   0.1267
  0.8250   0.1219
  0.8400   0.1149
  0.8550   0.1073
  0.8700   0.0984
  0.8850   0.0884
  0.9000   0.0780
  0.9150   0.0670
  0.9300   0.0563
  0.9450   0.0461
  0.9600   0.0376
  0.9750   0.0300
  0.9900   0.0189
 /\relax}\relax
%
\setdashes <3pt>
\Brown{\relax
\plot  
  0.0000   0.0000
  0.0150   0.0000
  0.0300   0.0001
  0.0450   0.0002
  0.0600   0.0004
  0.0750   0.0006
  0.0900   0.0011
  0.1050   0.0016
  0.1200   0.0023
  0.1350   0.0033
  0.1500   0.0043
  0.1650   0.0057
  0.1800   0.0071
  0.1950   0.0089
  0.2100   0.0108
  0.2250   0.0128
  0.2400   0.0152
  0.2550   0.0177
  0.2700   0.0203
  0.2850   0.0231
  0.3000   0.0260
  0.3150   0.0291
  0.3300   0.0322
  0.3450   0.0354
  0.3600   0.0386
  0.3750   0.0418
  0.3900   0.0451
  0.4050   0.0483
  0.4200   0.0515
  0.4350   0.0547
  0.4500   0.0579
  0.4650   0.0611
  0.4800   0.0643
  0.4950   0.0676
  0.5100   0.0711
  0.5250   0.0746
  0.5400   0.0784
  0.5550   0.0824
  0.5700   0.0867
  0.5850   0.0913
  0.6000   0.0960
  0.6150   0.1013
  0.6300   0.1066
  0.6450   0.1122
  0.6600   0.1178
  0.6750   0.1234
  0.6900   0.1287
  0.7050   0.1336
  0.7200   0.1378
  0.7350   0.1411
  0.7500   0.1438
  0.7650   0.1443
  0.7800   0.1440
  0.7950   0.1419
  0.8100   0.1377
  0.8250   0.1326
  0.8400   0.1244
  0.8550   0.1153
  0.8700   0.1042
  0.8850   0.0915
  0.9000   0.0781
  0.9150   0.0636
  0.9300   0.0493
  0.9450   0.0354
  0.9600   0.0237
  0.9750   0.0131
  0.9900   0.0068
 /\relax}\relax
%
\setsolid
\BrickRed{\relax
\plot  
  0.0000   0.0000
  0.0150   0.0000
  0.0300   0.0001
  0.0450   0.0002
  0.0600   0.0005
  0.0750   0.0008
  0.0900   0.0014
  0.1050   0.0021
  0.1200   0.0031
  0.1350   0.0043
  0.1500   0.0056
  0.1650   0.0073
  0.1800   0.0092
  0.1950   0.0113
  0.2100   0.0137
  0.2250   0.0162
  0.2400   0.0190
  0.2550   0.0219
  0.2700   0.0250
  0.2850   0.0282
  0.3000   0.0313
  0.3150   0.0345
  0.3300   0.0377
  0.3450   0.0408
  0.3600   0.0438
  0.3750   0.0466
  0.3900   0.0492
  0.4050   0.0517
  0.4200   0.0540
  0.4350   0.0561
  0.4500   0.0581
  0.4650   0.0599
  0.4800   0.0618
  0.4950   0.0637
  0.5100   0.0659
  0.5250   0.0681
  0.5400   0.0709
  0.5550   0.0741
  0.5700   0.0778
  0.5850   0.0823
  0.6000   0.0872
  0.6150   0.0932
  0.6300   0.0995
  0.6450   0.1064
  0.6600   0.1137
  0.6750   0.1212
  0.6900   0.1285
  0.7050   0.1355
  0.7200   0.1419
  0.7350   0.1470
  0.7500   0.1514
  0.7650   0.1532
  0.7800   0.1539
  0.7950   0.1524
  0.8100   0.1484
  0.8250   0.1433
  0.8400   0.1342
  0.8550   0.1240
  0.8700   0.1112
  0.8850   0.0963
  0.9000   0.0803
  0.9150   0.0623
  0.9300   0.0445
  0.9450   0.0270
  0.9600   0.0122
  0.9750  -0.0014
  0.9900  -0.0036
 /\relax}\relax
\endpicture
}
\setbox\figuresix=\vbox{\hsize=\xfiglen
\beginpicture
\eightrm
  \setcoordinatesystem units <\xfiglen, 0.95\yfiglen> 
  \setplotarea x from 0 to 1, y from 0 to 0.16
  \axis bottom shiftedto y=0 ticks short numbered from 0 to 1 by 0.2 /
  \axis left ticks short numbered from 0 to 0.16 by 0.04 /
\footnotesize
\put {$\kappa(x)$} [lt] at 0.02 0.16
\small
\put{\box\figurelegendfour} [lt] at 0.1 0.14
\linethickness=1.5pt
\setquadratic
\setlinear
\setsolid
\Black{\relax  
\plot
         0         0
    0.0100         0
    0.0200         0
    0.0300         0
    0.0400         0
    0.0500         0
    0.0600    0.0015
    0.0700    0.0030
    0.0800    0.0045
    0.0900    0.0060
    0.1000    0.0075
    0.1100    0.0090
    0.1200    0.0105
    0.1300    0.0120
    0.1400    0.0135
    0.1500    0.0150
    0.1600    0.0165
    0.1700    0.0180
    0.1800    0.0195
    0.1900    0.0210
    0.2000    0.0225
    0.2100    0.0240
    0.2200    0.0255
    0.2300    0.0270
    0.2400    0.0285
    0.2500    0.0300
    0.2600    0.0340
    0.2700    0.0380
    0.2800    0.0420
    0.2900    0.0460
    0.3000    0.0500
    0.3100    0.0503
    0.3200    0.0505
    0.3300    0.0508
    0.3400    0.0510
    0.3500    0.0513
    0.3600    0.0515
    0.3700    0.0517
    0.3800    0.0520
    0.3900    0.0523
    0.4000    0.0525
    0.4100    0.0528
    0.4200    0.0530
    0.4300    0.0533
    0.4400    0.0535
    0.4500    0.0538
    0.4600    0.0540
    0.4700    0.0543
    0.4800    0.0545
    0.4900    0.0548
    0.5000    0.0550
    0.5100    0.0570
    0.5200    0.0590
    0.5300    0.0610
    0.5400    0.0630
    0.5500    0.0650
    0.5600    0.0670
    0.5700    0.0690
    0.5800    0.0710
    0.5900    0.0730
    0.6000    0.0750
    0.6100    0.0825
    0.6200    0.0900
    0.6300    0.0975
    0.6400    0.1050
    0.6500    0.1125
    0.6600    0.1200
    0.6700    0.1275
    0.6800    0.1350
    0.6900    0.1425
    0.7000    0.1500
    0.7100    0.1500
    0.7200    0.1500
    0.7300    0.1500
    0.7400    0.1500
    0.7500    0.1500
    0.7600    0.1500
    0.7700    0.1500
    0.7800    0.1500
    0.7900    0.1500
    0.8000    0.1500
    0.8100    0.1450
    0.8200    0.1400
    0.8300    0.1350
    0.8400    0.1300
    0.8500    0.1250
    0.8600    0.1200
    0.8700    0.1150
    0.8800    0.1100
    0.8900    0.1050
    0.9000    0.1000
    0.9100    0.0800
    0.9200    0.0600
    0.9300    0.0400
    0.9400    0.0200
    0.9500         0
    0.9600         0
    0.9700         0
    0.9800         0
    0.9900         0
    1.0000         0
 /\relax}\relax
\linethickness=0.6pt
\setsolid
\Blue{\relax
\plot  
  0.0000   0.0000
  0.0150   0.0000
  0.0300   0.0001
  0.0450   0.0003
  0.0600   0.0006
  0.0750   0.0010
  0.0900   0.0018
  0.1050   0.0027
  0.1200   0.0038
  0.1350   0.0053
  0.1500   0.0069
  0.1650   0.0091
  0.1800   0.0113
  0.1950   0.0139
  0.2100   0.0167
  0.2250   0.0197
  0.2400   0.0230
  0.2550   0.0264
  0.2700   0.0298
  0.2850   0.0333
  0.3000   0.0368
  0.3150   0.0401
  0.3300   0.0433
  0.3450   0.0463
  0.3600   0.0489
  0.3750   0.0515
  0.3900   0.0534
  0.4050   0.0552
  0.4200   0.0566
  0.4350   0.0578
  0.4500   0.0587
  0.4650   0.0595
  0.4800   0.0603
  0.4950   0.0611
  0.5100   0.0623
  0.5250   0.0636
  0.5400   0.0659
  0.5550   0.0687
  0.5700   0.0723
  0.5850   0.0770
  0.6000   0.0823
  0.6150   0.0891
  0.6300   0.0963
  0.6450   0.1044
  0.6600   0.1128
  0.6750   0.1215
  0.6900   0.1299
  0.7050   0.1377
  0.7200   0.1447
  0.7350   0.1501
  0.7500   0.1547
  0.7650   0.1562
  0.7800   0.1566
  0.7950   0.1546
  0.8100   0.1503
  0.8250   0.1447
  0.8400   0.1356
  0.8550   0.1253
  0.8700   0.1126
  0.8850   0.0975
  0.9000   0.0812
  0.9150   0.0624
  0.9300   0.0437
  0.9450   0.0251
  0.9600   0.0095
  0.9750  -0.0046
  0.9900  -0.0057
 /\relax}\relax
%
\Red{\relax
\plot  
  0.0000   0.0000
  0.0150   0.0000
  0.0300   0.0001
  0.0450   0.0003
  0.0600   0.0007
  0.0750   0.0012
  0.0900   0.0021
  0.1050   0.0032
  0.1200   0.0045
  0.1350   0.0063
  0.1500   0.0082
  0.1650   0.0107
  0.1800   0.0133
  0.1950   0.0163
  0.2100   0.0195
  0.2250   0.0229
  0.2400   0.0265
  0.2550   0.0302
  0.2700   0.0339
  0.2850   0.0375
  0.3000   0.0411
  0.3150   0.0444
  0.3300   0.0474
  0.3450   0.0501
  0.3600   0.0524
  0.3750   0.0544
  0.3900   0.0556
  0.4050   0.0566
  0.4200   0.0572
  0.4350   0.0574
  0.4500   0.0574
  0.4650   0.0573
  0.4800   0.0573
  0.4950   0.0575
  0.5100   0.0582
  0.5250   0.0592
  0.5400   0.0616
  0.5550   0.0646
  0.5700   0.0688
  0.5850   0.0743
  0.6000   0.0805
  0.6150   0.0885
  0.6300   0.0969
  0.6450   0.1061
  0.6600   0.1155
  0.6750   0.1250
  0.6900   0.1336
  0.7050   0.1415
  0.7200   0.1480
  0.7350   0.1525
  0.7500   0.1561
  0.7650   0.1563
  0.7800   0.1555
  0.7950   0.1526
  0.8100   0.1478
  0.8250   0.1421
  0.8400   0.1336
  0.8550   0.1242
  0.8700   0.1126
  0.8850   0.0983
  0.9000   0.0827
  0.9150   0.0636
  0.9300   0.0443
  0.9450   0.0249
  0.9600   0.0090
  0.9750  -0.0050
  0.9900  -0.0056
 /\relax}\relax
%
\Goldenrod{\relax
\plot  
  0.0000   0.0000
  0.0150   0.0000
  0.0300   0.0002
  0.0450   0.0004
  0.0600   0.0008
  0.0750   0.0014
  0.0900   0.0025
  0.1050   0.0037
  0.1200   0.0053
  0.1350   0.0073
  0.1500   0.0096
  0.1650   0.0124
  0.1800   0.0153
  0.1950   0.0187
  0.2100   0.0222
  0.2250   0.0259
  0.2400   0.0298
  0.2550   0.0336
  0.2700   0.0374
  0.2850   0.0410
  0.3000   0.0445
  0.3150   0.0475
  0.3300   0.0502
  0.3450   0.0524
  0.3600   0.0540
  0.3750   0.0553
  0.3900   0.0558
  0.4050   0.0560
  0.4200   0.0558
  0.4350   0.0554
  0.4500   0.0548
  0.4650   0.0542
  0.4800   0.0539
  0.4950   0.0540
  0.5100   0.0549
  0.5250   0.0561
  0.5400   0.0591
  0.5550   0.0628
  0.5700   0.0679
  0.5850   0.0743
  0.6000   0.0814
  0.6150   0.0904
  0.6300   0.0996
  0.6450   0.1094
  0.6600   0.1192
  0.6750   0.1288
  0.6900   0.1369
  0.7050   0.1440
  0.7200   0.1494
  0.7350   0.1525
  0.7500   0.1545
  0.7650   0.1535
  0.7800   0.1516
  0.7950   0.1483
  0.8100   0.1440
  0.8250   0.1392
  0.8400   0.1327
  0.8550   0.1252
  0.8700   0.1156
  0.8850   0.1021
  0.9000   0.0868
  0.9150   0.0656
  0.9300   0.0443
  0.9450   0.0225
  0.9600   0.0068
  0.9750  -0.0060
  0.9900  -0.0049
 /\relax}\relax
\endpicture
}
\setbox\figureseven=\vbox{\hsize=\xfiglen
\beginpicture
\eightrm
  \setcoordinatesystem units <\xfiglen,0.8\yfiglen> 
  \setplotarea x from 0 to 1, y from 0 to 0.21
  \axis bottom shiftedto y=0 ticks short numbered from 0 to 1 by 0.2 /
  \axis left ticks short numbered from 0 to 0.2 by 0.1 /
\footnotesize
\put {$\kappa(x)$} [lt] at 0.02 0.21
\small
\linethickness=1.5pt
\putrule from 0 0 to 0.1 0
\putrule from 0.1 0.01 to 0.2 0.01
\putrule from 0.20 0.03 to 0.30 0.03
\putrule from 0.30 0.14 to 0.45 0.14
\putrule from 0.45 0.11 to 0.60 0.11
\putrule from 0.60 0.19 to 0.80 0.19
\putrule from 0.80 0.11 to 0.90 0.11
\putrule from 0.90 0.08 to 1.00 0.08
\linethickness=0.6pt
\setsolid
\putrule from 0.10 0.00 to 0.10 0.01
\putrule from 0.20 0.01 to 0.20 0.03
\putrule from 0.30 0.03 to 0.30 0.14
\putrule from 0.45 0.14 to 0.45 0.11
\putrule from 0.60 0.11 to 0.60 0.19
\putrule from 0.80 0.19 to 0.80 0.11
\putrule from 0.90 0.11 to 0.90 0.08
%
\setsolid
\setquadratic
\setlinear
\linethickness=0.6pt
\Green{\relax\putrule from 0.5 0.045 to 0.6 0.045 }\relax
\footnotesize
\put {Haar Basis} [l] at 0.65 0.045
\Red{\relax\putrule from 0.5 0.03 to 0.6 0.03 }\relax
\put {chapeau Basis} [l] at 0.65 0.03
%
\Red{\relax   
\plot
         0    0.0000
    0.0100    0.0000
    0.0200    0.0002
    0.0300    0.0005
    0.0400    0.0012
    0.0500    0.0021
    0.0600    0.0035
    0.0700    0.0051
    0.0800    0.0070
    0.0900    0.0091
    0.1000    0.0113
    0.1100    0.0134
    0.1200    0.0154
    0.1300    0.0170
    0.1400    0.0183
    0.1500    0.0193
    0.1600    0.0199
    0.1700    0.0202
    0.1800    0.0204
    0.1900    0.0207
    0.2000    0.0213
    0.2100    0.0224
    0.2200    0.0242
    0.2300    0.0270
    0.2400    0.0309
    0.2500    0.0360
    0.2600    0.0423
    0.2700    0.0497
    0.2800    0.0583
    0.2900    0.0678
    0.3000    0.0782
    0.3100    0.0893
    0.3200    0.1008
    0.3300    0.1123
    0.3400    0.1235
    0.3500    0.1340
    0.3600    0.1434
    0.3700    0.1511
    0.3800    0.1568
    0.3900    0.1599
    0.4000    0.1602
    0.4100    0.1577
    0.4200    0.1523
    0.4300    0.1444
    0.4400    0.1347
    0.4500    0.1241
    0.4600    0.1134
    0.4700    0.1038
    0.4800    0.0962
    0.4900    0.0914
    0.5000    0.0898
    0.5100    0.0914
    0.5200    0.0960
    0.5300    0.1030
    0.5400    0.1115
    0.5500    0.1207
    0.5600    0.1297
    0.5700    0.1380
    0.5800    0.1450
    0.5900    0.1508
    0.6000    0.1554
    0.6100    0.1591
    0.6200    0.1625
    0.6300    0.1662
    0.6400    0.1707
    0.6500    0.1761
    0.6600    0.1823
    0.6700    0.1888
    0.6800    0.1949
    0.6900    0.2000
    0.7000    0.2036
    0.7100    0.2051
    0.7200    0.2046
    0.7300    0.2019
    0.7400    0.1973
    0.7500    0.1910
    0.7600    0.1836
    0.7700    0.1752
    0.7800    0.1664
    0.7900    0.1570
    0.8000    0.1473
    0.8100    0.1374
    0.8200    0.1280
    0.8300    0.1198
    0.8400    0.1133
    0.8500    0.1089
    0.8600    0.1065
    0.8700    0.1053
    0.8800    0.1041
    0.8900    0.1014
    0.9000    0.0962
    0.9100    0.0889
    0.9200    0.0818
    0.9300    0.0777
    0.9400    0.0782
    0.9500    0.0818
    0.9600    0.0833
    0.9700    0.0778
    0.9800    0.0702
    0.9900    0.0949
    1.0000    0.0647
 /\relax}\relax
\setlinear
\Green{\relax
\plot    
         0    0.0001
    0.0050    0.0001
    0.0100    0.0001
    0.0150    0.0001
    0.0200    0.0001
    0.0250    0.0001
    0.0300    0.0001
    0.0350    0.0006
    0.0400    0.0006
    0.0450    0.0006
    0.0500    0.0006
    0.0550    0.0006
    0.0600    0.0006
    0.0650    0.0021
    0.0700    0.0021
    0.0750    0.0021
    0.0800    0.0021
    0.0850    0.0021
    0.0900    0.0021
    0.0950    0.0053
    0.1000    0.0053
    0.1050    0.0053
    0.1100    0.0053
    0.1150    0.0053
    0.1200    0.0053
    0.1250    0.0078
    0.1300    0.0103
    0.1350    0.0103
    0.1400    0.0103
    0.1450    0.0103
    0.1500    0.0103
    0.1550    0.0103
    0.1600    0.0173
    0.1650    0.0173
    0.1700    0.0173
    0.1750    0.0173
    0.1800    0.0173
    0.1850    0.0173
    0.1900    0.0282
    0.1950    0.0282
    0.2000    0.0282
    0.2050    0.0282
    0.2100    0.0282
    0.2150    0.0282
    0.2200    0.0443
    0.2250    0.0443
    0.2300    0.0443
    0.2350    0.0443
    0.2400    0.0443
    0.2450    0.0443
    0.2500    0.0534
    0.2550    0.0625
    0.2600    0.0625
    0.2650    0.0625
    0.2700    0.0625
    0.2750    0.0625
    0.2800    0.0625
    0.2850    0.0808
    0.2900    0.0808
    0.2950    0.0808
    0.3000    0.0808
    0.3050    0.0808
    0.3100    0.0808
    0.3150    0.1028
    0.3200    0.1028
    0.3250    0.1028
    0.3300    0.1028
    0.3350    0.1028
    0.3400    0.1028
    0.3450    0.1270
    0.3500    0.1270
    0.3550    0.1270
    0.3600    0.1270
    0.3650    0.1270
    0.3700    0.1270
    0.3750    0.1336
    0.3800    0.1402
    0.3850    0.1402
    0.3900    0.1402
    0.3950    0.1402
    0.4000    0.1402
    0.4050    0.1402
    0.4100    0.1383
    0.4150    0.1383
    0.4200    0.1383
    0.4250    0.1383
    0.4300    0.1383
    0.4350    0.1383
    0.4400    0.1313
    0.4450    0.1313
    0.4500    0.1313
    0.4550    0.1313
    0.4600    0.1313
    0.4650    0.1313
    0.4700    0.1230
    0.4750    0.1230
    0.4800    0.1230
    0.4850    0.1230
    0.4900    0.1230
    0.4950    0.1230
    0.5000    0.1174
    0.5050    0.1117
    0.5100    0.1117
    0.5150    0.1117
    0.5200    0.1117
    0.5250    0.1117
    0.5300    0.1117
    0.5350    0.1062
    0.5400    0.1062
    0.5450    0.1062
    0.5500    0.1062
    0.5550    0.1062
    0.5600    0.1062
    0.5650    0.1199
    0.5700    0.1199
    0.5750    0.1199
    0.5800    0.1199
    0.5850    0.1199
    0.5900    0.1199
    0.5950    0.1508
    0.6000    0.1508
    0.6050    0.1508
    0.6100    0.1508
    0.6150    0.1508
    0.6200    0.1508
    0.6250    0.1655
    0.6300    0.1801
    0.6350    0.1801
    0.6400    0.1801
    0.6450    0.1801
    0.6500    0.1801
    0.6550    0.1801
    0.6600    0.1961
    0.6650    0.1961
    0.6700    0.1961
    0.6750    0.1961
    0.6800    0.1961
    0.6850    0.1961
    0.6900    0.2049
    0.6950    0.2049
    0.7000    0.2049
    0.7050    0.2049
    0.7100    0.2049
    0.7150    0.2049
    0.7200    0.2034
    0.7250    0.2034
    0.7300    0.2034
    0.7350    0.2034
    0.7400    0.2034
    0.7450    0.2034
    0.7500    0.1914
    0.7550    0.1793
    0.7600    0.1793
    0.7650    0.1793
    0.7700    0.1793
    0.7750    0.1793
    0.7800    0.1793
    0.7850    0.1440
    0.7900    0.1440
    0.7950    0.1440
    0.8000    0.1440
    0.8050    0.1440
    0.8100    0.1440
    0.8150    0.1206
    0.8200    0.1206
    0.8250    0.1206
    0.8300    0.1206
    0.8350    0.1206
    0.8400    0.1206
    0.8450    0.1102
    0.8500    0.1102
    0.8550    0.1102
    0.8600    0.1102
    0.8650    0.1102
    0.8700    0.1102
    0.8750    0.1038
    0.8800    0.0974
    0.8850    0.0974
    0.8900    0.0974
    0.8950    0.0974
    0.9000    0.0974
    0.9050    0.0974
    0.9100    0.0812
    0.9150    0.0812
    0.9200    0.0812
    0.9250    0.0812
    0.9300    0.0812
    0.9350    0.0812
    0.9400    0.0845
    0.9450    0.0845
    0.9500    0.0845
    0.9550    0.0845
    0.9600    0.0845
    0.9650    0.0845
    0.9700    0.0769
    0.9750    0.0769
    0.9800    0.0769
    0.9850    0.0769
    0.9900    0.0769
    0.9950    0.0769
    1.0000    0.0532
 /\relax}\relax
\endpicture
}
\setbox\figurenine=\vbox{\hsize=\xfiglen
\beginpicture
\eightrm
  \setcoordinatesystem units <\xfiglen,\yfiglen> 
  \setplotarea x from 0 to 1, y from 0 to 0.16
  \axis bottom shiftedto y=0 ticks short numbered from 0 to 1 by 0.2 /
  \axis left ticks short numbered from 0 to 0.16 by 0.04 /
%
\footnotesize
\put {$\kappa(x)$} [lt] at 0.02 0.16
\small
\linethickness=1.2pt
\putrule from 0.00 0.00 to 0.10 0.00
\putrule from 0.10 0.01 to 0.20 0.10
\putrule from 0.20 0.03 to 0.30 0.30
\putrule from 0.30 0.14 to 0.45 0.14
\putrule from 0.45 0.11 to 0.60 0.11
\putrule from 0.60 0.19 to 0.80 0.19
\putrule from 0.80 0.11 to 0.90 0.11
\putrule from 0.90 0.08 to 1.00 0.08
%
\setsolid
\setquadratic
\linethickness=0.5pt
%
\Blue{\relax   
\plot
         0    0.0000
    0.0100    0.0000
    0.0200    0.0002
    0.0300    0.0005
    0.0400    0.0012
    0.0500    0.0021
    0.0600    0.0035
    0.0700    0.0051
    0.0800    0.0070
    0.0900    0.0091
    0.1000    0.0113
    0.1100    0.0134
    0.1200    0.0154
    0.1300    0.0170
    0.1400    0.0183
    0.1500    0.0193
    0.1600    0.0199
    0.1700    0.0202
    0.1800    0.0204
    0.1900    0.0207
    0.2000    0.0213
    0.2100    0.0224
    0.2200    0.0242
    0.2300    0.0270
    0.2400    0.0309
    0.2500    0.0360
    0.2600    0.0423
    0.2700    0.0497
    0.2800    0.0583
    0.2900    0.0678
    0.3000    0.0782
    0.3100    0.0893
    0.3200    0.1008
    0.3300    0.1123
    0.3400    0.1235
    0.3500    0.1340
    0.3600    0.1434
    0.3700    0.1511
    0.3800    0.1568
    0.3900    0.1599
    0.4000    0.1602
    0.4100    0.1577
    0.4200    0.1523
    0.4300    0.1444
    0.4400    0.1347
    0.4500    0.1241
    0.4600    0.1134
    0.4700    0.1038
    0.4800    0.0962
    0.4900    0.0914
    0.5000    0.0898
    0.5100    0.0914
    0.5200    0.0960
    0.5300    0.1030
    0.5400    0.1115
    0.5500    0.1207
    0.5600    0.1297
    0.5700    0.1380
    0.5800    0.1450
    0.5900    0.1508
    0.6000    0.1554
    0.6100    0.1591
    0.6200    0.1625
    0.6300    0.1662
    0.6400    0.1707
    0.6500    0.1761
    0.6600    0.1823
    0.6700    0.1888
    0.6800    0.1949
    0.6900    0.2000
    0.7000    0.2036
    0.7100    0.2051
    0.7200    0.2046
    0.7300    0.2019
    0.7400    0.1973
    0.7500    0.1910
    0.7600    0.1836
    0.7700    0.1752
    0.7800    0.1664
    0.7900    0.1570
    0.8000    0.1473
    0.8100    0.1374
    0.8200    0.1280
    0.8300    0.1198
    0.8400    0.1133
    0.8500    0.1089
    0.8600    0.1065
    0.8700    0.1053
    0.8800    0.1041
    0.8900    0.1014
    0.9000    0.0962
    0.9100    0.0889
    0.9200    0.0818
    0.9300    0.0777
    0.9400    0.0782
    0.9500    0.0818
    0.9600    0.0833
    0.9700    0.0778
    0.9800    0.0702
    0.9900    0.0949
    1.0000    0.0647
 /\relax}\relax
\setlinear
\Green{\relax
\plot
         0    0.0001
    0.0050    0.0001
    0.0100    0.0001
    0.0150    0.0001
    0.0200    0.0001
    0.0250    0.0001
    0.0300    0.0001
    0.0350    0.0006
    0.0400    0.0006
    0.0450    0.0006
    0.0500    0.0006
    0.0550    0.0006
    0.0600    0.0006
    0.0650    0.0021
    0.0700    0.0021
    0.0750    0.0021
    0.0800    0.0021
    0.0850    0.0021
    0.0900    0.0021
    0.0950    0.0053
    0.1000    0.0053
    0.1050    0.0053
    0.1100    0.0053
    0.1150    0.0053
    0.1200    0.0053
    0.1250    0.0078
    0.1300    0.0103
    0.1350    0.0103
    0.1400    0.0103
    0.1450    0.0103
    0.1500    0.0103
    0.1550    0.0103
    0.1600    0.0173
    0.1650    0.0173
    0.1700    0.0173
    0.1750    0.0173
    0.1800    0.0173
    0.1850    0.0173
    0.1900    0.0282
    0.1950    0.0282
    0.2000    0.0282
    0.2050    0.0282
    0.2100    0.0282
    0.2150    0.0282
    0.2200    0.0443
    0.2250    0.0443
    0.2300    0.0443
    0.2350    0.0443
    0.2400    0.0443
    0.2450    0.0443
    0.2500    0.0534
    0.2550    0.0625
    0.2600    0.0625
    0.2650    0.0625
    0.2700    0.0625
    0.2750    0.0625
    0.2800    0.0625
    0.2850    0.0808
    0.2900    0.0808
    0.2950    0.0808
    0.3000    0.0808
    0.3050    0.0808
    0.3100    0.0808
    0.3150    0.1028
    0.3200    0.1028
    0.3250    0.1028
    0.3300    0.1028
    0.3350    0.1028
    0.3400    0.1028
    0.3450    0.1270
    0.3500    0.1270
    0.3550    0.1270
    0.3600    0.1270
    0.3650    0.1270
    0.3700    0.1270
    0.3750    0.1336
    0.3800    0.1402
    0.3850    0.1402
    0.3900    0.1402
    0.3950    0.1402
    0.4000    0.1402
    0.4050    0.1402
    0.4100    0.1383
    0.4150    0.1383
    0.4200    0.1383
    0.4250    0.1383
    0.4300    0.1383
    0.4350    0.1383
    0.4400    0.1313
    0.4450    0.1313
    0.4500    0.1313
    0.4550    0.1313
    0.4600    0.1313
    0.4650    0.1313
    0.4700    0.1230
    0.4750    0.1230
    0.4800    0.1230
    0.4850    0.1230
    0.4900    0.1230
    0.4950    0.1230
    0.5000    0.1174
    0.5050    0.1117
    0.5100    0.1117
    0.5150    0.1117
    0.5200    0.1117
    0.5250    0.1117
    0.5300    0.1117
    0.5350    0.1062
    0.5400    0.1062
    0.5450    0.1062
    0.5500    0.1062
    0.5550    0.1062
    0.5600    0.1062
    0.5650    0.1199
    0.5700    0.1199
    0.5750    0.1199
    0.5800    0.1199
    0.5850    0.1199
    0.5900    0.1199
    0.5950    0.1508
    0.6000    0.1508
    0.6050    0.1508
    0.6100    0.1508
    0.6150    0.1508
    0.6200    0.1508
    0.6250    0.1655
    0.6300    0.1801
    0.6350    0.1801
    0.6400    0.1801
    0.6450    0.1801
    0.6500    0.1801
    0.6550    0.1801
    0.6600    0.1961
    0.6650    0.1961
    0.6700    0.1961
    0.6750    0.1961
    0.6800    0.1961
    0.6850    0.1961
    0.6900    0.2049
    0.6950    0.2049
    0.7000    0.2049
    0.7050    0.2049
    0.7100    0.2049
    0.7150    0.2049
    0.7200    0.2034
    0.7250    0.2034
    0.7300    0.2034
    0.7350    0.2034
    0.7400    0.2034
    0.7450    0.2034
    0.7500    0.1914
    0.7550    0.1793
    0.7600    0.1793
    0.7650    0.1793
    0.7700    0.1793
    0.7750    0.1793
    0.7800    0.1793
    0.7850    0.1440
    0.7900    0.1440
    0.7950    0.1440
    0.8000    0.1440
    0.8050    0.1440
    0.8100    0.1440
    0.8150    0.1206
    0.8200    0.1206
    0.8250    0.1206
    0.8300    0.1206
    0.8350    0.1206
    0.8400    0.1206
    0.8450    0.1102
    0.8500    0.1102
    0.8550    0.1102
    0.8600    0.1102
    0.8650    0.1102
    0.8700    0.1102
    0.8750    0.1038
    0.8800    0.0974
    0.8850    0.0974
    0.8900    0.0974
    0.8950    0.0974
    0.9000    0.0974
    0.9050    0.0974
    0.9100    0.0812
    0.9150    0.0812
    0.9200    0.0812
    0.9250    0.0812
    0.9300    0.0812
    0.9350    0.0812
    0.9400    0.0845
    0.9450    0.0845
    0.9500    0.0845
    0.9550    0.0845
    0.9600    0.0845
    0.9650    0.0845
    0.9700    0.0769
    0.9750    0.0769
    0.9800    0.0769
    0.9850    0.0769
    0.9900    0.0769
    0.9950    0.0769
    1.0000    0.0532
 /\relax}\relax
\endpicture
}
\yfiglen=17 true in
\setbox\figureten=\vbox{\hsize=\xfiglen
\beginpicture
\eightrm
  \setcoordinatesystem units <\xfiglen,\yfiglen> 
  \setplotarea x from 0 to 1, y from 0 to 0.13
  \axis bottom shiftedto y=0 ticks short numbered from 0 to 1 by 0.2 /
  \axis left ticks short numbered from 0 to 0.12 by 0.02 /
\footnotesize
\put {$\kappa(x)$} [lt] at 0.02 0.13
\put{\box\figurelegendfive} [rt] at 0.9 0.12
\linethickness=0.6pt
\setplotsymbol ({\fiverm .})
\setquadratic
\setlinear
\setsolid
\Black{\relax  
\plot
         0         0
    0.0100         0
    0.0200         0
    0.0300         0
    0.0400         0
    0.0500         0
    0.0600         0
    0.0700         0
    0.0800         0
    0.0900         0
    0.1000         0
    0.1100    0.0003
    0.1200    0.0012
    0.1300    0.0027
    0.1400    0.0047
    0.1500    0.0073
    0.1600    0.0104
    0.1700    0.0140
    0.1800    0.0181
    0.1900    0.0227
    0.2000    0.0277
    0.2100    0.0331
    0.2200    0.0389
    0.2300    0.0450
    0.2400    0.0514
    0.2500    0.0580
    0.2600    0.0647
    0.2700    0.0716
    0.2800    0.0784
    0.2900    0.0851
    0.3000    0.0917
    0.3100    0.0979
    0.3200    0.1038
    0.3300    0.1092
    0.3400    0.1140
    0.3500    0.1182
    0.3600    0.1215
    0.3700    0.1239
    0.3800    0.1254
    0.3900    0.1259
    0.4000    0.1254
    0.4100    0.1237
    0.4200    0.1209
    0.4300    0.1171
    0.4400    0.1122
    0.4500    0.1064
    0.4600    0.0996
    0.4700    0.0922
    0.4800    0.0841
    0.4900    0.0755
    0.5000    0.0667
    0.5100    0.0578
    0.5200    0.0490
    0.5300    0.0405
    0.5400    0.0326
    0.5500    0.0253
    0.5600    0.0190
    0.5700    0.0138
    0.5800    0.0097
    0.5900    0.0068
    0.6000    0.0053
    0.6100    0.0051
    0.6200    0.0062
    0.6300    0.0084
    0.6400    0.0117
    0.6500    0.0159
    0.6600    0.0207
    0.6700    0.0260
    0.6800    0.0315
    0.6900    0.0369
    0.7000    0.0419
    0.7100    0.0464
    0.7200    0.0500
    0.7300    0.0526
    0.7400    0.0541
    0.7500    0.0544
    0.7600    0.0534
    0.7700    0.0512
    0.7800    0.0479
    0.7900    0.0437
    0.8000    0.0387
    0.8100    0.0332
    0.8200    0.0275
    0.8300    0.0218
    0.8400    0.0164
    0.8500    0.0116
    0.8600    0.0074
    0.8700    0.0041
    0.8800    0.0018
    0.8900    0.0004
    0.9000         0
    0.9100         0
    0.9200         0
    0.9300         0
    0.9400         0
    0.9500         0
    0.9600         0
    0.9700         0
    0.9800         0
    0.9900         0
    1.0000         0
 /\relax}\relax
\setsolid
\Blue{\relax  
\plot
  0.0000   0.0000
  0.0150   0.0000
  0.0300   0.0000
  0.0450   0.0000
  0.0600   0.0001
  0.0750   0.0001
  0.0900   0.0002
  0.1050   0.0003
  0.1200   0.0004
  0.1350   0.0005
  0.1500   0.0007
  0.1650   0.0009
  0.1800   0.0012
  0.1950   0.0015
  0.2100   0.0019
  0.2250   0.0022
  0.2400   0.0027
  0.2550   0.0031
  0.2700   0.0037
  0.2850   0.0042
  0.3000   0.0048
  0.3150   0.0055
  0.3300   0.0062
  0.3450   0.0069
  0.3600   0.0077
  0.3750   0.0085
  0.3900   0.0094
  0.4050   0.0103
  0.4200   0.0112
  0.4350   0.0122
  0.4500   0.0132
  0.4650   0.0142
  0.4800   0.0152
  0.4950   0.0163
  0.5100   0.0174
  0.5250   0.0185
  0.5400   0.0196
  0.5550   0.0207
  0.5700   0.0218
  0.5850   0.0230
  0.6000   0.0241
  0.6150   0.0253
  0.6300   0.0264
  0.6450   0.0276
  0.6600   0.0287
  0.6750   0.0298
  0.6900   0.0310
  0.7050   0.0321
  0.7200   0.0332
  0.7350   0.0342
  0.7500   0.0353
  0.7650   0.0363
  0.7800   0.0373
  0.7950   0.0382
  0.8100   0.0392
  0.8250   0.0401
  0.8400   0.0409
  0.8550   0.0417
  0.8700   0.0425
  0.8850   0.0432
  0.9000   0.0439
  0.9150   0.0446
  0.9300   0.0452
  0.9450   0.0458
  0.9600   0.0467
  0.9750   0.0478
  0.9900   0.0328
 /\relax}\relax
\Orange{\relax  
\plot
  0.0000   0.0000
  0.0150   0.0000
  0.0300   0.0000
  0.0450   0.0001
  0.0600   0.0001
  0.0750   0.0002
  0.0900   0.0004
  0.1050   0.0006
  0.1200   0.0009
  0.1350   0.0013
  0.1500   0.0017
  0.1650   0.0022
  0.1800   0.0028
  0.1950   0.0035
  0.2100   0.0042
  0.2250   0.0051
  0.2400   0.0060
  0.2550   0.0070
  0.2700   0.0081
  0.2850   0.0093
  0.3000   0.0105
  0.3150   0.0119
  0.3300   0.0132
  0.3450   0.0146
  0.3600   0.0160
  0.3750   0.0174
  0.3900   0.0189
  0.4050   0.0203
  0.4200   0.0217
  0.4350   0.0230
  0.4500   0.0243
  0.4650   0.0256
  0.4800   0.0267
  0.4950   0.0278
  0.5100   0.0288
  0.5250   0.0298
  0.5400   0.0306
  0.5550   0.0314
  0.5700   0.0320
  0.5850   0.0326
  0.6000   0.0331
  0.6150   0.0335
  0.6300   0.0338
  0.6450   0.0340
  0.6600   0.0342
  0.6750   0.0343
  0.6900   0.0343
  0.7050   0.0343
  0.7200   0.0341
  0.7350   0.0339
  0.7500   0.0336
  0.7650   0.0332
  0.7800   0.0327
  0.7950   0.0321
  0.8100   0.0314
  0.8250   0.0307
  0.8400   0.0300
  0.8550   0.0292
  0.8700   0.0284
  0.8850   0.0276
  0.9000   0.0268
  0.9150   0.0262
  0.9300   0.0256
  0.9450   0.0251
  0.9600   0.0249
  0.9750   0.0248
  0.9900   0.0169
  1.00     0.0089
 /\relax}\relax
\setsolid
\OliveGreen{\relax  
\plot
  0.0000   0.0000
  0.0150   0.0000
  0.0300   0.0001
  0.0450   0.0002
  0.0600   0.0005
  0.0750   0.0009
  0.0900   0.0016
  0.1050   0.0024
  0.1200   0.0035
  0.1350   0.0049
  0.1500   0.0064
  0.1650   0.0084
  0.1800   0.0106
  0.1950   0.0131
  0.2100   0.0160
  0.2250   0.0190
  0.2400   0.0224
  0.2550   0.0260
  0.2700   0.0298
  0.2850   0.0337
  0.3000   0.0378
  0.3150   0.0419
  0.3300   0.0459
  0.3450   0.0499
  0.3600   0.0537
  0.3750   0.0574
  0.3900   0.0606
  0.4050   0.0636
  0.4200   0.0661
  0.4350   0.0680
  0.4500   0.0697
  0.4650   0.0704
  0.4800   0.0708
  0.4950   0.0705
  0.5100   0.0697
  0.5250   0.0685
  0.5400   0.0666
  0.5550   0.0645
  0.5700   0.0621
  0.5850   0.0594
  0.6000   0.0567
  0.6150   0.0540
  0.6300   0.0513
  0.6450   0.0487
  0.6600   0.0463
  0.6750   0.0439
  0.6900   0.0419
  0.7050   0.0399
  0.7200   0.0379
  0.7350   0.0360
  0.7500   0.0340
  0.7650   0.0318
  0.7800   0.0295
  0.7950   0.0269
  0.8100   0.0241
  0.8250   0.0211
  0.8400   0.0178
  0.8550   0.0145
  0.8700   0.0112
  0.8850   0.0080
  0.9000   0.0049
  0.9150   0.0023
  0.9300  -0.0001
  0.9450  -0.0022
  0.9600  -0.0039
  0.9750  -0.0054
  0.9900  -0.0040
  1.0000  -0.0025
 /\relax}\relax
\Goldenrod{\relax  
\plot
  0.0000   0.0000
  0.0150   0.0001
  0.0300   0.0002
  0.0450   0.0006
  0.0600   0.0013
  0.0750   0.0021
  0.0900   0.0038
  0.1050   0.0057
  0.1200   0.0082
  0.1350   0.0114
  0.1500   0.0149
  0.1650   0.0196
  0.1800   0.0246
  0.1950   0.0303
  0.2100   0.0366
  0.2250   0.0432
  0.2400   0.0506
  0.2550   0.0581
  0.2700   0.0658
  0.2850   0.0735
  0.3000   0.0812
  0.3150   0.0882
  0.3300   0.0949
  0.3450   0.1007
  0.3600   0.1053
  0.3750   0.1093
  0.3900   0.1109
  0.4050   0.1115
  0.4200   0.1103
  0.4350   0.1072
  0.4500   0.1031
  0.4650   0.0962
  0.4800   0.0886
  0.4950   0.0796
  0.5100   0.0697
  0.5250   0.0595
  0.5400   0.0492
  0.5550   0.0394
  0.5700   0.0305
  0.5850   0.0233
  0.6000   0.0169
  0.6150   0.0140
  0.6300   0.0122
  0.6450   0.0128
  0.6600   0.0155
  0.6750   0.0192
  0.6900   0.0247
  0.7050   0.0304
  0.7200   0.0361
  0.7350   0.0407
  0.7500   0.0448
  0.7650   0.0459
  0.7800   0.0457
  0.7950   0.0432
  0.8100   0.0386
  0.8250   0.0330
  0.8400   0.0257
  0.8550   0.0185
  0.8700   0.0115
  0.8850   0.0058
  0.9000   0.0007
  0.9150  -0.0016
  0.9300  -0.0032
  0.9450  -0.0033
  0.9600  -0.0028
  0.9750  -0.0020
  0.9900  -0.0012
  1.0000  -0.0004
 /\relax}\relax
\Red{\relax  
\plot
  0.0000   0.0000
  0.0150   0.0001
  0.0300   0.0002
  0.0450   0.0005
  0.0600   0.0012
  0.0750   0.0020
  0.0900   0.0035
  0.1050   0.0053
  0.1200   0.0077
  0.1350   0.0109
  0.1500   0.0145
  0.1650   0.0193
  0.1800   0.0245
  0.1950   0.0306
  0.2100   0.0376
  0.2250   0.0450
  0.2400   0.0534
  0.2550   0.0621
  0.2700   0.0713
  0.2850   0.0805
  0.3000   0.0897
  0.3150   0.0983
  0.3300   0.1063
  0.3450   0.1132
  0.3600   0.1184
  0.3750   0.1228
  0.3900   0.1237
  0.4050   0.1234
  0.4200   0.1203
  0.4350   0.1146
  0.4500   0.1076
  0.4650   0.0970
  0.4800   0.0857
  0.4950   0.0728
  0.5100   0.0595
  0.5250   0.0460
  0.5400   0.0338
  0.5550   0.0227
  0.5700   0.0138
  0.5850   0.0081
  0.6000   0.0039
  0.6150   0.0048
  0.6300   0.0073
  0.6450   0.0127
  0.6600   0.0200
  0.6750   0.0282
  0.6900   0.0367
  0.7050   0.0444
  0.7200   0.0506
  0.7350   0.0539
  0.7500   0.0557
  0.7650   0.0527
  0.7800   0.0483
  0.7950   0.0413
  0.8100   0.0330
  0.8250   0.0242
  0.8400   0.0162
  0.8550   0.0092
  0.8700   0.0041
  0.8850   0.0012
  0.9000  -0.0005
  0.9150  -0.0002
  0.9300   0.0002
  0.9450   0.0006
  0.9600   0.0004
  0.9750  -0.0001
  0.9900  -0.0004
  1.0000   0.0001
 /\relax}\relax
\Black{\relax  
\plot
         0         0
    0.0100         0
    0.0200         0
    0.0300         0
    0.0400         0
    0.0500         0
    0.0600         0
    0.0700         0
    0.0800         0
    0.0900         0
    0.1000         0
    0.1100    0.0003
    0.1200    0.0012
    0.1300    0.0027
    0.1400    0.0047
    0.1500    0.0073
    0.1600    0.0104
    0.1700    0.0140
    0.1800    0.0181
    0.1900    0.0227
    0.2000    0.0277
    0.2100    0.0331
    0.2200    0.0389
    0.2300    0.0450
    0.2400    0.0514
    0.2500    0.0580
    0.2600    0.0647
    0.2700    0.0716
    0.2800    0.0784
    0.2900    0.0851
    0.3000    0.0917
    0.3100    0.0979
    0.3200    0.1038
    0.3300    0.1092
    0.3400    0.1140
    0.3500    0.1182
    0.3600    0.1215
    0.3700    0.1239
    0.3800    0.1254
    0.3900    0.1259
    0.4000    0.1254
    0.4100    0.1237
    0.4200    0.1209
    0.4300    0.1171
    0.4400    0.1122
    0.4500    0.1064
    0.4600    0.0996
    0.4700    0.0922
    0.4800    0.0841
    0.4900    0.0755
    0.5000    0.0667
    0.5100    0.0578
    0.5200    0.0490
    0.5300    0.0405
    0.5400    0.0326
    0.5500    0.0253
    0.5600    0.0190
    0.5700    0.0138
    0.5800    0.0097
    0.5900    0.0068
    0.6000    0.0053
    0.6100    0.0051
    0.6200    0.0062
    0.6300    0.0084
    0.6400    0.0117
    0.6500    0.0159
    0.6600    0.0207
    0.6700    0.0260
    0.6800    0.0315
    0.6900    0.0369
    0.7000    0.0419
    0.7100    0.0464
    0.7200    0.0500
    0.7300    0.0526
    0.7400    0.0541
    0.7500    0.0544
    0.7600    0.0534
    0.7700    0.0512
    0.7800    0.0479
    0.7900    0.0437
    0.8000    0.0387
    0.8100    0.0332
    0.8200    0.0275
    0.8300    0.0218
    0.8400    0.0164
    0.8500    0.0116
    0.8600    0.0074
    0.8700    0.0041
    0.8800    0.0018
    0.8900    0.0004
    0.9000         0
    0.9100         0
    0.9200         0
    0.9300         0
    0.9400         0
    0.9500         0
    0.9600         0
    0.9700         0
    0.9800         0
    0.9900         0
    1.0000         0
 /\relax}\relax
\endpicture
}
\setbox\figuretwelve=\vbox{\hsize=\xfiglen
\beginpicture
\eightrm
  \setcoordinatesystem units <0.01\xfiglen,10\yfiglen> 
  \setplotarea x from 0 to 100, y from 0 to 0.012
  \axis bottom shiftedto y=0 ticks short from 10 to 100  by 10 /
  \axis left ticks short numbered from 0 to 0.012 by 0.002 /
\put {10000} at 100 -0.001
\put {5000} at 50 -0.001
\put {1000} at 10 -0.001
\put {$n$} at 100 0.001
\put {$\|\kappa_{n}-\kappa_{\rm act}\|_\infty$} [lt] at 1 0.012
\plot
     1	    0.011106  
     2	    0.007826  
     3	    0.006022  
     4	    0.004831  
     5	    0.003974  
     6	    0.003331  
     7	    0.002834  
     8	    0.002439  
     9	    0.002118  
    10	    0.001854  
    11	    0.001632  
    12	    0.001445  
    13	    0.001286  
    14	    0.001149  
    15	    0.001032  
    16	    0.000932  
    17	    0.000845  
    18	    0.000770  
    19	    0.000706  
    20	    0.000650  
    21	    0.000601  
    22	    0.000559  
    23	    0.000522  
    24	    0.000489  
    25	    0.000461  
    26	    0.000436  
    27	    0.000413  
    28	    0.000393  
    29	    0.000375  
    30	    0.000359  
    31	    0.000344  
    32	    0.000331  
    33	    0.000318  
    34	    0.000307  
    35	    0.000297  
    36	    0.000287  
    37	    0.000278  
    38	    0.000270  
    39	    0.000262  
    40	    0.000255  
    41	    0.000248  
    42	    0.000242  
    43	    0.000236  
    44	    0.000231  
    45	    0.000225  
    46	    0.000220  
    47	    0.000216  
    48	    0.000211  
    49	    0.000207  
    50	    0.000203  
    51	    0.000199  
    52	    0.000195  
    53	    0.000192  
    54	    0.000188  
    55	    0.000185  
    56	    0.000182  
    57	    0.000179  
    58	    0.000176  
    59	    0.000174  
    60	    0.000171  
    61	    0.000169  
    62	    0.000166  
    63	    0.000164  
    64	    0.000161  
    65	    0.000159  
    66	    0.000157  
    67	    0.000155  
    68	    0.000153  
    69	    0.000151  
    70	    0.000149  
    71	    0.000147  
    72	    0.000146  
    73	    0.000144  
    74	    0.000142  
    75	    0.000140  
    76	    0.000139  
    77	    0.000137  
    78	    0.000136  
    79	    0.000134  
    80	    0.000133  
    81	    0.000131  
    82	    0.000130  
    83	    0.000128  
    84	    0.000127  
    85	    0.000126  
    86	    0.000125  
    87	    0.000123  
    88	    0.000122  
    89	    0.000121  
    90	    0.000120  
    91	    0.000118  
    92	    0.000117  
    93	    0.000116  
    94	    0.000115  
    95	    0.000114  
    96	    0.000113  
    97	    0.000112  
    98	    0.000111  
    99	    0.000110  
   100	    0.000109 
/
\endpicture
}
\yfiglen=17 true in
\setbox\figurethirteen=\vbox{\hsize=\xfiglen
\beginpicture
\eightrm
  \setcoordinatesystem units <\xfiglen,\yfiglen> 
  \setplotarea x from 0 to 1, y from 0 to 0.13
  \axis bottom shiftedto y=0 ticks short numbered from 0 to 1 by 0.2 /
  \axis left ticks short numbered from 0 to 0.12 by 0.02 /
%
\footnotesize
\put {$\kappa(x)$} [lt] at 0.02 0.13
\small
\linethickness=0.5pt
\setplotsymbol ({\sevenrm .})
\setquadratic
\setlinear
\setsolid
\Black{\relax  
\plot
         0         0
    0.0100         0
    0.0200         0
    0.0300         0
    0.0400         0
    0.0500         0
    0.0600         0
    0.0700         0
    0.0800         0
    0.0900         0
    0.1000         0
    0.1100    0.0003
    0.1200    0.0012
    0.1300    0.0027
    0.1400    0.0047
    0.1500    0.0073
    0.1600    0.0104
    0.1700    0.0140
    0.1800    0.0181
    0.1900    0.0227
    0.2000    0.0277
    0.2100    0.0331
    0.2200    0.0389
    0.2300    0.0450
    0.2400    0.0514
    0.2500    0.0580
    0.2600    0.0647
    0.2700    0.0716
    0.2800    0.0784
    0.2900    0.0851
    0.3000    0.0917
    0.3100    0.0979
    0.3200    0.1038
    0.3300    0.1092
    0.3400    0.1140
    0.3500    0.1182
    0.3600    0.1215
    0.3700    0.1239
    0.3800    0.1254
    0.3900    0.1259
    0.4000    0.1254
    0.4100    0.1237
    0.4200    0.1209
    0.4300    0.1171
    0.4400    0.1122
    0.4500    0.1064
    0.4600    0.0996
    0.4700    0.0922
    0.4800    0.0841
    0.4900    0.0755
    0.5000    0.0667
    0.5100    0.0578
    0.5200    0.0490
    0.5300    0.0405
    0.5400    0.0326
    0.5500    0.0253
    0.5600    0.0190
    0.5700    0.0138
    0.5800    0.0097
    0.5900    0.0068
    0.6000    0.0053
    0.6100    0.0051
    0.6200    0.0062
    0.6300    0.0084
    0.6400    0.0117
    0.6500    0.0159
    0.6600    0.0207
    0.6700    0.0260
    0.6800    0.0315
    0.6900    0.0369
    0.7000    0.0419
    0.7100    0.0464
    0.7200    0.0500
    0.7300    0.0526
    0.7400    0.0541
    0.7500    0.0544
    0.7600    0.0534
    0.7700    0.0512
    0.7800    0.0479
    0.7900    0.0437
    0.8000    0.0387
    0.8100    0.0332
    0.8200    0.0275
    0.8300    0.0218
    0.8400    0.0164
    0.8500    0.0116
    0.8600    0.0074
    0.8700    0.0041
    0.8800    0.0018
    0.8900    0.0004
    0.9000         0
    0.9100         0
    0.9200         0
    0.9300         0
    0.9400         0
    0.9500         0
    0.9600         0
    0.9700         0
    0.9800         0
    0.9900         0
    1.0000         0
/\relax}\relax
\Red{\relax  
\plot
  0.0000   0.0000
  0.0150   0.0000
  0.0300   0.0001
  0.0450   0.0002
  0.0600   0.0004
  0.0750   0.0008
  0.0900   0.0016
  0.1050   0.0027
  0.1200   0.0044
  0.1350   0.0069
  0.1500   0.0098
  0.1650   0.0142
  0.1800   0.0192
  0.1950   0.0253
  0.2100   0.0326
  0.2250   0.0405
  0.2400   0.0499
  0.2550   0.0595
  0.2700   0.0697
  0.2850   0.0800
  0.3000   0.0903
  0.3150   0.0997
  0.3300   0.1084
  0.3450   0.1158
  0.3600   0.1211
  0.3750   0.1255
  0.3900   0.1260
  0.4050   0.1251
  0.4200   0.1214
  0.4350   0.1150
  0.4500   0.1072
  0.4650   0.0960
  0.4800   0.0841
  0.4950   0.0709
  0.5100   0.0574
  0.5250   0.0439
  0.5400   0.0321
  0.5550   0.0215
  0.5700   0.0131
  0.5850   0.0079
  0.6000   0.0043
  0.6150   0.0056
  0.6300   0.0084
  0.6450   0.0139
  0.6600   0.0211
  0.6750   0.0292
  0.6900   0.0373
  0.7050   0.0447
  0.7200   0.0505
  0.7350   0.0536
  0.7500   0.0554
  0.7650   0.0525
  0.7800   0.0482
  0.7950   0.0414
  0.8100   0.0333
  0.8250   0.0246
  0.8400   0.0166
  0.8550   0.0095
  0.8700   0.0043
  0.8850   0.0011
  0.9000   0.0000
  0.9150   0.0000
  0.9300   0.0001
  0.9450   0.0001
  0.9600   0.0002
  0.9750   0.0002
  0.9900   0.0001
/\relax}\relax
\Blue{\relax  
\plot
         0         0
    0.0100         0
    0.0200         0
    0.0300         0
    0.0400         0
    0.0500         0
    0.0600         0
    0.0700         0
    0.0800         0
    0.0900         0
    0.1000    0.0002
    0.1100    0.0010
    0.1200    0.0018
    0.1300    0.0032
    0.1400    0.0051
    0.1500    0.0070
    0.1600    0.0103
    0.1700    0.0136
    0.1800    0.0176
    0.1900    0.0223
    0.2000    0.0270
    0.2100    0.0329
    0.2200    0.0388
    0.2300    0.0451
    0.2400    0.0517
    0.2500    0.0583
    0.2600    0.0651
    0.2700    0.0719
    0.2800    0.0786
    0.2900    0.0852
    0.3000    0.0917
    0.3100    0.0977
    0.3200    0.1036
    0.3300    0.1089
    0.3400    0.1136
    0.3500    0.1183
    0.3600    0.1212
    0.3700    0.1240
    0.3800    0.1256
    0.3900    0.1259
    0.4000    0.1262
    0.4100    0.1236
    0.4200    0.1210
    0.4300    0.1171
    0.4400    0.1118
    0.4500    0.1066
    0.4600    0.0993
    0.4700    0.0919
    0.4800    0.0840
    0.4900    0.0754
    0.5000    0.0668
    0.5100    0.0580
    0.5200    0.0491
    0.5300    0.0407
    0.5400    0.0328
    0.5500    0.0249
    0.5600    0.0192
    0.5700    0.0136
    0.5800    0.0095
    0.5900    0.0070
    0.6000    0.0046
    0.6100    0.0054
    0.6200    0.0063
    0.6300    0.0085
    0.6400    0.0121
    0.6500    0.0156
    0.6600    0.0208
    0.6700    0.0260
    0.6800    0.0313
    0.6900    0.0367
    0.7000    0.0421
    0.7100    0.0461
    0.7200    0.0501
    0.7300    0.0527
    0.7400    0.0539
    0.7500    0.0551
    0.7600    0.0531
    0.7700    0.0511
    0.7800    0.0479
    0.7900    0.0434
    0.8000    0.0389
    0.8100    0.0332
    0.8200    0.0275
    0.8300    0.0219
    0.8400    0.0166
    0.8500    0.0112
    0.8600    0.0076
    0.8700    0.0041
    0.8800    0.0017
    0.8900    0.0006
    0.9000    0.0001
    0.9100    0.0000
    0.9200         0
    0.9300         0
    0.9400         0
    0.9500    0.0000
    0.9600    0.0000
    0.9700         0
    0.9800         0
    0.9900         0
    1.0000    0.0000
/\relax}\relax
\endpicture
}
\yfiglen=14 true in
\setbox\figurefourteen=\vbox{\hsize=\xfiglen
\beginpicture
\eightrm
  \setcoordinatesystem units <\xfiglen,\yfiglen>  point at 0 0
  \setplotarea x from 0 to 1, y from 0 to 0.16
  \axis bottom shiftedto y=0 ticks short numbered from 0 to 1 by 0.2 /
  \axis left ticks short numbered from 0 to 0.16 by 0.04 /
\footnotesize
\put {$\kappa(x)$} [lt] at 0.02 0.16
\small
\linethickness=1.8pt
\setplotsymbol ({\sevenrm .})
\setquadratic
\setsolid
\Black{\relax  
\plot
         0         0
    0.0100         0
    0.0200         0
    0.0300         0
    0.0400         0
    0.0500         0
    0.0600    0.0015
    0.0700    0.0030
    0.0800    0.0045
    0.0900    0.0060
    0.1000    0.0075
    0.1100    0.0090
    0.1200    0.0105
    0.1300    0.0120
    0.1400    0.0135
    0.1500    0.0150
    0.1600    0.0165
    0.1700    0.0180
    0.1800    0.0195
    0.1900    0.0210
    0.2000    0.0225
    0.2100    0.0240
    0.2200    0.0255
    0.2300    0.0270
    0.2400    0.0285
    0.2500    0.0300
    0.2600    0.0340
    0.2700    0.0380
    0.2800    0.0420
    0.2900    0.0460
    0.3000    0.0500
    0.3100    0.0503
    0.3200    0.0505
    0.3300    0.0508
    0.3400    0.0510
    0.3500    0.0513
    0.3600    0.0515
    0.3700    0.0517
    0.3800    0.0520
    0.3900    0.0523
    0.4000    0.0525
    0.4100    0.0528
    0.4200    0.0530
    0.4300    0.0533
    0.4400    0.0535
    0.4500    0.0538
    0.4600    0.0540
    0.4700    0.0543
    0.4800    0.0545
    0.4900    0.0548
    0.5000    0.0550
    0.5100    0.0570
    0.5200    0.0590
    0.5300    0.0610
    0.5400    0.0630
    0.5500    0.0650
    0.5600    0.0670
    0.5700    0.0690
    0.5800    0.0710
    0.5900    0.0730
    0.6000    0.0750
    0.6100    0.0825
    0.6200    0.0900
    0.6300    0.0975
    0.6400    0.1050
    0.6500    0.1125
    0.6600    0.1200
    0.6700    0.1275
    0.6800    0.1350
    0.6900    0.1425
    0.7000    0.1500
    0.7100    0.1500
    0.7200    0.1500
    0.7300    0.1500
    0.7400    0.1500
    0.7500    0.1500
    0.7600    0.1500
    0.7700    0.1500
    0.7800    0.1500
    0.7900    0.1500
    0.8000    0.1500
    0.8100    0.1450
    0.8200    0.1400
    0.8300    0.1350
    0.8400    0.1300
    0.8500    0.1250
    0.8600    0.1200
    0.8700    0.1150
    0.8800    0.1100
    0.8900    0.1050
    0.9000    0.1000
    0.9100    0.0800
    0.9200    0.0600
    0.9300    0.0400
    0.9400    0.0200
    0.9500         0
    0.9600         0
    0.9700         0
    0.9800         0
    0.9900         0
    1.0000         0
 /\relax}\relax
\linethickness=0.6pt
\setdashes <3pt>
\setplotsymbol ({\tiny{\BrickRed{$\bullet$}}})
\plotsymbolspacing 3pt
\Red{\relax  
\plot
         0    0.0000
    0.0100    0.0001
    0.0200    0.0002
    0.0300    0.0004
    0.0400    0.0008
    0.0500    0.0012
    0.0600    0.0021
    0.0700    0.0030
    0.0800    0.0040
    0.0900    0.0054
    0.1000    0.0067
    0.1100    0.0083
    0.1200    0.0099
    0.1300    0.0115
    0.1400    0.0131
    0.1500    0.0147
    0.1600    0.0163
    0.1700    0.0179
    0.1800    0.0194
    0.1900    0.0210
    0.2000    0.0226
    0.2100    0.0244
    0.2200    0.0261
    0.2300    0.0281
    0.2400    0.0304
    0.2500    0.0326
    0.2600    0.0352
    0.2700    0.0378
    0.2800    0.0404
    0.2900    0.0430
    0.3000    0.0457
    0.3100    0.0478
    0.3200    0.0499
    0.3300    0.0515
    0.3400    0.0525
    0.3500    0.0536
    0.3600    0.0536
    0.3700    0.0537
    0.3800    0.0534
    0.3900    0.0528
    0.4000    0.0523
    0.4100    0.0519
    0.4200    0.0515
    0.4300    0.0515
    0.4400    0.0518
    0.4500    0.0521
    0.4600    0.0531
    0.4700    0.0541
    0.4800    0.0552
    0.4900    0.0565
    0.5000    0.0577
    0.5100    0.0588
    0.5200    0.0599
    0.5300    0.0610
    0.5400    0.0622
    0.5500    0.0633
    0.5600    0.0652
    0.5700    0.0672
    0.5800    0.0700
    0.5900    0.0738
    0.6000    0.0776
    0.6100    0.0837
    0.6200    0.0898
    0.6300    0.0968
    0.6400    0.1048
    0.6500    0.1127
    0.6600    0.1207
    0.6700    0.1287
    0.6800    0.1355
    0.6900    0.1411
    0.7000    0.1468
    0.7100    0.1487
    0.7200    0.1507
    0.7300    0.1515
    0.7400    0.1511
    0.7500    0.1507
    0.7600    0.1503
    0.7700    0.1499
    0.7800    0.1495
    0.7900    0.1490
    0.8000    0.1485
    0.8100    0.1449
    0.8200    0.1413
    0.8300    0.1363
    0.8400    0.1300
    0.8500    0.1237
    0.8600    0.1195
    0.8700    0.1154
    0.8800    0.1105
    0.8900    0.1049
    0.9000    0.0993
    0.9100    0.0799
    0.9200    0.0605
    0.9300    0.0403
    0.9400    0.0195
    0.9500    0.0003
    0.9600    0.0000
    0.9700    0.0000
    0.9800         0
    0.9900         0
    1.0000         0
/\relax}\relax
\setplotsymbol ({\sixrm{\Blue{.}}})
\plotsymbolspacing 2pt
\RoyalBlue{\relax  
\plot
         0    0.0000
    0.0100    0.0002
    0.0200    0.0003
    0.0300    0.0007
    0.0400    0.0013
    0.0500    0.0020
    0.0600    0.0032
    0.0700    0.0044
    0.0800    0.0058
    0.0900    0.0074
    0.1000    0.0089
    0.1100    0.0103
    0.1200    0.0117
    0.1300    0.0129
    0.1400    0.0139
    0.1500    0.0150
    0.1600    0.0159
    0.1700    0.0168
    0.1800    0.0179
    0.1900    0.0192
    0.2000    0.0205
    0.2100    0.0227
    0.2200    0.0248
    0.2300    0.0274
    0.2400    0.0303
    0.2500    0.0332
    0.2600    0.0363
    0.2700    0.0394
    0.2800    0.0422
    0.2900    0.0448
    0.3000    0.0473
    0.3100    0.0488
    0.3200    0.0502
    0.3300    0.0512
    0.3400    0.0516
    0.3500    0.0520
    0.3600    0.0519
    0.3700    0.0518
    0.3800    0.0518
    0.3900    0.0519
    0.4000    0.0519
    0.4100    0.0523
    0.4200    0.0527
    0.4300    0.0532
    0.4400    0.0536
    0.4500    0.0541
    0.4600    0.0544
    0.4700    0.0547
    0.4800    0.0551
    0.4900    0.0555
    0.5000    0.0559
    0.5100    0.0570
    0.5200    0.0582
    0.5300    0.0598
    0.5400    0.0618
    0.5500    0.0638
    0.5600    0.0663
    0.5700    0.0689
    0.5800    0.0718
    0.5900    0.0751
    0.6000    0.0783
    0.6100    0.0836
    0.6200    0.0889
    0.6300    0.0955
    0.6400    0.1035
    0.6500    0.1114
    0.6600    0.1201
    0.6700    0.1288
    0.6800    0.1361
    0.6900    0.1421
    0.7000    0.1481
    0.7100    0.1496
    0.7200    0.1511
    0.7300    0.1515
    0.7400    0.1506
    0.7500    0.1497
    0.7600    0.1495
    0.7700    0.1493
    0.7800    0.1492
    0.7900    0.1492
    0.8000    0.1492
    0.8100    0.1455
    0.8200    0.1418
    0.8300    0.1366
    0.8400    0.1297
    0.8500    0.1229
    0.8600    0.1190
    0.8700    0.1152
    0.8800    0.1106
    0.8900    0.1053
    0.9000    0.0999
    0.9100    0.0801
    0.9200    0.0602
    0.9300    0.0399
    0.9400    0.0194
    0.9500    0.0007
    0.9600    0.0002
    0.9700         0
    0.9800         0
    0.9900         0
    1.0000    0.0001
/\relax}\relax
\endpicture
}
\yfiglen=67 true in
\setbox\figurefifteen=\vbox{\hsize=\xfiglen
\beginpicture
\eightrm
  \setcoordinatesystem units <0.03\xfiglen,\yfiglen>  point at 1 0.015
  \setplotarea x from 1 to 30, y from 0.015 to 0.048
  \axis bottom shiftedto y=0.015 ticks short numbered from 5 to 30 by 5 /
  \axis left ticks short numbered from 0.015 to 0.45 by 0.005 /
\footnotesize
\put {$n$} [rb] at 30 0.016
\put {$\|\kappa_{n}-\kappa_{\rm act}\|_\infty$} [lt] at 2 0.048
\eightrm
\Red{\relax  
\plot
    2.0000    0.0340
    3.0000    0.0357
    4.0000    0.0360
    5.0000    0.0356
    6.0000    0.0349
    7.0000    0.0344
    8.0000    0.0339
    9.0000    0.0335
   10.0000    0.0331
   11.0000    0.0328
   12.0000    0.0324
   13.0000    0.0321
   14.0000    0.0318
   15.0000    0.0315
   16.0000    0.0313
   17.0000    0.0310
   18.0000    0.0308
   19.0000    0.0306
   20.0000    0.0304
   21.0000    0.0302
   22.0000    0.0300
   23.0000    0.0298
   24.0000    0.0296
   25.0000    0.0294
   26.0000    0.0293
   27.0000    0.0291
   28.0000    0.0290
   29.0000    0.0288
   30.0000    0.0287
/\relax}\relax
\RoyalBlue{\relax  
\plot
    2.0000    0.0438
    3.0000    0.0406
    4.0000    0.0356
    5.0000    0.0350
    6.0000    0.0393
    7.0000    0.0382
    8.0000    0.0319
    9.0000    0.0281
   10.0000    0.0266
   11.0000    0.0249
   12.0000    0.0237
   13.0000    0.0230
   14.0000    0.0222
   15.0000    0.0213
   16.0000    0.0215
   17.0000    0.0216
   18.0000    0.0218
   19.0000    0.0218
   20.0000    0.0219
   21.0000    0.0219
   22.0000    0.0219
   23.0000    0.0219
   24.0000    0.0220
   25.0000    0.0222
   26.0000    0.0223
   27.0000    0.0225
   28.0000    0.0226
   29.0000    0.0228
   30.0000    0.0229
/\relax}\relax
\endpicture
}

\captionsetup{width=0.9\textwidth}
\captionsetup{format=hang}

\subsection{Newton iteration}\label{sect:Newton_graphs}

Figures~\oldref{fig:Newton_iteration_kappa4} and \oldref{fig:Newton_iteration_kappa6}
below show reconstructions of  two actual $\kappa$ functions:
one $C^\infty$ and the other just piecewise linear.
In each case the frozen Newton scheme was used with the leftmost graphs
for the case of a noise level of $0.1\%$ and the rightmost for $1\%$ noise.
Note the very rapid convergence of the scheme.
In the case of $1\%$ noise the scheme was stopped by the Discrepancy Principle
at the second iteration, whereas with the lower noise level a third iteration
was reached.

 \begin{figure}[h]
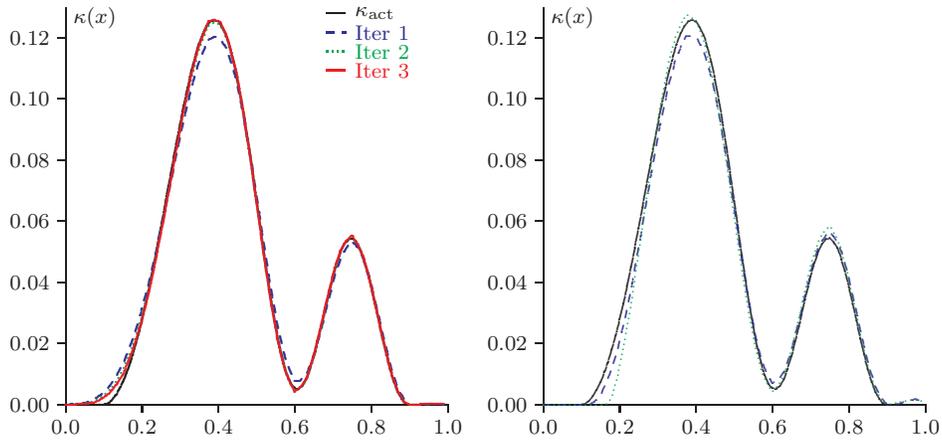

 \bigskip
 \hbox to\hsize{\hss\copy\figureone\hss\hss\copy\figuretwo\hss}
 \caption{
 Reconstructions of a smooth $\kappa(x)$ from time trace data at $\,x=1\,$
 under 0.1\%  (left) and 1\% (right) noise using Newton's method.
 }
 \label{fig:Newton_iteration_kappa4}
 \end{figure}

 \begin{figure}[h]
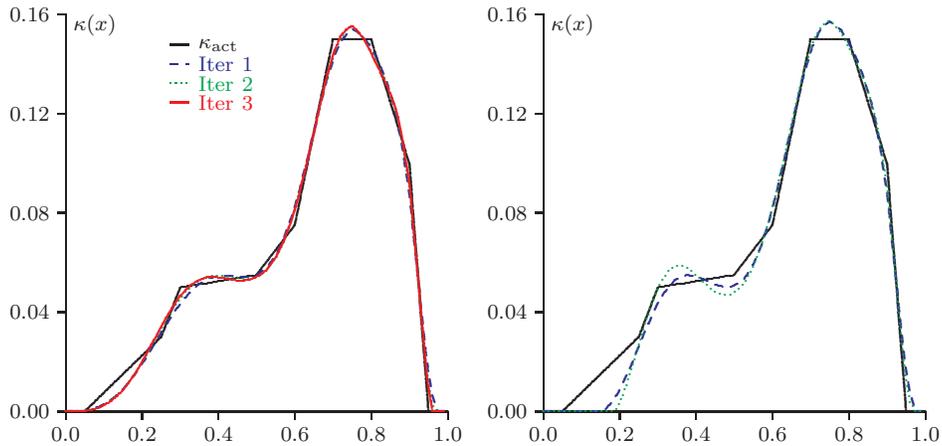

 \bigskip
 \hbox to\hsize{\hss\box\figurethree\hss\hss\box\figurefour\hss}
 \caption{
 Reconstructions of piecewise linear $\kappa(x)$ from time trace data at $\,x=1$
 under 0.1\%  (left) and 1\% (right) noise using Newton's method.
 }
 \label{fig:Newton_iteration_kappa6}
 \end{figure}

As a final example for Newton's method we show a reconstruction of a piecewise
constant function in Figure \oldref{fig:kappa9_reconstruction}.
This uses both piecewise linear or chapeau basis functions and also a
Haar piecewise constant basis.
We were careful to ensure that the basis breakpoints did not align with the
discontinuities of $\kappa$.
Clearly, the reconstruction here is much poorer but still able to follow the
significant features of the actual $\kappa(x)$ function.

 \begin{figure}[h]
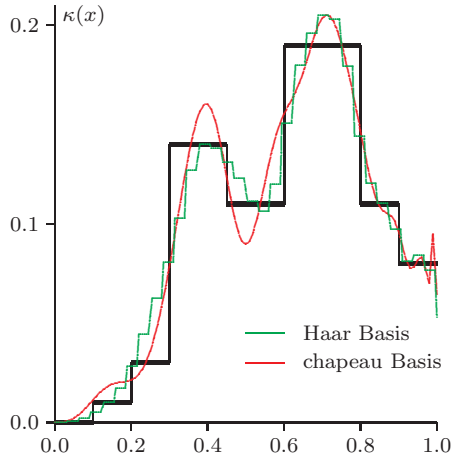

 \bigskip
 \hbox to\hsize{\hss\box\figureseven\hss}
 \caption{
 Reconstructions of a piecewise constant $\kappa(x)$ from time trace data at
$x=1$ under $0.1\%$ noise using Newton iteration.
 }
 \label{fig:kappa9_reconstruction}
 \end{figure}
\bigskip

\subsection{Halley iteration}\label{sect:Halley_graphs}

Theoretically, the frozen Newton scheme is only first order accurate and
thus it is natural to seek a higher order of convergence method.
Thus comparisons with Halley should be made here.
The frozen Halley with the added corrector step is second order and might
seem the obvious choice.
Implementation requires the computation of the Hessian and this if implemented
in the most straightforward way takes approximately $M$ times longer than
just computing the Jacobian where $M$ is the number of basis elements.
Thus, given the rapid convergence of the Newton scheme with smooth $\kappa$
the use of Halley to improve the convergence rate seems counterproductive
given what can be a considerable computational expense.
However, there is another factor to consider.
In many cases, (see for example, \cite{HeRu00}) the use of the corrector can
sometimes lead to a slightly improved reconstruction and this is
the case here as we show in Figure~\oldref{fig:Newton_Halley_compare}.
 \begin{figure}[h]
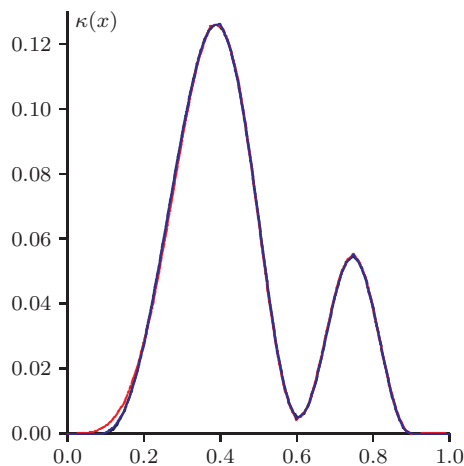

 \bigskip
 \hbox to\hsize{\hss\box\figurethirteen\hss}
 \caption{
 Comparison of Newton (in red) and Halley (in blue) final reconstructions under $0.1\%$ noise.}
 \label{fig:Newton_Halley_compare}
 \end{figure}
These two reconstructions (from the same data subject to $0.1\%$ noise)
are hardly distinguishable except for the extreme left hand endpoint
where Halley outperforms Newton.
The final errors:
$\;\|\kappa - \kappa_{\rm act}\|_{\infty}$,
$\;\|\kappa - \kappa_{\rm act}\|_{L^2}\,$
for Newton were $0.0084$ and $0.0059$ respectively.
The corresponding errors for Halley were
$0.0066$ and $0.0045$.

For this numerical run
there is a homogeneous Dirichlet condition at the left hand endpoint
and the small values of $p(x,t)$ near $x=0$ are a multiplier for the $\kappa$
to be reconstructed. With these boundary conditions reconstructions were
always worse at this endpoint.

Figure~\oldref{fig:Newton_Halley_norm_compare} below shows another
Newton-Halley comparison this time using a piecewise linear function for
$\kappa$.
Here we took Neumann boundary conditions at both endpoints and so the above
issue of the small multiplier is lessened.
Again, there is a slight improvement using Halley's method.
\medskip
 \begin{figure}[h]
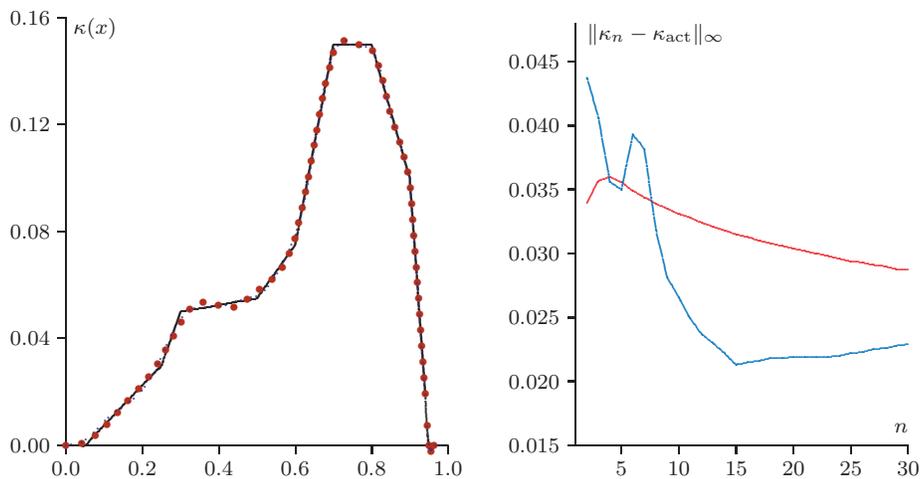

 \bigskip
 \hbox to\hsize{\hss\box\figurefourteen\hss\hss\box\figurefifteen\hss}
 \caption{
 Comparison of Newton (in red) and Halley (in blue) final reconstructions and
norm differences of the $n^{\rm th}$ iterate $\kappa_n$ and the actual $\kappa$.
Noise level was $1\%$.}
 \label{fig:Newton_Halley_norm_compare}
 \end{figure}

The rightmost figure shows the $L^\infty$ norm difference between the
computed and the actual $\kappa$ at each iteration. Several features are worthy of note.

First, both Newton and Halley schemes wander on the first  few iterations
and this is much more noticeable in the Halley.
There is no stepsize control implemented and the actual $\kappa$ is quite
far from the initial approximation $\kappa=0$.
This effect would disappear with proper stepsize control.

Second, after this initial phase the Newton progresses with approximately
linear convergence whereas the Halley decreases roughly quadratically.
In this simulation the use of the Discrepancy Principle to terminate the
iteration process was turned off and so after a certain point the noise in
the data plays a role and the norm differences start to increase again.
This point is after about 15 iterations with Halley but much later with Newton;
somewhere around 35 iterations.

Third, note that after terminating the schemes by some effective mechanism
the overall error in the Halley method is less than that of the Newton by a
factor of approximately the same amount as noted for
Figure~\oldref{fig:Newton_Halley_compare}.

Thus in summary here, the computational cost of a Halley implementation
and the need to compute the Hessian is not repaid in time to reach convergence
as opposed to Newton, but there is a relatively small but consistent advantage
of a superior reconstruction.

\subsection{Landweber iteration}\label{sect:Landweber_graphs}

As a point of reference, Figure~\oldref{fig:Landweber_iteration_kappa4}
shows the piecewise linear actual $\kappa$ reconstructed using Landweber
iteration.
After 5,000 steps the scheme was still converging so computational efficiency
is orders of magnitude below either Newton or Halley methods.

The rate was better for a smooth $\kappa$ as shown in
Figure~\oldref{fig:Landweber_iteration_kappa6} but typical of this scheme
it lags orders of magnitude behind Newton-schemes if these are applicable to
the problem.
 \begin{figure}[h]
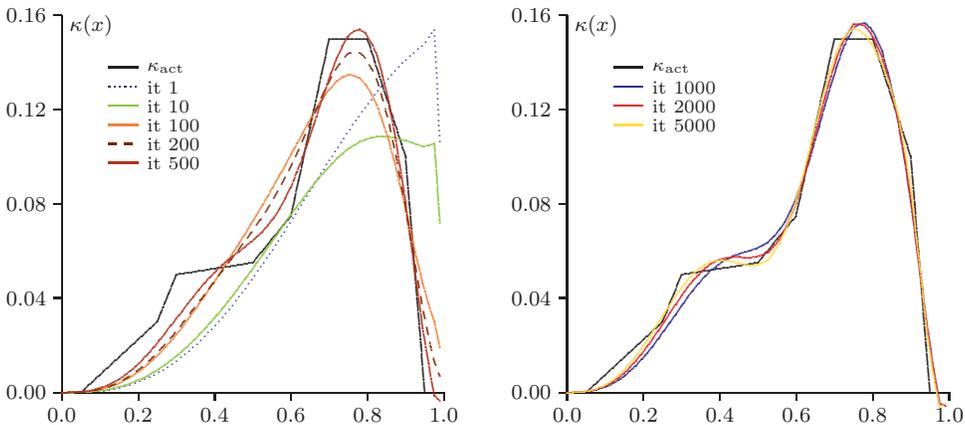

 \bigskip
 \hbox to\hsize{\hss\copy\figurefive\qquad\copy\figuresix\hss}
 \caption{
 Reconstructions of a piecewise linear $\kappa(x)$ from time trace data at $x=1$
 under $1\%$ noise using Landweber iteration.
 }
 \label{fig:Landweber_iteration_kappa4}
 \end{figure}

 \begin{figure}[h]
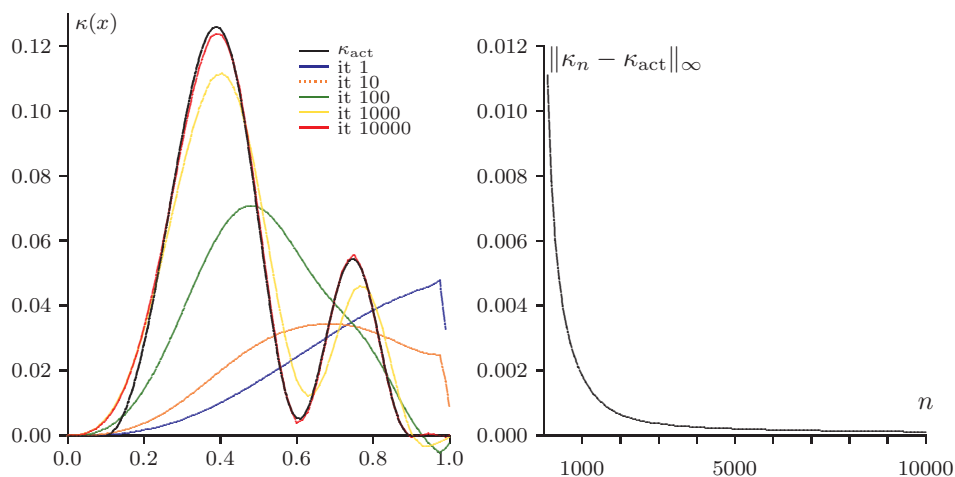

 \bigskip
 \hbox to\hsize{\hss\copy\figureten\hss\hss\copy\figuretwelve\hss}
 \caption{
The leftmost figure shows reconstructions of $\kappa(x)$
under $0.1\%$ noise using Landweber iteration.
The rightmost figure shows the decay of the norm
$\kappa_n(x)-\kappa_{\rm act}(x)$
 }
 \label{fig:Landweber_iteration_kappa6}
 \end{figure}
The main advantage to the Landweber scheme is its greater robustness against
noise than the Newton scheme and certainly so against the Halley which requires
two regularisation constants, one for each of the predictor and the corrector
steps.

As noted in Section \oldref{sec:LW}, there are acceleration methods available
for this scheme, but given the performance of the Newton and Halley schemes
shown here, these seem the methods of choice.

\section*{Acknowledgments}
The work of the first author was supported by the Austrian Science Fund {\sc fwf}
under the grants P30054 and DOC78. The work of the second author was supported
in part by the National Science Foundation through award {\sc dms}-1620138. We thank both reviewers for their careful reading of the manuscript and their detailed reports with valuable comments and suggestions that have led to an improved version of the paper.

\medskip
Received August 2020; 1st revision August 2020; 2nd revision November 2020.
\medskip

{\it E-mail address: } barbara.kaltenbacher@aau.at\\
\indent {\it E-mail address: } rundell@math.tamu.edu


\begin{thebibliography}{99}

\bibitem{Baku92} (MR1185952)
\newblock A.~B. Bakushinski\u{i},
\newblock On a convergence problem of the iterative-regularised Gauss-Newton method,
\newblock \emph{Comput. Math. Math. Phys.}, \textbf{32} (1992), 1353--1359.

\bibitem{Bakushinski1984} [10.1016/0041-5553(84)90253-2]
\newblock A. B. Bakushinskii,
\newblock {Remarks on choosing a regularisation parameter using the quasi-optimality and ratio criterion},
\newblock \emph{USSR Comput. Math. Math. Phys.}, \textbf{24} (1984), 181--182.
\bibitem{Bjorno1986} [10.1016/0041-624x(86)90102-2]
\newblock L.~Bj{\o}rn{\o},
\newblock {Characterization of biological media by means of their non-linearity},
\newblock \emph{Ultrasonics}, \textbf{24} (1986), 254--259.

\bibitem{Blackstock63}
\newblock D. T. Blackstock,
\newblock Approximate equations governing finite-amplitude sound in thermoviscous fluids,
\newblock Tech Report, GD/E Report, GD-1463-52, General Dynamics Corp., Rochester, NY, 1963.

\bibitem{BNS97} (MR1459331) [10.1093/imanum/17.3.421]
\newblock B.~Blaschke, A.~Neubauer and O.~Scherzer,
\newblock {On convergence rates for the iteratively regularised {G}auss-{N}ewton method},
\newblock \emph{IMA J. Numer. Anal.}, \textbf{17} (1997), 421--436.

\bibitem{Burgers74} [10.1007/978-94-010-1745-9]
\newblock J. M. Burgers,
\newblock \emph{The Nonlinear Diffusion Equation},
\newblock Springer, Netherlands, 1974.

\bibitem{BurovGurinovichRudenkoTagunov1994}
\newblock V.~Burov, I.~Gurinovich, O.~Rudenko and E.~Tagunov,
\newblock Reconstruction of the spatial distribution of the nonlinearity parameter and sound velocity in acoustic nonlinear tomography,
\newblock \emph{Acoustical Physics}, \textbf{40} (1994), 816--823.

\bibitem{Cain1986} [10.1109/ULTSYM.1985.198640]
\newblock C.~A. Cain,
\newblock {Ultrasonic reflection mode imaging of the nonlinear parameter B/A: A theoretical basis},
\newblock IEEE 1985 Ultrasonics Symposium, San Francisco, CA, USA, 1985.

\bibitem{CK18} (MR3859297) [10.1515/jiip-2018-0026]
\newblock C. Clason and A. Klassen,
\newblock {Quasi-solution of linear inverse problems in non-reflexive Banach spaces},
\newblock \emph{J. Inverse Ill-Posed Probl.}, \textbf{26} (2018), 689--702.

\bibitem{CK11} (MR2775195) [10.1051/cocv/2010003]
\newblock C. Clason and K. Kunisch,
\newblock {A duality-based approach to elliptic control problems in non-reflexive Banach spaces},
\newblock \emph{ESAIM Control Optim. Calc. Var.}, \textbf{17} (2011), 243--266.

\bibitem{Crighton79} [10.1146/annurev.fl.11.010179.000303]
\newblock D.~G. Crighton,
\newblock {Model equations of nonlinear acoustics},
\newblock \emph{Ann. Rev. Fluid Mech.}, \textbf{11} (1979), 11--33.

\bibitem{EnglHankeNeubauer:1996} (MR1408680)
\newblock H.~W. Engl, M. Hanke and A. Neubauer,
\newblock \emph{Regularization of {I}nverse {P}roblems},
\newblock Mathematics and its Applications, 375, Kluwer Academic Publishers Group, Dordrecht, 1996.

\bibitem{EKN89} (MR1009037) [10.1088/0266-5611/5/4/007]
\newblock H.~W. Engl, K.~Kunisch and A.~Neubauer,
\newblock {Convergence rates for {Tikhonov} regularisation of non-linear ill-posed problems},
\newblock \emph{Inverse Problems}, \textbf{5} (1989), 523--540.

\bibitem{Evansbuch} (MR1625845)
\newblock L.~C. Evans,
\newblock \emph{Partial Differential Equations},
\newblock Graduate Studies in Mathematics, 19,  American Mathematical Society, Providence, RI, 1998.

\bibitem{HamiltonBlackstock98}
\newblock M.~F. Hamilton and D.~T. Blackstock,
\newblock \emph{Nonlinear Acoustics, Vol. 1},
\newblock Academic Press, San Diego, 1998.

\bibitem{Hanke97} (MR1435869) [10.1088/0266-5611/13/1/007]
\newblock M.~Hanke,
\newblock {A regularizing Levenberg--Marquardt scheme, with applications to inverse groundwater filtration problems},
\newblock \emph{Inverse Problems}, \textbf{13} (1997), 79--95.

\bibitem{HNS95} (MR1359706) [10.1007/s002110050158]
\newblock M.~Hanke, A.~Neubauer and O.~Scherzer,
\newblock {A convergence analysis of the {L}andweber iteration for nonlinear ill-posed problems},
\newblock \emph{Numer. Math.}, \textbf{72} (1995), 21--37.

\bibitem{HeRu00} (MR1740766) [10.1137/S0036142998341246]
\newblock F.~Hettlich and W.~Rundell,
\newblock {A second degree method for nonlinear inverse problems},
\newblock \emph{SIAM J. Numer. Anal.}, \textbf{37} (2000), 587--620.

\bibitem{HKPS07} (MR2329928) [10.1088/0266-5611/23/3/009]
\newblock B.~Hofmann, B.~Kaltenbacher, C. P\"oschl and O.~Scherzer,
\newblock {A convergence rates result for Tikhonov regularisation in Banach spaces with non-smooth operators},
\newblock \emph{Inverse Problems}, \textbf{23} (2007), 987--1010.

\bibitem{Hoha97} (MR1474369) [10.1088/0266-5611/13/5/012]
\newblock T.~Hohage,
\newblock {Logarithmic convergence rates of the iteratively regularised {G}au\ss-{N}ewton method for an inverse potential and an inverse scattering problem},
\newblock \emph{Inverse Problems}, \textbf{13} (1997), 1279--1299.

\bibitem{HubmerRamlau2018} (MR3830146) [10.1088/1361-6420/aacebe]
\newblock S. Hubmer and R. Ramlau,
\newblock {Nesterov's accelerated gradient method for nonlinear ill-posed problems with a locally convex residual functional},
\newblock \emph{Inverse Problems}, \textbf{34} (2018), 30pp.

\bibitem{IchidaSatoLinzer1983} [10.1177/016173468300500401]
\newblock N. Ichida, T. Sato and M. Linzer,
\newblock {Imaging the nonlinear ultrasonic parameter of a medium},
\newblock \emph{Ultrasonic Imaging}, \textbf{5} (1983), 295--299.

\bibitem{ImanuvilovYamamoto:2019} (MR4041949) [10.1088/1361-6420/ab323e]
\newblock O.~Y. Imanuvilov and M.~Yamamoto,
\newblock {Carleman estimate and an inverse source problem for the Kelvin-Voigt model for viscoelasticity},
\newblock \emph{Inverse Problems}, \textbf{35} (2019), 45pp.

\bibitem{Isakov:2006} (MR2193218) [10.1007/0-387-32183-7]
\newblock V. Isakov,
\newblock \emph{Inverse Problems for Partial Differential Equations},
\newblock Applied Mathematical Sciences, 127, Springer, New York, 2006.

\bibitem{Ivanov62} (MR0140944)
\newblock V.~K. Ivanov,
\newblock On linear problems which are not well-posed,
\newblock \emph{Dokl. Akad. Nauk SSSR}, \textbf{145} (1962), 270--272.

\bibitem{BK15} (MR3383327) [10.1007/s00211-014-0682-5]
\newblock B.~Kaltenbacher,
\newblock {An iteratively regularized Gauss-Newton-Halley method for solving nonlinear ill-posed problems},
\newblock \emph{Numer. Math.}, \textbf{131} (2015), 33--57.

\bibitem{reviewNonlinearAcoustics} (MR3461695) [10.3934/eect.2015.4.447]
\newblock B.~Kaltenbacher,
\newblock {Mathematics of nonlinear acoustics},
\newblock \emph{Evol. Equ. Control Theory}, \textbf{4} (2015), 447--491.

\bibitem{periodicWestervelt} [10.3934/eect.2020063]
\newblock B.~Kaltenbacher,
\newblock {Periodic solutions and multiharmonic expansions for the {W}estervelt equation},
\newblock to appear, \emph{Evol. Equ. Control Theory}.

\bibitem{KL09} (MR2525765) [10.3934/dcdss.2009.2.503]
\newblock B.~Kaltenbacher and I.~Lasiecka,
\newblock {Global existence and exponential decay rates for the Westervelt equation},
\newblock \emph{Discrete Contin. Dyn. Syst. Ser. S}, \textbf{2} (2009), 503--523.

\bibitem{KKlassen17} (MR3788158) [10.1088/1361-6420/aab739]
\newblock B. Kaltenbacher and A. Klassen,
\newblock {On convergence and convergence rates for Ivanov and Morozov regularisation and application to some parameter identification problems in elliptic PDEs},
\newblock \emph{Inverse Problems}, \textbf{34} (2018), 24pp.

\bibitem{KNS08} (MR2459012) [10.1515/9783110208276]
\newblock B.~Kaltenbacher, A.~Neubauer and O.Scherzer,
\newblock {Iterative Regularization Methods for Nonlinear Ill-Posed Problems},
\newblock Radon Series on Computational and Applied Mathematics, 6, Walter de Gruyter GmbH \& Co. KG, Berlin, 2008.

\bibitem{Kuznetsov71}
\newblock V.~Kuznetsov,
\newblock Equations of nonlinear acoustics,
\newblock \emph{Soviet Physics - Acoustics}, \textbf{16} (1971), 467--470.

\bibitem{LesserSeebass68} [10.1017/S0022112068000303]
\newblock M.~B. Lesser and R.~Seebass,
\newblock {The structure of a weak shock wave undergoing reflexion from a wall},
\newblock \emph{J. Fluid Mech.}, \textbf{31} (1968), 501--528.

\bibitem{Lighthill56} (MR0077346)
\newblock M. J. Lighthill,
\newblock Viscosity effects in sound waves of finite amplitude,
\newblock in \emph{Surveys in Mechanics}, Cambridge, at the University Press, 1956, 250--351.

\bibitem{LorenzWorliczek13} (MR3080477) [10.1088/0266-5611/29/7/075016]
\newblock D. Lorenz and N. Worliczek,
\newblock {Necessary conditions for variational regularisation schemes},
\newblock \emph{Inverse Problems}, \textbf{29} (2013), 19pp.

\bibitem{MeyerWilke11} (MR2822410) [10.1007/s00245-011-9138-9]
\newblock S. Meyer and M. Wilke,
\newblock {Optimal regularity and long-time behavior of solutions for the {W}estervelt equation},
\newblock \emph{Appl. Math. Optim.}, \textbf{64} (2011), 257--271.

\bibitem{Morozov66} (MR0208819)
\newblock V.~A. Morozov,
\newblock On the solution of functional equations by the method of regularisation,
\newblock \emph{Soviet Math. Dokl.}, \textbf{7} (1966), 414--417.

\bibitem{MNWW19} (MR3923817) [10.3934/eect.2019010]
\newblock M. Muhr, V. Nikoli\'{c}, B. Wohlmuth and L. Wunderlich,
\newblock {Isogeometric shape optimization for nonlinear ultrasound focusing},
\newblock \emph{Evol. Equ. Control Theory}, \textbf{8} (2019), 163--202.

\bibitem{Neubauer2017} (MR3652264) [10.1515/jiip-2016-0060]
\newblock A. Neubauer,
\newblock {On Nesterov acceleration for Landweber iteration of linear ill-posed problems},
\newblock \emph{J. Inverse Ill-Posed Probl.}, \textbf{25} (2017), 381--390.

\bibitem{Neubauer88} (MR947434) [10.1016/0021-9045(88)90025-1]
\newblock A. Neubauer,
\newblock {Tikhonov-regularisation of ill-posed linear operator equations on closed convex sets},
\newblock \emph{J. Approx. Theory}, \textbf{53} (1988), 304--320.

\bibitem{NeSc95} (MR1337265) [10.4171/ZAA/679]
\newblock A.~Neubauer and O.~Scherzer,
\newblock {A convergent rate result for a steepest descent method and a minimal error method for the solution of nonlinear ill-posed problems},
\newblock \emph{Z. Anal. Anwendungen}, \textbf{14} (1995), 369--377.

\bibitem{Ockendon06} (MR2027679) [10.1007/b97537] 
\newblock H.~Ockendon and J. R. Ockendon,
\newblock \emph{Waves and Compressible Flow},
\newblock Texts in Applied Mathematics, 47, Springer-Verlag, New York, 2004.

\bibitem{Pierce:1979} (MR534419) [10.1137/0317035]
\newblock A. Pierce,
\newblock {Unique identification of eigenvalues and coefficients in a parabolic problem},
\newblock \emph{SIAM J. Control Optim.}, \textbf{17} (1979), 494--499.

\bibitem{Ried01} (MR1826857) [10.1007/PL00005448]
\newblock A.~Rieder,
\newblock {On convergence rates of inexact {N}ewton regularizations},
\newblock \emph{Numer. Math.}, \textbf{88} (2001), 347--365.

\bibitem{RundellSacks:1992a} (MR1106979) [10.1090/S0025-5718-1992-1106979-0]
\newblock W. Rundell and P.~E. Sacks,
\newblock {Reconstruction techniques for classical inverse {S}turm-{L}iouville problems},
\newblock \emph{Math. Comp.}, \textbf{58} (1992), 161--183.

\bibitem{Sche98} (MR1620795) [10.1007/s002459900081]
\newblock O.~Scherzer,
\newblock {A modified {L}andweber iteration for solving parameter estimation problems},
\newblock \emph{Appl. Math. Optim.}, \textbf{38} (1998), 45--68.

\bibitem{SeidmanVogel89} (MR991919) [10.1088/0266-5611/5/2/008]
\newblock T.~I. Seidman and C.~R. Vogel,
\newblock {Well-posedness and convergence of some regularisation methods for non-linear ill posed problems},
\newblock \emph{Inverse Problems}, \textbf{5} (1989), 227--238.

\bibitem{VarrayBassetTortoliCachard2011} [10.1109/TUFFC.2011.1933]
\newblock F. Varray, O. Basset, P. Tortoli and C. Cachard,
\newblock {Extensions of nonlinear B/A parameter imaging methods for echo mode},
\newblock \emph{IEEE Trans. Ultrasonics, Ferroelectrics, and Frequency Control}, \textbf{58} (2011), 1232--1244.

\bibitem{Westervelt63} [10.1121/1.1918525]
\newblock P.~J. Westervelt,
\newblock {Parametric acoustic array},
\newblock \emph{J. Acoustical Soc. Amer.}, \textbf{35} (1963), 535--537.

\bibitem{YamamotoBK20}
\newblock M. Yamamoto and B. Kaltenbacher,
\newblock An inverse source problem related to acoustic nonlinearity parameter imaging,
\newblock to appear, \emph{Time-Dependent Problems in Imaging and Parameter Identification}, Springer, 2021.

\bibitem{ZabolotskayaKhokhlov69}
\newblock E.~A. Zabolotskaya and R.~V. Khokhlov,
\newblock Quasi-plane waves in the non-linear acoustics of confined beams,
\newblock \emph{Soviet Physics - Acoustics}, \textbf{15} (1969), 35--40.

\bibitem{ZhangChenGong2001} [10.1121/1.1344160]
\newblock D. Zhang, X.~Chen and X.-F. Gong,
\newblock {Acoustic nonlinearity parameter tomography for biological tissues via parametric array from a circular piston source - Theoretical analysis and computer simulations},
\newblock \emph{J. Acoustical Soc. Amer.}, \textbf{109} (2001), 1219--1225.

\bibitem{ZhangChenYe1996} [10.1121/1.415427]
\newblock D. Zhang, X. Gong and S. Ye,
\newblock {Acoustic nonlinearity parameter tomography for biological specimens via measurements of the second harmonics},
\newblock \emph{J. Acoustical Soc. Amer.}, \textbf{99} (1996), 2397--2402.



\end{thebibliography}
\end{document}